\documentclass[reqno, 11pt]{amsart}

\usepackage{fullpage} 
\usepackage{srcltx} \usepackage{pdfsync}
\usepackage{calc,amsfonts,amsthm,amscd,epsfig,psfrag,amsmath,amssymb,enumerate,graphicx}
\newcommand{\ie}{\hbox{\it i.e.\ }}

\usepackage[showlabels,sections,floats,textmath,displaymath]{}

\reversemarginpar
\newlength\fullwidth
\numberwithin{equation}{section}

\DeclareMathSymbol{\leqslant}{\mathalpha}{AMSa}{"36} % nicer `smaller or equal' 
\DeclareMathSymbol{\geqslant}{\mathalpha}{AMSa}{"3E} % nicer `larger or equal' 
\DeclareMathSymbol{\eset}{\mathalpha}{AMSb}{"3F}     % nicer `emptyset' 
\renewcommand{\leq}{\;\leqslant\;}                   % redef. of < or = 
\renewcommand{\geq}{\;\geqslant\;}                   % redef. of > or = 
\newcommand{\mintwo}[2]{\min_{\substack{#1 \\ #2}}} % min with 2 lines 
\def\1{\ifmmode {1\hskip -3pt \rm{I}} \else {\hbox {$1\hskip -3pt \rm{I}$}}\fi}

\newcommand{\var}{\operatorname{Var}}

\newcommand{\si}{\sigma } 
\newcommand{\tmix}{T_{\rm mix}} 
\newcommand{\trel}{T_{\rm rel}} 
\newcommand{\sign}{{\rm sign}}

\renewcommand{\l}{\lambda}
\renewcommand{\a}{\alpha}
\renewcommand{\d}{\delta}
\renewcommand{\t}{\tau}

\newcommand{\g}{\gamma}
\newcommand{\G}{\Gamma}
\newcommand{\z}{\zeta}
\newcommand{\e}{\varepsilon}
\renewcommand{\r}{\rho}

\renewcommand{\o}{\omega}
\renewcommand{\O}{\Omega}

\newcommand{\gap}{{\rm gap}}
\newcommand{\tc}{\thinspace |\thinspace}

\newcommand{\integ}[1]{\lfloor #1\rfloor}
\newtheorem{theorem}{Theorem}[section] 
\newtheorem{lemma}[theorem]{Lemma} 
\newtheorem{proposition}[theorem]{Proposition} 
\newtheorem{corollary}[theorem]{Corollary} 
\newtheorem{remark}[theorem]{Remark}

\newcommand{\cD}{\ensuremath{\mathcal D}} 
\newcommand{\cE}{\ensuremath{\mathcal E}} 
 
\newcommand{\cG}{\ensuremath{\mathcal G}}

\newcommand{\cL}{\ensuremath{\mathcal L}}

\newcommand{\cS}{\ensuremath{\mathcal S}}

\newcommand{\bbE}{{\ensuremath{\mathbb E}} }

\newcommand{\bbL}{{\ensuremath{\mathbb L}} } 
 
\newcommand{\bbN}{{\ensuremath{\mathbb N}} } 
 
\newcommand{\bbP}{{\ensuremath{\mathbb P}} } 
 
\newcommand{\bbR}{{\ensuremath{\mathbb R}} }

\newcommand{\bbZ}{{\ensuremath{\mathbb Z}} }

\newcommand{\wb}{\overline } 
\newcommand{\gep}{\varepsilon}

\newcommand{\Tm}{T_{\rm mix}}
\newcommand{\Tr}{T_{\rm rel}}

\renewcommand{\ll}{\left\langle}
\newcommand{\rr}{\right\rangle}
\newcommand{\gO}{\Omega}
\newcommand{\ind}{{\bf 1}}
\newcommand{\gd}{\delta}
\newcommand{\dd}{\mathrm{d}}
\newcommand{\gga}{\gamma}

\newcommand{\D}{\Delta}

\newcommand{\gG}{\Gamma}

\begin{document}

\title[Dynamics of a polymer with repulsion]{Polymer dynamics in the
  depinned phase: metastability with logarithmic barriers}

\author[P. Caputo]{Pietro Caputo}
\address{P. Caputo, Dipartimento di Matematica,
  Universit\`a Roma Tre, Largo S.\ Murialdo 1, 00146 Roma, Italia.
  e--mail: {\tt caputo@mat.uniroma3.it}}

\author[H. Lacoin]{Hubert Lacoin}
\address{H. Lacoin, Dipartimento di Matematica, Universit\`a Roma Tre,
Largo S.\ Murialdo 1, 00146 Roma, Italia. e--mail: {\tt lacoin@math.jussieu.fr}}

\author[F. Martinelli]{Fabio Martinelli} 
\address{F. Martinelli, Dipartimento di Matematica,
  Universit\`a Roma Tre, Largo S.\ Murialdo 1, 00146 Roma, Italia.
  e--mail: {\tt martin@mat.uniroma3.it}} 
\author[F. Simenhaus]{Fran\c{c}ois Simenhaus}
\address{F. Simenhaus, Dipartimento di Matematica,
  Universit\`a Roma Tre, Largo S.\ Murialdo 1, 00146 Roma, Italia.
  e--mail: {\tt simenhaus@mat.uniroma3.it}}
\author[F.L. Toninelli]{Fabio Lucio Toninelli} 
\address{F.L. Toninelli, CNRS and ENS Lyon, Laboratoire de Physique\\
  46 All\'ee d'Italie, 69364 Lyon, France.  e--mail: {\tt
    fabio-lucio.toninelli@ens-lyon.fr}}

\begin{abstract} We consider the stochastic evolution of a $(1+1)$-dimensional polymer in the depinned regime. At equilibrium the
  system exhibits a double well structure: the polymer lies
  (essentially) either above or below the repulsive line. As a
  consequence  one expects a metastable behavior with
  rare jumps between the two phases combined with a fast
  thermalization inside each phase. However the
  energy barrier between these two phases is only logarithmic in the
  system size $L$ and therefore the two relevant time scales are only polynomial
  in $L$ with no clear-cut separation between them. The whole evolution is governed by a subtle competition
  between the diffusive behavior inside one phase and the jumps across
  the energy barriers. In particular the usual scenario in which the
  tunneling time coincides with the exponential of the energy
  barrier breaks down. Our main results are:
  (i) a proof that the mixing time of the system lies between
  $L^{\frac 52}$ and $L^{\frac 52
    +2}$;
(ii) the identification of two regions associated 
with the positive and  negative phase of the polymer  together with
the proof of the
asymptotic exponentiality  of the tunneling time between them with
rate equal to a half of the spectral gap.
\end{abstract}

\thanks{This work was
  supported by the European Research Council through the ``Advanced
  Grant'' PTRELSS 228032}

\maketitle

\thispagestyle{empty}

\section{Introduction, model and results}\label{introd}
Random polymers are commonly used in statistical mechanics to model a variety of 
interesting physical phenomena.
A rich class of models with a non trivial behavior is
obtained by considering a simple random walk path interacting with a
defect line in the thermodynamic limit when the length of the path
tends to infinity.  
The equilibrium of these so-called polymer pinning models has
been studied in depth in the mathematical literature, and the
associated localization/delocalization phase transition is, nowadays,
a well understood phenomenon, even in the presence of non homogeneous
interactions; see \cite{cf:GB} for a recent survey. 

Markovian stochastic
dynamics of random pinned polymers, on the other hand, 
have received much less attention from a mathematical point of view. 
Besides their importance in  bio-physical applications (see 
e.g.\ \cite{cf:dna2,cf:dna1} and references therein), the stochastic evolution of polymer models
poses new challenging probabilistic problems from many points of view and the
connection between the equilibrium and dynamical properties of the
model is
still largely unexplored. 
In particular, we feel that 
the problem of how the polymer relaxes to the stationary distribution
(time scales, overcoming of energy barriers, metastability, patterns
leading to equilibrium) still lacks a satisfactory solution
even in the simplest homogeneous models; see \cite{CMT1} for some
initial results in this direction.

In this paper we consider the dynamics 
of a homogeneous polymer model interacting with a repulsive defect
line with two main motivations in mind:
\begin{enumerate}[(i)]
\item the repulsive regime is characterized 
by a relaxation to equilibrium occurring on a time scale
certainly much larger \cite{CMT1}  than the usual diffusive one which
is typical of the e.g.
neutral case\footnote{In the neutral case (absence of
  an interaction
  between the polymer and the line) our process is nothing but
  the usual (finite) symmetric simple exclusion model.} \cite{cf:Wilson}. The new
scale is clearly the result of a complicate competition in the polymer
evolution between diffusive
behavior and jumps against energy barriers. 
\item The whole relaxation mechanism should show certain typical
features of metastable evolution but in a very \emph{a-typical}
context\footnote{Over the years there have many different formulations
  of ``metastability''. We refer in particular to \cite{metabook,Sco,Bo} and
  to the recent contributions \cite{BL1,BL2} where, as in our case,
  energy barriers are only logarithmic in the characteristic size of
  the system. We feel however that our situation does not fit
  completely in any
  of the quoted contexts.} in
which the relevant relaxation 
time scales are only polynomial in the size $L$ of the system (\ie the
energy barriers are only logarithmic in $L$), with little
separation between the mixing time inside one phase
and the global mixing time\footnote{We cannot necessarily be
  very precise here about the exact meaning of these mixing times but their
  definition and their role will appear clearly later on.}. A
signature of this fact can be found in the anomalous growth with $L$
of the
global mixing time, a growth which is much more rapid than the naive guess
based on the usual rule $T_{\rm mix} \approx \exp(\D E)$, with $\D E$
the so called activation energy. In order
to appreciate the novelty of such a situation it is useful to compare
it to another well known case, namely the Glauber dynamics for the low
temperature Ising model in a square box with free boundary \cite{MM},
for which a very precise analysis of the metastable behavior was
possible exactly because of a sharp separation, at an exponential level,
between the two time scales.     
\end{enumerate}

\subsection{Dynamics of the polymer pinning model}
%We consider the following model for the evolution of 
%a $1+1$ dimensional interface interacting with a repulsive substrate.
Let $\O=\O_{2L}$ denote the set of all lattice 
paths (polymers) starting at $0$ and ending at $0$ after $2L$ steps, $L\in\bbN$:
$$
\O =
\{\eta\in\bbZ^{2L+1}:\;
\eta_{-L}=\eta_L=0\,,\;\eta_{x+1}=\eta_x\pm 1\,,\,\;x=-L,\dots,L-1\}\,.$$
The stochastic dynamics is defined by the natural spin-flip continuous time Markov chain with state space $\O$. Namely, sites $x=-L+1,\dots,L-1$ are equipped with independent rate $1$ Poisson clocks. When site $x$ rings, 
the height $\eta_x$ of the polymer at $x$ is updated according to the rules: if $\eta_{x-1}=\eta_{x+1}\pm 2$ then do nothing; if $\eta_{x-1}=\eta_{x+1}=h$, and $|h|\neq 1$, then set $\eta_x = h\pm 1$ with equal probabilities;  if 
$\eta_{x-1}=\eta_{x+1}=+1$, then set $\eta_x = 0$ with probability $\frac{\l}{\l+1}$ and $\eta_x = 2$ with probability $\frac{1}{\l+1}$; similarly, if $\eta_{x-1}=\eta_{x+1}=-1$, then set $\eta_x = 0$ with probability $\frac{\l}{\l+1}$ and $\eta_x = -2$ with probability $\frac{1}{\l+1}$.
Here $\l>0$ is a parameter describing the strength of the attraction ($\l>1$) or repulsion ($\l<1$) 
between the polymer and the line $h\equiv 0$. The infinitesimal generator of the Markov chain is given by 
\begin{equation}\label{genero}
\cL f(\eta) = \sum_{x=-L+1}^{L-1} r_{x,+}(\eta)\left[f(\eta^{x,+})-f(\eta)\right] 
+ \sum_{x=-L+1}^{L-1} r_{x,-}(\eta)\left[f(\eta^{x,-})-f(\eta)\right]\,,
\end{equation}
where: $f$ is a function $\O\mapsto\bbR$; $\eta^{x,\pm}$ denotes the configuration which coincides with $\eta$ at every site $y\neq x$ and equals $\eta_{x}\pm 2$ at site $x$; the rates $r_{x,\pm}$ are zero unless $\eta_{x-1}=\eta_{x+1}$, and if $\eta_{x-1}=\eta_{x+1}=h$ they satisfy  
$r_{x,\pm} = \frac12$ for $h\neq \pm 1$, and $r_{x,\mp} = \frac{\l}{\l+1}=1-r_{x,\pm}$,  for $h=\pm1$.

The process defined above is the {\em heat bath} dynamics for the homogeneous {\em polymer pinning model}, 
with equilibrium measure $\pi=\pi_{2L}^\l$ on $\O$  defined by
\begin{equation}\label{equ}
 \pi^{\l}_{2L}(\eta)
 = \frac{\l^{N(\eta)}}{Z_{2L}^\l}\,,
\end{equation}
where $N(\eta) = \#
\{x\in\{-L+1,\dots,L-1\}\,:\;\eta_x=0\}$ denotes the number of zeros in the path $\eta\in\O$ and $Z_{2L}^\l = \sum_{\eta'\in \O}
 \l^{N(\eta')}$. For every $\l>0$, $L\in\bbN$, $\pi=\pi_{2L}^\l$  is the unique reversible invariant measure for the Markov chain.

\subsection{Relaxation to equilibrium}
 The equilibrium properties of the polymer pinning model have been studied in detail, cf.\ e.g.\ \cite{Iso_Yos} or \cite[Section 2]{cf:GB} for an extensive review. 
 In particular,  it is well known that, under the measure $\pi_{2L}^\l$,  for $\l>1$ the path is strongly localized with a non vanishing density of zeros, while for $\l<1$ the path is delocalized with $\sqrt L$ height fluctuations and with 
 the number of zeros stochastically dominated by a geometric random variable with $\l$-dependent parameter. 
The dynamical counterpart of this localization/delocalization transition has not been fully understood yet. 
Some progress in this direction has been reported in \cite{CMT1}, where various bounds on the spectral gap %$\gap(L,\l)$ 
and mixing time of the Markov chain were obtained, together with estimates on the decay of time correlations. % in both the localized and the delocalized regime.

We recall that the spectral gap is the smallest nonzero eigenvalue of $-\cL$, and one is often interested in the {\em relaxation time} $\trel=1/\gap$ which governs decay to equilibrium in $\bbL^2(\pi)$, 
while the {\em mixing time} $\tmix(\d)$, for $\d\in(0,1)$, is the smallest time $t$ such that %the total variation distance
\begin{equation}\label{tvar}
\max_\eta\|P_t(\eta,\cdot)-\pi\|\leq \d\,,
\end{equation}
where $P_t(\eta,\cdot)$ denotes the distribution of the Markov chain at time $t$ with initial state $\eta\in\O$,  
%law on $\O$ of the Markov chain at time $t$ 
and $\|\mu-\nu\| $ is the usual total variation distance between two probability measures. 
%We shall write $\trel(L,\l)$ and $\tmix(L,\l,\d)$ when we need to emphasize the dependence on the parameters $L,\l$ of the model. 
When $\d=\frac1{2e}$ we often write simply $\tmix$
instead of $\tmix(\d)$.  With these conventions one has $\trel\leq \tmix$ always, and the inequality is strict in general.

A dynamical phase transition occurs when we move from the localized regime $\l>1$ to the delocalized regime $\l<1$. %Recall that the free case $\l=1$ is characterized by diffusive scaling, with 
%$\gap(L,1) =  1- \cos(\frac\pi{2L})$. 
It was shown in \cite{CMT1}, see Theorem 3.4 and Theorem 3.5
there, that  for $\l\geq 1$ one has $ \trel= O(L^2)$  and $\tmix= O(L^2\log L)$ for all $L$, while 
for $\l<1$ one has 
\begin{equation}\label{5/2}
\trel\geq L^{\frac52 - \e}\,,
\end{equation}
for all $\e>0$, provided $L$ is large enough.
% larger than some constant $L_0(\e,\l)$. 

We refer to \cite{CMT1} for results and conjectures concerning the localized regime $\l>1$. Here, we consider the %problem of relaxation to equilibrium in the 
delocalized regime, i.e.\ in the repulsive case $\l<1$. The first question we address concerns an upper bound on the relaxation time 
$\trel$ and the mixing time $\tmix$. It is worth noting that even a crude    
polynomial bound is non-trivial.  
We refer to \cite{MR,CMT1}
for polynomial bounds for the model with a horizontal wall at level zero, i.e.\ when lattice paths are constrained 
to be non negative. On the other hand, without the wall constraint,
the equilibrium measure $\pi$ is known to be concentrated, as
$L\to\infty$, on configurations in which the density of monomers in the
upper (lower) half plane is approximately either one  or zero. However a
mathematical working definition of a candidate for the plus or minus
phase  for the polymer is not so obvious and we have been forced
to introduce a mesoscopic parameter $\ell$ (\ie $L\gg \ell \gg 1$) and define  $\O^\pm$ by
\begin{equation}\label{oplusminus}
 \Omega^+=\{\eta\in\O\,:\;\eta_x> 0\,,\;-L+\ell < x< L-\ell\}\,,\qquad \O^-=-\O^+\,,
\end{equation}
%The region where 
% except for possibly few $x\in V_L$ such that 
%$|x|\approx L$
where, for any set $A$ of polymer configurations $-A = \{\eta\in
\O\,:\;-\eta\in A\}$). 

The presence of the two phases associated to $\O^\pm$ dramatically changes the relaxation scenario, with a bottleneck at the set $\O\setminus( \O^+\cup\O^-)$. As explained in \cite[Section 6]{CMT1}, one may suspect that $\trel\sim L^{\frac52}$ is the correct asymptotic behavior in the delocalized regime.  Let us briefly recall the heuristic reasoning behind this prediction.

The time to reach equilibrium can be roughly thought of as the time needed to switch from, say, $\O^-$ to $\O^+$. 
A point $x$ such that $\eta_x=0$ and $\eta_{x-1}\neq \eta_{x+1}$ is called a {\em crossing} of the polymer. Note that any zero (and therefore any crossing) $x$ must belong to the set $E_L$ of points in the segment 
$\{-L+2,\dots,L-2\}$ which have the same parity as $L$. 
Since there are typically very few zeros at equilibrium, 
%and configurations $\eta$ with $N(\eta)=0$ have a positive probability uniformly in $L$. 
one may consider the 
extreme case where at most one crossing $\xi$ is allowed at all times. In this case, the 
time evolution of $\xi$  
should be essentially described by a suitable birth and death process
or random walk
on $E_L$; see Figure \ref{fig:bulle}.

\begin{figure}[htp]
\begin{center}
\includegraphics[width=0.99\textwidth]{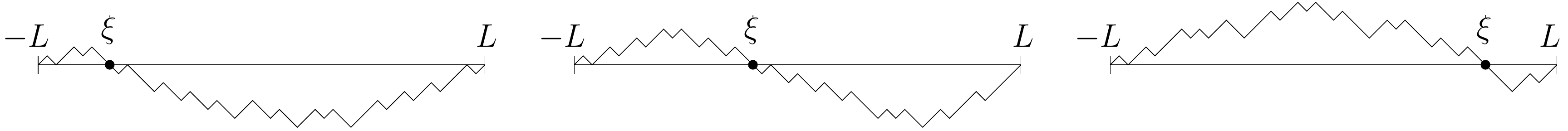}
\end{center}
\caption{From left to right a snapshot sequence of the motion of a single crossing $\xi$ which allows the system to switch from a mostly negative to a mostly positive configuration. 
%The point $\xi$ performs a random walk with drift pushing towards the boundaries, in such a way that the measure $\r$ in \eqref{rev32} is reversible. 
}
\label{fig:bulle} 
\end{figure}
From equilibrium considerations, one knows that this random walk should have 
reversible invariant measure $\rho$ roughly proportional to  
\begin{equation}\label{rev32}
\r(x) \propto L^{3/2}(L+x)^{-3/2}(L-x)^{-3/2}\,,\qquad x\in E_L\,,
\end{equation} 
and that its relaxation time can be 
bounded from above and below by constant multiples of $L^{5/2}$; 
see Lemma \ref{equo} and Lemma \ref{52} below for more details. 
Notice that, although the measure $\rho$ gives uniformly (in $L$)
positive mass to the two \emph{attractors} $x^{\pm}=\pm L$, the drift which
pushes the random walk away from the saddle $x=0$ is proportional to
the inverse of the distance from the attractors. In particular the
naive guess for the mixing time $T_{\rm mix}\approx 1/\rho(0)=
O(L^{3/2})$ is wrong.

This heuristics is turned into the rigorous bound (\ref{5/2}) by using a suitable test function 
in the variational principle that characterizes the spectral gap; see \cite[Section 6]{CMT1}. 
However, it seems very hard to give a rigorous upper bound on $\trel$ of the same order of magnitude. We obtain a bound that can be off by at most two powers of $L$.
\begin{theorem}
 \label{th1}
For any $\lambda<1$,
\begin{equation*}
 \limsup_{L\to\infty} \frac{\log \tmix}{\log L}\leq \frac{5}{2}+2\,.
\end{equation*}
\end{theorem}

The main tool for the proof of Theorem \ref{th1} is 
the analysis of an effective dynamics for the crossings of the polymer. 
% for the crossings.
To describe this, we introduce the variable $\si\in\{-1,+1\}^{O_L}$, where 
$O_L=\{-L,\dots,L\}\setminus E_L$ denotes the sites  %\in\{-L,\dots,L\}$ 
with the same parity of $L+1$.  
If $\eta$ is a configuration of the polymer, then $\eta_x\neq 0$ at any  $x\in O_L$, and we define 
$\si(\eta)$ by 
$\si_x=\sign(\eta_x)$. The projection of $\pi$ on $\cS=\{-1,+1\}^{O_L}$ is then
\begin{equation}\label{nus}
\nu(\si) = \sum_{\eta\,:\;\eta\sim\si}\pi(\eta)\,,
\end{equation}
where the sum is over all configurations $\eta$ compatible with the signs $\si$.  The field $\nu$ has non trivial long range correlations. 
%
%crossing configurations and the associated process.
%We say that there is a crossing at $x\in\{-L,\dots,L\}$ if either $(\eta_{x-1},\eta_x,\eta_{x+1})=(-1,0,+1)$
%or $(\eta_{x-1},\eta_x,\eta_{x+1})=(+1,0,-1)$. Call $V_L$ the even sites $x\in\{-L,\dots,L\}$. 
Consider the heat bath dynamics for the variables $\si$: sites $x\in O_L$ are equipped with
independent rate $1$ Poisson clocks; when site $x$ rings we replace $\si_x$ by $\si'_x$ where the new sign $\si'_x$ is distributed according to the conditional probability $\nu(\cdot\tc\si_y\,,\;y\neq x)$, i.e.\ the probability \eqref{nus} conditioned on the value of $\si_y\,,\;y\neq x$. 
Denote by $\trel^\cS$ the corresponding relaxation time.  For this process, the exponent $5/2$ can be shown to be optimal. 
% of the heat bath dynamics for the variables $\si$. 
%As a byproduct of our analysis we shall obtain:
%
%The heuristics that we briefly described above can be given more substance by considering the full
%configuration of the zeros of the polymer. Actually, rather than the zeros, we consider the crossings of the zero line, i.e.\ the points in the line $\{-L,\dots,L\}$ where 
%the polymer crosses. 
\begin{theorem}
 \label{th2}
For any $\lambda<1$,
\begin{equation*}
 \lim_{L\to\infty} \frac{\log \trel^\cS}{\log L}= \frac{5}{2}\,.
\end{equation*}
\end{theorem}
The proof of Theorem \ref{th1}
and Theorem \ref{th2} combines several different tools which play a prominent role in the analysis of convergence to equilibrium of 
Markov chains: decomposition methods, spectral gap analysis, comparison inequalities, and coupling estimates. 
An outline of the main steps of the proof is given at the beginning of Section \ref{cinquemezzi}.
%
%is based on a general decomposition technique \cite{Jerrum}, combined with various comparison inequalities and coupling arguments. We refer to the beginning of Section \ref{cinquemezzi} for a more detailed description of the arguments involved.

\subsection{Metastability}
Recall the definition \eqref{oplusminus} of the two sets $\O^\pm$, and define the associated phases as the 
%The picture that emerges from the previous discussion is roughly speaking the following: the system has two phases, corresponding to the two sets $\O^\pm$,
%where the variables $\si_x$ are essentially all $+1$ or all $-1$.  
%and one where the variables $\si_x$ are essentially all $-1$. To describe these phases, we introduce the set 
%%%%%%%%%%
%Then, one defines the two phases as the 
restricted equilibrium measures $\pi^\pm:=\pi(\cdot\tc\O^\pm)$, so that 
(cf.\ Section \ref{setup})
% we show that, as $L\to\infty$
\begin{equation}
\label{mixture}
\Big\|\pi-\frac12(\pi^++\pi^-)\Big\| = o(1)\,.
\end{equation} 
%As the size of the system $L$ diverges
In the thermodynamic limit, we expect relaxation to equilibrium within each phase 
to occur on time scales $\trel^\pm$ such that  $\trel\gg \trel^\pm$, 
while on a time scale proportional to $\trel$ one should 
see the system jump from one phase to the other according to i.i.d.\ exponentially distributed times. A strong indication of this metastable behavior comes from the following theorem. 
%Below we use $\O^+$ (resp.\ $\O^-$) to denote the set of $\eta\in\O$ such that $\eta_x>0$ (resp.\ $\eta_x<0$) for all $x\in\{-L+\ell(L),L-\ell(L)\}$ where $\ell(L)^{-1}=o(1)$ and $\ell(L)/L = o(1)$. 
Below, we use $\eta(t)$ to denote the state at time $t$ 
of the Markov chain with generator \eqref{genero}. The notation $o(1)$ refers to asymptotics as $L\to\infty$.  
\begin{theorem}\label{th3}
 There exists a set $S^{+}\subset \O^{+}$ such that $\pi(S^{+})=1/2+o(1)$, and that uniformly in $\eta \in S^+$ and uniformly in $t>0$:
\begin{equation*}
 \bbP^{\eta}\left(\tau^{-}>t\right)=e^{-t/(2\trel)}\,+\,o(1)\,,
% \quad\exp(-t/2\trel)+o(1).
\end{equation*}
%(resp.\ the same with $\tau^+$) 
where
\begin{equation*}
 \tau^-=\inf\{t \geq 0:\; \eta(t)\in S^-\}\,,\qquad S^- = -S^+\,.
\end{equation*}
\end{theorem}
\begin{remark}
\rm From the proof of Theorem \ref{th3} it will be clear that the set
$S^+$ is {\sl increasing} w.r.t. the natural partial order among
polymer configurations defined in Section \ref{mono}, so that
in particular the maximal configuration (in the sequel denoted by
$\wedge$) is in $S^+$.
\end{remark}

Here $\bbP^{\eta}$ stands for the law of the process with initial state $\eta\in\O$. By symmetry, Theorem \ref{th3} also implies that 
uniformly in $\eta \in S^-=-S^+$ and uniformly in $t>0$, 
\begin{equation*}
 \bbP^{\eta}\left(\tau^{+}>t\right)=e^{-t/(2\trel)}\,+\,o(1)\,,
% \quad\exp(-t/2\trel)+o(1).
\end{equation*}
%(resp.\ the same with $\tau^+$) 
where $\tau^+=\inf\{t \geq 0:\; \eta(t)\in S^+\}$.

If we define the renormalized process 
$$\o_s = \ind\left(\eta(s\trel) \in\O^+\right)-\ind\left(\eta(s\trel) \in\O^-\right)\,,%\qquad s\geq 0\,.
$$
we expect that, starting from \emph{any} configuration in $\Omega^\pm$, $\{\o_s,\;s\geq0\}$ converges to the simple two-state Markov chain with 
switching rate (from $\pm 1$ to $\mp1$) equal to $1/2$, whose spectral gap equals $1$.    
Such a strong uniform result  seems very hard to obtain for our model.
The difficulty is that, in contrast with familiar metastability results \cite{metabook}, here there is no dramatic separation of time scales: while \eqref{5/2} and Proposition \ref{muretto} below imply $\trel\gg\trel^+$, the ratio $\trel/\trel^+$ is only polynomially large in $L$. 
However, we do have a detailed description of the
renormalized process when the initial condition is the maximal configuration. Namely, define the maximal element of $\O$ as $\eta_{\max}=\wedge$, i.e.\ 
$\wedge_x = x+L$ for $x\leq 0$ and $\wedge_x = L-x$ for $x\geq 0$, and let $
\tmix^{\wedge}(\e)$ denote the first time $t$ such that $\|P_t(\wedge,\cdot)-\pi\|\leq \e$.
\begin{theorem}\label{th4} 
For any $\d>0$, uniformly in $t\geq L^{2+\d}$
\begin{equation} \label{PTT}
\Big\| P_t(\wedge,\cdot)-\Big[\,\,\frac{1+e^{-t/\Tr}}{2}\,\pi^+ + \frac{1-e^{-t/\Tr}}{2} \,\pi^-\Big]\Big\|=o(1)
\end{equation}
and uniformly in $t\ge 0$
\begin{equation} \label{PTT2}
\Big\| \nu^{\pi^+}_t-\Big[\,\,\frac{1+e^{-t/\Tr}}{2}\,\pi^+ + \frac{1-e^{-t/\Tr}}{2} \,\pi^-\Big]\Big\|=o(1).
\end{equation}
Moreover, 
for any $\gep\in(0,1/2)$ one has
\begin{equation} \label{tmixhat}
\Tm^{\wedge}(\gep)= \Tr\log\left(\frac{1}{2\gep}\right)(1+o(1)).
\end{equation}
 \end{theorem}

\begin{remark}
\rm
Theorem \ref{th4} shows in particular that, when the dynamics is
started from either $\pi^+$ or $\wedge$, there is no {\em cut-off}
phenomenon \cite{LPW}, i.e. the variation distance from equilibrium
does not fall abruptly to zero, but rather does so smoothly (on the
timescale $T_{rel}$). That is another signature of the metastable
behavior of our system and it is in contrast with what one expects for
the neutral or attractive case $\lambda \ge 1$. 
 \end{remark}

One of the key features of metastability is that,
once the system decides to jump from e.g.\
$S^+$ to $S^-$, then it does so very quickly on the time scale of the mixing time.
 We
verify that indeed this is the case for most starting configurations
inside $S^+\cup S^-$.

Let $ \mathcal T$ denote the random time
spent outside $S^+\cup S^-$ up to the hitting time of $S^-$:
\begin{eqnarray}
  \label{eq:1}
  \mathcal T:=\int_0^{\tau^-} \ind_{\{\eta(s)\in (S^+\cup S^-)^c\}} \, \dd s.
\end{eqnarray}
From Theorems \ref{th3} and \ref{th4} one easily deduces that, for
most initial conditions in $S^+$, 
$\tau^-\gg \mathcal T$:
\begin{corollary}\label{corol}
  There exists a subset $\tilde S^+$ of the set $S^+$ of Theorem
  \ref{th3} satisfying $\pi(\tilde S^+)=1/2+o(1)$ such that, uniformly on $\eta\in \tilde S^+$,
 $\mathcal T=o(\tau^{-})$ in probability, i.e.\ there exists a sequence $\delta_L$ tending to zero as $L\to\infty$ 
such that for every $\eta\in
\tilde S^+$,
\begin{equation}
  \label{eq:2}
  \bbP^{\eta} \left[\mathcal T\ge \gd_L \tau^-\right]\le \gd_L.
\end{equation}

\end{corollary}
Along the same lines of the proof of the Corollary, one can
establish the weak convergence of the renormalized process $\omega_s$
to the two-state Markov chain, provided that the initial configuration
is inside a suitable subset of ${\tilde S}^+\cup {\tilde S}^-$ with
almost full measure. We decided to omit details for shortness.

\subsection{Organization of the paper}
The rest of the paper consists of three sections. Section \ref{setup} starts with standard material 
%on continuos time Markov chains 
and then proceeds with the introduction of some essential tools to be
used in the proof of the main results, including general results for
monotone systems that can be of independent interest. This section
contains also some new results concerning the relaxation within one
phase and the properties of the principal eigenfunction of the
generator. 
 The metastability results are discussed in Section \ref{metastab}. Here, we start with  the proof of Theorem~\ref{th4}. In later subsections we develop 
%the main tools for 
the construction 
needed for the proof of Theorem \ref{th3}. 
Finally, Section \ref{cinquemezzi} proves Theorem \ref{th1} and Theorem \ref{th2}. This section is  broken into several subsections corresponding to the various steps of the proof. A high level description of the arguments involved is given at the beginning of the section.
Finally,

\bigskip

\noindent
{\bf Notational conventions}. Whenever we write $o(L^{p})$ or $O(L^p)$ for some $p\in\bbR$ it is understood that this refers to the thermodynamic limit $L\to\infty$. Also, we use the notation $f(L)=\O(L^{p})$ when there exists a constant $c>0$ such that $f(L)\geq c\,L^p$ for all sufficiently large $L$. 
For positive functions $f,g$, we use the notation $f(L)\gg g(L)$ whenever 
$\liminf_{L\to\infty} f(L)/g(L)=+\infty$, and $f(L)\sim g(L)$ when $\lim_{L\to\infty} f(L)/g(L)=1$. 
Also, we write $f\asymp g$ if there exists 
some constant $c>0$ such that $c^{-1}g\leq f\leq c\,g$. 

\section{Some tools}\label{setup}
We begin with some generalities about reversible Markov chains. Then, we recall the definition of the polymer dynamics and derive some consequences of monotonicity. Next, we give some estimates
on convergence to equilibrium in the ``plus'' phase. Finally, we characterize in detail 
an eigenfunction of $\cL$
with eigenvalue $-\gap$.

\subsection{Preliminaries}\label{preli}
We will consider reversible continuous time Markov chains with finite state space $X$, defined by the infinitesimal generator
$\cL$ acting on functions $f:X\mapsto\bbR$, 
\begin{equation}\label{genergen}
[\cL f](x) = \sum_{y\in X}c(x,y)[f(y)-f(x)]\,,
\end{equation}
where $c(\cdot,\cdot)$ is a bounded non negative function on $X\times X$ satisfying 
$\pi(x)c(x,y) = \pi(y)c(y,x)$, for a probability measure $\pi$ on $X$. 
In the applications below, the rates $c(x,y)$ will always be such that the Markov chain is irreducible and the reversible invariant measure $\pi$ is positive on $X$. We refer e.g.\ to \cite{AF,LPW} for more details on reversible Markov chains. 

Let $\nu_t^x=\bbP(v_t^x\in\cdot)$ denote the law of the state $v_t^x$ of the Markov chain at time $t$ with initial condition $x\in X$. We shall investigate the rate of convergence of $\nu_t^x$ to $\pi$. If the initial condition $x$ is distributed according to a probability  $\mu$ on $X$, we write $\nu_t^\mu=\sum_{x\in X}\mu(x)\nu_t^{x}$ for the distribution at time $t$. As usual, one can associate a semi-group $\{P_t,\,t\geq 0\}$  to the generator $\cL$ in such a way that 
$[P_tf](x)=[e^{t\cL}f](x)=\sum_{y\in X}\nu_t^x(y)f(y)$. We also use the notation 
$P_t(x,y)=\nu_t^x(y)$, and $\nu_t^\mu = \mu P_t$.

The mixing time of the Markov chain is defined by 
\begin{equation}\label{tmix}
\tmix(\e)=\inf\Big\{t>0\,,\;\,\max_{x\in X}\|\nu_t^{x} - \pi\|\leq \e\Big\}\,,
\end{equation}
where $$\| \mu-\nu\|= \frac12\sum_{x}|\mu(x)-\nu(x)|\,$$ is the total variation distance. We shall use the convention that $\tmix$ stands for $\tmix(\frac1{2e})$. It is well known that with this notation one has 
\begin{equation}\label{tmix2}
\|\nu_t^{x} - \pi\|\leq e^{-\integ{t/\tmix}}\,,
\end{equation}
for all $t\geq 0$, where $\integ{a}$ denotes the integer part of $a\geq 0$. The spectral gap and the relaxation time of the process are defined by
\begin{equation}\label{trel}
\gap=\min_{f:X\mapsto\bbR}\,\frac{\cE(f,f)}{\var_\pi(f)}\,,
\quad\;\; \trel = \frac1\gap\,,
\end{equation}
where for $f:X\mapsto\bbR$,
\begin{equation}\label{dirich}
\cE(f,f) = \sum_{x\in X}\pi(x) f(x) [-\cL f](x) = \frac12\sum_{x,y\in X}\pi(x)c(x,y)[f(y)-f(x)]^2\,
\end{equation}
is the quadratic form of the generator, a.k.a.\ the Dirichlet form, while $\var_\pi(f)$ stands for the 
variance $\pi(f^2)-\pi(f)^2$. Thus, $\gap$ is the lowest non zero eigenvalue of $-\cL$. 
The following bound relating total variation distance and relaxation time is an immediate consequence of reversibility and Schwarz' inequality:
\begin{equation}\label{trelvar}
\|\nu_t^{\mu} - \pi\|\leq \frac12 \,e^{-t/\trel}\,\sqrt{\var(f)}\,,
\end{equation}
where $f(\sigma)=\mu(\sigma)/\pi(\sigma)$ and $\mu$ is a probability on $X$. 
Another standard relation between total variation and relaxation time is the identity
\begin{equation}\label{clasgap}
 \gap=-\lim_{t\to\infty} \frac{1}{t}\log \max_{x,y} \| \nu_t^{x} -\nu_t^{y}\|.
\end{equation}
Combining \eqref{tmix2}, \eqref{clasgap}
and \eqref{trelvar},  one can obtain the following well known relations: %bounds relating $\trel$, $\tmix$:
\begin{equation}\label{trelmix}
\trel\leq \tmix \leq \big(1-\log{\pi_*}\big) \trel\,,\quad\;\text{where $\pi_* =\min_{x\in X}\pi(x)$}\,.
\end{equation}
 
%We refer to \cite{LPW} for an introduction to total variation analysis, and to mixing time and spectral gap estimates.  

\subsection{A general decomposition bound on the spectral gap}
We shall need a 
continuous time version of a general decomposition bound obtained by Jerrum et al.\
 \cite{Jerrum}. 
Consider the continuous time reversible Markov chain defined by \eqref{genergen}.
Suppose the space $X$ is partitioned in the disjoint union of subspaces $X_1,\dots,X_m$, for some $m\in\bbN$ and define the generators
$$
[\cL_i f](x) =  \sum_{y\in X}c_i(x,y)[f(y)-f(x)]\,,\quad \;c_i(x,y)=c(x,y)\ind(y\in X_i)\,,\quad x\in X_i\,.
$$
Then $\cL_i$ is the generator of the Markov chain restricted to $X_i$, its reversible invariant measure 
being given by $\pi_i = \pi(\cdot\tc X_i)$. Let $\l_{{\rm min}}$ denote the minimum of the spectral gaps of the Markov chains generated by $\cL_i$, $i=1,\dots,m$. 
Next, let $\wb\cL$ denote the infinitesimal generator defined by
$$
[\wb\cL \varphi](i) = \sum_{j=1}^m\bar c(i,j)[\varphi(j)-\varphi(i)]\,,
$$
for $\varphi\in\bbR^m$, where $$
\bar c(i,j)=\sum_{x\in X_i\,,\; y\in X_j}\pi(x\tc X_i)\,c(x,y)\,.$$  
This defines a continuous time Markov chain on $\{1,\dots,m\}$ with reversible invariant measure
$\bar \pi(i) = \pi(X_i)$. Let $\bar\l$ denote the gap of this chain. 
A straightforward adaptation of~\cite[Theorem 1]{Jerrum} yields the following estimate.  
\begin{proposition}\label{jerr}
Define $\g=\max_{i}\max_{x\in X_i}
\sum_{y\in X\setminus X_i}c(x,y)\,.$ Then, with the notation of \eqref{trel},
\begin{equation}\label{jerr11}
\gap\geq
\min\Big\{\frac{\bar\l}3,\frac{\bar \l\,\l_{{\rm min}}}{\bar\l + 3 \g}\Big\}\,.
\end{equation}
\end{proposition}

\subsection{Killed process and quasi-stationary distribution}\label{killing}
Here we recall some standard facts about killed processes, their generators and quasi-stationary distributions for reversible Markov chains; 
we refer to \cite{AF} for an introduction. 
Given a reversible Markov chain with generator $\cL$ as above, and a subset $\G\subset X$, 
we consider the process killed upon entering $\G$, with sub-probability law defined by
\begin{equation}\label{killedp}
\nu_t^{x,\G}(B) = \bbP^x(v_t\in B\,;\;\t_\G>t)\,,\quad\;x\in \G^c\,,
\end{equation}
where $B\subset X$, $v_t$ denotes the state of the Markov chain with generator $\cL$ at time $t$, $\bbP^x$ denotes the law of the process started at $x$, and
$\t_\G$ denotes the hitting time of the set $\G$. The associated semi-group $P_t^\G$ is given by
\begin{equation}\label{killedp1}
[P_t^\G f](x)=[e^{t\cL^\G}f](x)=\sum_{y\in \G^c}\nu_t^{x,\G}(y)f(y)\,,\quad\;x\in \G^c\,,
\end{equation}
where the killed generator $\cL^\G$ satisfies, for every $x\in \G^c$:
\begin{equation}\label{killedp2}
[\cL^\G f](x) = [\cL (f1_{\G^c})] (x) = [\cL f](x) - \sum_{y\in \G} c(x,y) f(y)\,.
\end{equation}
We assume that $P_t^\G$ is irreducible. 
Then $\cL^\G$ is a negative definite,
self-adjoint operator in $\bbL_2(\pi)$, and its top eigenvalue $-\g_\G$ 
is characterized by
\begin{equation}\label{killedp3}
\g_\G= 
\mintwo{f:X\mapsto\bbR,}{f 1_\G=0} \frac{\ll -\mathcal L^{\G}f,f  \rr_{\pi}}{\pi(f^2)}=\mintwo{f:X\mapsto\bbR,}{f1_\G=0} \frac{\cE(f,f)}{\pi(f^2)}\,,
\end{equation}
where we use $\ll\cdot,\cdot\rr_{\pi}$ for the scalar product in $  \bbL_2(\pi)$, and $\cE(f,f)$ is defined by \eqref{dirich}. 

Let $g_\G$ denote the (unique, positive on $\G^c$) eigenfunction of $\cL^\G$ 
associated to $-\g_\G$. Extending $g_\G$ to all $x\in X$ by setting $g_\G(x)=0$ for $x\in \G$, one defines the {\em quasi-stationary} distribution $\nu^\G$, i.e.\
 the probability  on $X$ given by % with density w.r.t.\ $\pi$ proportional to $g_\G$, i.e.\
\begin{equation}\label{killedp4}
\nu_\G(y) = \frac{\pi(y) g_\G(y)}{\pi(g_\G)}\,,\quad y\in X\,. 
\end{equation}
An equivalent characterization of $\nu_\G$ is as the limit
\begin{equation}\label{killedp5}
\nu_\G(B) = \lim_{t\to\infty} \bbP^x(v_t\in B\tc\t_\G>t)\,,
\end{equation}
where $B\subset X$, and the chosen initial point $x\in\G^c$ is arbitrary. The fundamental property of the 
quasi-stationary distribution is that, starting from $\nu_\G$, the hitting time $\t_\G$ is exponentially distributed with parameter $\g_\G$:
 \begin{equation}\label{killedp6}
\bbP^{\nu_\G}(\t_\G>t) = e^{-\g_\G\,t}\,,%\qquad t\geq 0\,,
\end{equation}
where $\bbP^{\nu_\G}$ stands for the law of the of the process when the initial state is distributed according to $\nu_\G$. Another way of expressing quasi-stationarity is $\nu^\G P^\G_t=e^{-\g_\G\,t}\nu^\G$, for all $t\geq 0$. 

A general property of $\gamma_\Gamma$ (cf. Lemma  \ref{wesh} below) is
that $\gamma_\Gamma\ge \gap \,\pi(\Gamma)$.

\subsection{Polymer model}\label{polymers}
Let $\O=\O_{2L}$ stand for the space of all lattice paths defined in the introduction. 
A partial order in $\O$ is given by 
\begin{equation}\label{order}
\eta\leq \eta' \iff \eta_x\leq \eta'_x\,,\quad x=-L,\dots,L\,.
\end{equation}
Given $\z,\xi\in\O$ such that $\z\leq \xi$ we define the restricted space
$\O^{\z,\xi}$ of all paths $\eta\in\O$ such that 
$\z\leq \eta\leq \xi$. The dynamics is defined by the continuous time Markov chain with state space $\O^{\z,\xi}$, 
with infinitesimal generator $\cL^{\z,\xi}$ given by \eqref{genero} where the rates $r_{x,\pm}(\eta)$
are replaced by 
\begin{equation}\label{generozx}
r_{x,\pm}^{\z,\xi}(\eta) = r_{x,\pm}(\eta) \ind( \eta^{x,\pm}\in\O^{\z,\xi})
\end{equation}
This process is the heat bath dynamics associated to the probability measure $\pi^{\l,\z,\xi}_{2L}$ on $\O^{\z,\xi}$
defined as in \eqref{equ} with the normalization now given by 
\begin{equation}\label{partiz}
Z_{2L}^{\l,\z,\xi} = \sum_{\eta'\in \O^{\z,\xi}}
 \l^{N(\eta')}\,.
\end{equation}
Equivalently, $\pi^{\l,\z,\xi}_{2L}=\pi_{2L}^\l(\cdot\tc\O^{\z,\xi})$. This is referred to as the polymer model with top/bottom constraints ($\z$ is the bottom, $\xi$ is the top). For simplicity, when no confusion arises, we often omit the superscripts $\l,\z,\xi$ and the subscript $L$ from our notation in what follows. We write $v^{\eta}_t$ for the state of the Markov chain at time $t$ when the initial configuration is 
some $\eta$, and let $\nu_t^{\eta}$ denote its distribution. When the initial condition $\eta$ 
is distributed according to a probability measure $\mu$ on $\O$ we write $\nu_t^\mu$ as in Section \ref{preli}.

Note that the generator  $\cL^{\z,\xi}$ can be written in the form \eqref{genergen} by setting 
$c(\eta,\eta') = r^{\z,\xi}_{x,\pm}(\eta)\ind(\eta'=\eta^{x,\pm})$, and $\pi=\pi^{\l,\z,\xi}_{2L}$ is reversible. While this holds for every value $\l>0$ of the parameter describing the strength of the interaction, we will only consider the case $\l<1$ below, which corresponds to a strictly delocalized regime for the polymer. 

The minimal path $\vee$ and maximal path $\wedge$ for the order \eqref{order} are defined by 
$\vee_x = -x-L$ for $x\leq 0$, $\vee_x = -L+x$ for $x\geq 0$, 
and $\wedge=-\vee$.
%$\wedge_x = x+L$ for $x\leq 0$, $\wedge_x = L-x$ for $x\geq 0$, respectively.
Clearly, if $\z=\vee$ and $\xi=\wedge$, then $\O^{\z,\xi}=\O$. This case is referred to as 
%the free polymer model or 
the polymer model with no top/bottom constraint.

The following well known
estimates will be often used in our proofs. We refer e.g.\ to \cite[Section 2]{cf:GB} %or \cite[Section 2]{CMT1} 
for the proof of Lemma \ref{equo} below, as well as for other known properties of the delocalized equilibrium measure. Let $Z_{2L}=Z_{2L}^\l$ denote the partition function \eqref{partiz} with no top/bottom boundaries 
and write $Z_{2L}^+=Z_{2L}^{+,\l}$ for the partition function  \eqref{partiz} with $\xi=\wedge$ and $\z$ given by the minimal 
non negative element of $\O$, i.e.\ $\z_x=0$ if $x\in E_L$ ($x$ has the same parity as $L$) 
and $\z_x=1$ if $x\in O_L$ ($x$ has opposite parity w.r.t. $L$), i.e.\ $Z^+_{2L}$ is 
the partition function of the polymer with a horizontal wall at height zero. Recall that $N=N(\eta)$ stands for the number of zeros
in the path $\eta$ lying strictly between $-L$ and $L$.  Considering reflections of the path between consecutive zeros one obtains 
\begin{equation}\label{Z++}
2Z_{2L}^{+,\l} = Z_{2L}^{\l/2}\,.
\end{equation}

\begin{lemma}\label{equo}
Consider the polymer with no top/bottom constraint with $\l\in(0,1)$.
There exist constants $c_i=c_i(\l)>0$, $i=1,2$ such that
\begin{equation}\label{3/2}
2^{-2L}Z_{2L}^\l\sim c_1\,L^{-3/2}\,,
%c^{-1}\,L^{-3/2}\leq 2^{-2L}Z^+_{2L}\leq 2^{-2L}Z_{2L}\leq c\,L^{-3/2}\,.
\end{equation}
%Moreover, for some other constant $c'=c'(\l)>0$:
and
\begin{equation}\label{expdec}
\pi(N(\eta) > k) \leq c_2\,e^{-k/c_2}\,,\quad \forall\ k\in\bbN\,.
\end{equation}
\end{lemma}
An immediate implication of \eqref{Z++} and \eqref{3/2} is that
\begin{equation}\label{3/2+}
2^{-2L}Z_{2L}^{+,\l}\sim c_+\,L^{-3/2}\,,
%c^{-1}\,L^{-3/2}\leq 2^{-2L}Z^+_{2L}\leq 2^{-2L}Z_{2L}\leq c\,L^{-3/2}\,.
\end{equation}
for some constant $c_+>0$ as soon as $\l<2$. Moreover, \eqref{3/2} and \eqref{3/2+} imply the bounds 
\begin{align}\label{32}
\pi\left(\eta_y\geq 0 \;\;\forall  y\in\{-L,\dots,x\}\,,\;\text{and}\;\; \eta_x=0\right)& \,\asymp\, \pi\left(\eta_x=0\right)\nonumber\\
&\,\asymp \,L^{3/2}(L+x)^{-3/2}(L-x)^{-3/2}\,,
\end{align}
for every $x\in E_L$.
 
\subsection{Monotonicity}\label{mono}
An important property satisfied by the Markov chains introduced above is the monotonicity with respect to the 
partial order \eqref{order}. %We shall repeatedly use it in what follows. 
A convenient way of stating the monotonicity property
is that there exists a coupling $\bbP$ of the trajectories of the Markov chains corresponding  to distinct initial conditions such that
if $\eta\leq \eta'$ then $\bbP$ almost surely 
$v_t^{\eta}\leq v_t^{\eta'}$ for all $t\geq 0$. More generally, one can define a coupling $\bbP$ of trajectories corresponding 
to distinct top/bottom constraints and distinct initial conditions such that
if $\z\leq \z'$, $\xi\leq \xi'$, and $\eta\leq \eta'$, then $\bbP$ almost surely
$v_t^{\eta; \z,\xi}\leq v_t^{\eta'; \z',\xi'}$ for all $t\geq 0$.
Recall that a
function $f:\O\mapsto\bbR$ is said to be increasing if $f(\eta)\leq f(\eta')$ whenever $\eta\leq \eta'$. An event $A$ is
increasing if the indicator function $\ind_A$ is increasing.  The monotonicity property of the dynamics
implies the so-called FKG property of the equilibrium measures $\pi=\pi_{2L}^{\l,\z,\xi}$: for every pair of increasing functions $f,g:\O\mapsto\bbR$, one has $\pi(fg)\geq \pi(f)\pi(g)$. 
We refer to \cite[Section 2]{CMT1} for a more detailed discussion of the monotone coupling
and the consequences of monotonicity. 
%The following  consequences of monotonicity will be often used.

\begin{lemma}\label{density}
Let $\mu$ be a probability on $\O$ and write $f(\eta)=\mu(\eta)/\pi(\eta)$, and
$f_t(\eta) = \nu_t^\mu(\eta)/\pi(\eta)$, $t>0$. If $f$ is increasing then,  for every $t>0$,
$f_t$ is increasing. As a consequence, there exists an increasing event $A$ such that
\begin{equation}\label{eventa}
\|\nu_t^\mu - \pi\| = \nu_t^\mu(A) - \pi(A)\,.
\end{equation}
\end{lemma} 
 \proof
Write $\nu_t^\mu(\eta)=\sum_{\eta_0\in\O}\mu(\eta_0)P_t(\eta_0,\eta)$, where $P_t(\cdot,\cdot)$ stands for the kernel of the Markov chain.
Reversibility then gives 
\begin{equation}\label{density01}
f_t(\eta)=\sum_{\eta_0\in\O}f(\eta_0)\pi(\eta_0)P_t(\eta_0,\eta)/\pi(\eta)
=\sum_{\eta_0\in\O}f(\eta_0)P_t(\eta,\eta_0)\,.
\end{equation}
 Next, 
 let $\bbP$ denote the monotone coupling introduced above and let $\bbE$ denote expectation w.r.t.\ $\bbP$. 
 Then, \eqref{density01} coincides with $\bbE[f(v_t^\eta)]$, and if $\eta\leq\eta'$, 
 $$
 f_t(\eta')-f_t(\eta) = \bbE[f(v_t^{\eta'})-f(v_t^{\eta})] =  \bbE[f(v_t^{\eta'})-f(v_t^{\eta})\,;\;v_t^{\eta'}\geq v_t^{\eta}] \,.
 $$
 Thus, 
 $f_t$ is increasing whenever $f$ is.   
 Finally, it is well known that the total variation distance can be written in the form \eqref{eventa} where $A=\{\eta\,:\,\nu_t^\mu(A)\geq \pi(A)\}$. Since $A=\{f_t\geq 1\}$, $A$ is increasing whenever $f$ is.
\qed

\bigskip
Lemma \ref{le:monopm} compares arbitrary initial conditions to the extremal initial conditions. Lemma~\ref{lemma:hub} states a useful sub-multiplicativity property satisfied by extremal evolutions. For lightness of notation, we state these results only in the case of no top/bottom boundaries, i.e.\ $\z=\vee$, $\xi=\wedge$, but the same applies for general $\z,\xi$ with exactly the same proof. 
%With some abuse of notation, $\wedge$ and $\vee$ are used to denote the maximal and minimal configuration in $\O$, respectively. 
\begin{lemma}
\label{le:monopm} 
For any $t>0$ and any $\eta,\eta'\in\O$:
$$\|\nu_t^{\eta} - \nu_t^{\eta'}\|\leq 4L^2\,\|\nu_t^{\wedge} - \nu_t^{\vee}\|\,.$$
As a consequence,
\begin{equation*}
 \gap=-\lim_{t\to\infty} \frac{1}{t}\log \|\nu_t^{\wedge} - \nu_t^{\vee}\|\
 \,.
\end{equation*}
\end{lemma}
\proof 
Let $\bbP$ denote the monotone coupling as above. 
Then, 
\begin{align*}
\|\nu_t^{\eta} - \nu_t^{\eta'}\|&\leq 
\bbP(v_t^{\eta}\neq v_t^{\eta'})
\leq \bbP(v_t^{\wedge}\neq v_t^{\vee})\\ & \leq 
\sum_{x=-L+1}^{L-1}\sum_{h=-L}^{L-1}
[\bbP((v_t^{\wedge})_x> h) - \bbP((v_t^{\vee})_x> h)]\\ &
\leq 
4L^2\,\|\nu_t^{\wedge} - \nu_t^{\vee}\|\,.
\end{align*}
The second point follows from the first one and the classical characterization \eqref{clasgap} of the spectral gap.
\qed

\begin{lemma}
 \label{lemma:hub}
For any $s,t\geq 0$,
\begin{equation*}
  \|\nu^{\wedge}_{t+s}-\nu^{\vee}_{t+s}\|
  \leq \|\nu^{\wedge}_{t}-\nu^{\vee}_{t}\|\, \|\nu^{\wedge}_{s}-\nu^{\vee}_{s}\|\,.
\end{equation*}
\end{lemma}
\begin{proof}%[Proof of Lemma \ref{lemma:hub}]
With the same argument of Lemma \ref{density}, 
for some increasing event $A$ 
\begin{equation*}
\|\nu^{\wedge}_{t+s}-\nu^{\vee}_{t+s}\|=
\nu^{\wedge}_{t+s}(A)-\nu^{\vee}_{t+s}(A)\,.
\end{equation*}
Let $\rho$  be a coupling beween 
$\nu^{\wedge}_{t}$ and $\nu^{\vee}_{t}$ at fixed time $t\geq 0$. Then 
\begin{align*}
 \nu^{\wedge}_{t+s}(A)-\nu^{\vee}_{t+s}(A)&
 =\int (\nu^{\eta}_{s}(A)-\nu^{\sigma}_{s}(A)) d\rho(\eta,\sigma)\\
					  &=\int (\nu^{\eta}_{s}(A)-\nu^{\sigma}_{s}(A)) \ind(\sigma\neq\eta) d\rho(\eta,\sigma)\\
					  &\leq (\mu^{\wedge}_{s}(A)-\mu^{\vee}_{s}(A))\rho(\sigma\neq\eta)\\
					  &\leq \|\nu^{\wedge}_{s}-\nu^{\vee}_{s}||\,\rho(\sigma\neq\eta)\,.
\end{align*}
To conclude, we take $\r$ as the maximal coupling, i.e.\ such that $\r(\sigma\neq\eta)=\|\nu^{\wedge}_{t}-\nu^{\vee}_{t}\|$.
%In the third line, we use that $A$ is an increasing event and $\mu^{\eta}_{s}\stackrel{sto}{\leq}\mu^{\wedge}_{s}$ for any $\eta\in\Omega$.
\end{proof}

\subsection{Relaxation in one phase}\label{piplus}
Here we obtain results concerning the polymer dynamics in the phase $\pi^+$ defined after \eqref{oplusminus}; cf.\ Proposition \ref{muretto} below. Then, we show that the polymer started at the maximal  configuration $\wedge$ 
relaxes first to the restricted equilibrium $\pi^+$ in a time $O(L^{2+\d})$ for arbitrarily small $\d>0$, 
while %by \eqref{5/2} 
it takes much longer to reach the full equilibrium $\pi$; cf.\ Lemma \ref{lemma:plus} and Lemma \ref{lemma:peq} below. 

Recall the definition \eqref{oplusminus} of the subspace $\O^+\subset\O$,  
where $L\gg \ell$, and $\ell$ diverges as $L\to\infty$; see \eqref{ellL} below.  
The corresponding restricted equilibrium is given by $\pi^{+}=\pi(\cdot|\Omega^+)$.
Note that this is a particular instance of the polymer equilibrium $\pi^{\l,\z,\xi}_{2L}$ with top/bottom boundaries: 
the top is $\xi=\wedge$ while the bottom $\z=\z(\O^+)$ is the lowest element of $\O^+$. 
Similarly, one defines $\O^-=-\O^+$, i.e.\ use \eqref{oplusminus} with $\eta_x>0$ replaced by $\eta_x<0$, and the equilibrium  
$\pi^-$ is defined accordingly. 

Since $\l<1$, the equilibrium bounds \eqref{32} imply that 
\begin{equation}\label{pioplusminus}
\pi(\O^+) = \pi(\O^-) = \frac12 + O(\ell^{-1/2})\,;
\end{equation}
see e.g.\ \cite[Section 2]{CMT1}. In particular, if $\ell$ diverges as $L\to\infty$,  then
\begin{equation}\label{piplusminus}
\big\|\pi-\frac12 (\pi^+ + \pi^-)\big\| = o(1)\,.
\end{equation}
What follows depends only marginally on the precise dependence of
$\ell$ on $L$, provided that $L\gg \ell\gg 1$. For the sake  of simplicity we shall fix its value as 
\begin{equation}\label{ellL}
\ell(L) = (\log L)^\frac14\,.
\end{equation}
This choice turns out to be convenient in the proof of Proposition \ref{muretto} below, but we point out that any choice of the form $\ell(L)=O(L^\e)$ for small $\e>0$ would be sufficient to obtain the same conclusion with a little more work.

We start by establishing a mixing time upper bound for the dynamics constrained to stay in $\O^+$, i.e.\ the process evolving with bottom boundary given by $\z=\z(\O^+)$. To avoid confusion we shall write 
$\mu_t^\eta$ (instead of $\nu^\eta_t$) for the law at time $t$ of this Markov chain with state space $\O^+$ and initial condition $\eta$. We write $\mathcal L^+$ for its generator and $\gap^+$ for the associated spectral gap.
\begin{proposition}\label{muretto}
For every $\e>0$, there exists $L_0=L_0(\e)$ such that for all $L\geq L_0$, for all $t\geq 0$ and all initial conditions $\eta\in\O^+:$
\begin{equation}\label{promuretto10}
\|\mu^{\eta}_t-\pi^+\|\leq 4L^2\exp{\left(-\,t/L^{2+\e}\right)}\,.
\end{equation}
In particular, 
$%\begin{equation*}
\gap^+\geq L^{-2-\e}\,.$
%\end{equation*}
\end{proposition}
\proof 
The last statement follows from \eqref{promuretto10} and \eqref{clasgap}. 
To prove \eqref{promuretto10} we establish that
for every $\e>0$, there is a constant $L_0=L_0(\e)>0$ such that, 
taking $T=L^{2+\e}$,  
we have
\begin{equation}\label{promuretto1}
\|\mu^{\wedge}_T-\mu^{\z}_T\|\leq 1-L^{-\e}\,,
%\leq (1-\g(\ell))\,,
\end{equation}
for all $L\geq L_0(\e)$, where $\z$ stands for 
%. With some abuse of notation here we write $\vee$ for 
the minimal element $\z=\z(\O^+)$ of $\O^+$. 
Once \eqref{promuretto1}
is available, we obtain  \eqref{promuretto10} (with a new value of $\e$) from Lemma \ref{le:monopm}, since 
Lemma \ref{lemma:hub} (which is also valid for the restricted dynamic) and \eqref{promuretto1} imply 
\begin{align*}
\|\mu^{\wedge}_t-\mu^{\z}_t\| &\leq \left(\|\mu^{\wedge}_T-\mu^{\z}_T\|
\right)^{\lfloor t/T\rfloor}\\
&
\leq \exp{\left(-\lfloor t/T\rfloor L^{-\e})\right)}\leq 2\exp{\left(-\,t/L^{2+2\e}\right)}\,,
\end{align*}
for any $L$ large enough. 

To prove \eqref{promuretto1}, 
we divide the sites $x\in\{-L,\dots,L\}$ in three overlapping regions: 
$$I_1=\{-L,\dots,-L+\ell^2)\,, \;I_2=\{-L+\ell,\dots,L-\ell\}\,,\;\text{ and} \; I_3=\{L-\ell^2,\dots,L\}\,,$$
where $\ell=\ell(L)$ is given by \eqref{ellL}.  
 Let $T_{2}=L^{2+\e_1}$, $T_1=L^{\e_1}$ 
 with some $\e_1>0$ such that $T=L^{2+\e}\geq T':=2T_{2} + T_{1}$.
We shall prove (\ref{promuretto1}) with $T$ replaced by $T'$ (this implies the claim since by Lemma \ref{le:monopm} the left hand side of \eqref{promuretto1} 
is monotone as a function of $T$).
 Call $\mu_{T'}^{\eta_0,c}$ the law of the 
``censored'' process obtained as follows. Start from $\eta_0$ at time $0$ and, 
for time $t\in [0,T_2]$ reject all the updates involving $x\notin I_2$. For time $t\in(T_2,T_2+T_1]$ reject all updates involving  $x\notin I_1\cup I_3$,
and for time $t\in(T_2+T_1,2T_2+T_1=T']$ reject all the updates involving $x\notin I_2$. 
From the Peres-Winkler censoring inequality \cite[Theorem 16.5]{notePeres} one has that 
$\mu^{\wedge,c}_{T'}$ stochastically dominates $\mu^{\wedge}_{T'}$.
Similarly, $\mu^{\z,c}_{T'}$ is stochastically dominated by $\mu^{\z}_{T'}$. On the other hand, as in Lemma \ref{density}, one has 
$$
\|\mu^{\wedge}_{T'}-\mu^{\z}_{T'}\| = \mu^{\wedge}_{T'}(A) - \mu^{\z}_{T'}(A)\,,
$$
 where $A\subset\O^+$ is an increasing event.
Therefore, $$\|\mu^{\wedge}_{T'}-\mu^{\z}_{T'}\| \leq  
\mu^{\wedge,c}_{T'}(A) - \mu^{\z,c}_{T'}(A)\leq \|\mu^{\wedge,c}_{T'}-\mu^{\z,c}_{T'}\|\,,$$ and the lemma follows once we show that 
 \begin{equation}\label{muretto10}
  \|\mu^{\wedge,c}_{T'}-\mu^{\z,c}_{T'}\| \leq 1-L^{-\e}
  %\leq (1-\g(\ell))
  \,. 
\end{equation}

To prove (\ref{muretto10}) we shall couple the two configurations 
$\eta^{\wedge,c}_{T'}$, $\eta^{\z,c}_{T'}$ with law $\mu^{\wedge,c}_{T'} $, $\mu^{\z,c}_{T'}$ respectively. 
%Assume for simplicity that $\ell^2$ is even. 
From the analysis of the polymer with the wall \cite[Section 4]{CMT1},
 it is not hard to infer that uniformly in the boundary values at $-L+\ell-1$ and $L-\ell+1$ the system evolving in the region $I_2$ has a mixing time 
 $O(L^2\log L)$. Therefore, after a time $T_2=L^{2+\e_1}$, up to $O(L^{-p})$ corrections for a large constant $p>0$ (for $L\geq L_0(p)$),  
 for any event $E$, $\mu^{\wedge,c}_{T_2}(E)$ coincides with the equilibrium probability of $E$  in $I_2$ 
 with boundary conditions $\wedge$ at $-L+\ell-1$ and $L-\ell+1$. The same applies to $\mu^{\z,c}_{T_2}(E)$ provided the equilibrium is taken with boundary conditions $\z$ at $-L+\ell-1$ and $L-\ell+1$.
 Choose the event $E$ that the configuration is minimal (in $\O^+$) at both 
 $-L+\ell^2-1$ and $L-\ell^2+1$
 (i.e.\ $\eta_{-L+\ell^2-1}=\eta_{L-\ell^2+1}=0$ if $L-\ell^2$ is odd, and $\eta_{-L+\ell^2-1}=\eta_{L-\ell^2+1}=1$ if $L-\ell^2$ is
 even). From known equilibrium estimates \cite[Section 2]{cf:GB}, it is not difficult to show that at equilibrium, with either of the two boundary conditions considered above, the probability of $E$ is 
 bounded below by $c_1\,\ell^{-6}$ for some constant $c_1>0$ depending
 only on $\l$. Therefore, using an independent coupling in the time-lag $[0,T_2]$, we have that the event $E$ occurs for both  $\eta^{\wedge,c}_{T_2}$, $\eta^{\z,c}_{T_2}$
 with probability at least $c\,\ell^{-12}$. 
  
 Next, conditioned on the event $E$ we see that from time $T_2$ up to time $T_2+T_1$ the two processes evolve (in the regions $I_1$ and $I_3$ only)
 with the same boundary conditions (equal to the minimal configuration at both $-L+\ell^2$ and $L-\ell^2$). Since $T_1=L^\e$ and the mixing time of the system in the regions $I_1$ and $I_3$ (which evolve independently) is certainly at most $e^{O(\ell^2)}$, with probability $1+O(L^{-p})$ (conditionally on event $E$) we have that the two configurations coincide in both regions $I_1,I_2$ at time $T_2+T_1$. Therefore if we let the system run only in $I_2$ now for an additional time 
 $T_2$ we have a 
 probability close to $1$ to have the two configurations coinciding everywhere. It follows that, conditionally on the event $E$, there is a probability of, say, at least $1/2$ of no discrepancy between $\eta^{\wedge,c}_{T'}$ and $\eta^{\z,c}_{T'}$. Therefore, letting $\bbP$ denote the coupling described above, 
 \begin{align*}
  \|\mu^{\wedge,c}_{T'}-\mu^{\z,c}_{T'}\| &\leq \bbP(\eta^{\wedge,c}_{T'}\neq \eta^{\z,c}_{T'})\\
  & \leq 
  \bbP(\eta^{\wedge,c}_{T'}\neq \eta^{\z,c}_{T'}\tc E)\bbP(E) + 
  1-\bbP(E)\\ & \leq 1-\frac12\,\bbP(E)\,.
 \end{align*}
  Since $\bbP(E)\geq c/\ell^{12}\geq 2/L^{\e}$, for $L\geq L_0(\e)$, this implies (\ref{muretto10}).
  \qed
  
%  \bigskip \bigskip
 % 
 % \noindent
 % {\em Remark}. 
 % The choice $\ell = (\log L)^{1/4}$ is not really necessary to prove Proposition \ref{muretto}: any bound of the form $\ell = O(L^\e)$ would be sufficient to obtain the same conclusion with a little more work. Indeed, $\ell = (\log L)^{1/4}$ was only used to estimate $O(L^\e)$ the mixing time in the regions $I_1$ and $I_3$ above. These regions are of size $\ell^2$ and if we have an apriori polynomial bound $O(\ell^k)$ on the mixing time in the region $I_1$ with zero boundary conditions, then the case
%  $\ell = O(L^\e)$ can be handled by taking $T_1=L^{\e'}$, $\e'=\e k$.  We note that a polynomial upper bound  $O(L^p)$ on the mixing time (with some very large $p$) can be obtained e.g.\ by using the method of fluxes together 
 % with some extra work. 

\bigskip

We now go back to the model with no top/bottom boundaries,
that is the law
$\nu_t^\eta$ corresponds to the evolution with 
$\xi=\wedge$, $\z=\vee$. 
The next result is crucially based on estimates obtained in \cite[Section 6]{CMT1} for the delocalized regime $\l<1$.  
\begin{lemma}
\label{lemma:plus}
 Uniformly in $t\leq L^{5/2}(\log L)^{-9}$,
\begin{equation*}
 \nu_t^{\wedge}(\Omega^+)= 1+o(1)\,.
\end{equation*}
\end{lemma}
\proof%[Proof of Lemma \ref{lemma:plus}]
Define the event 
$A=\{\ \sum_{x=-L}^L\eta_x < L^{3/2}(\log L)^{-3}\ \}$.  
Proposition 6.2 in \cite{CMT1} proves that $ \nu_t^{\wedge}(A) = o(1)$ uniformly in 
$t\leq L^{5/2}(\log L)^{-9}$. Since $\O^-\subset A$ we have 
\begin{equation}\label{plus1}
 \nu_t^{\wedge}(\Omega^-)=o(1)\,,\quad \;t\leq L^{5/2}(\log L)^{-9}\,.
\end{equation}
Next, let us check that 
\begin{equation}\label{plus2}
 \nu_t^{\wedge}(\Omega^+\tc(\O^-)^c)\geq \pi (\Omega^+\tc(\O^-)^c)
 \,,\quad \;t\geq 0\,.
\end{equation}
To this end, observe that since $(\Omega^-)^c$ is increasing, using Lemma \ref{density} the function 
$$
f(\sigma)=1_{(\Omega^-)^c}(\sigma)\,
\frac{\nu^{\wedge}_t(\sigma\tc(\Omega^-)^c)}{\pi(\sigma\tc (\Omega^-)^c)}=
1_{(\Omega^-)^c}(\sigma)\,
\frac{\nu^{\wedge}_t(\sigma)}{\pi(\sigma)}\,\frac{\pi((\Omega^-)^c)}{\nu^{\wedge}_t((\Omega^-)^c)}\,,$$
is increasing. Since $\O^+\subset (\Omega^-)^c$ is increasing, with the FKG property for $\pi$,  
this implies \eqref{plus2}. 
From \eqref{plus1} and \eqref{plus2} we obtain
$$
\nu_t^{\wedge}(\Omega^+)\geq (1+o(1))\nu_t^{\wedge}(\Omega^+\tc(\O^-)^c)
\geq (1+o(1))\pi (\Omega^+\tc(\O^-)^c) = 1+o(1)\,,
$$
where the last bound follows from \eqref{pioplusminus}.
\qed

\begin{lemma}
 \label{lemma:peq}
For any $\e>0$, uniformly in $t\in [L^{2+\e}, L^{5/2}(\log  L)^{-9}]:$
\begin{equation*}
 \|\nu^{\wedge}_t-\pi^+\| = o(1)\,,
\end{equation*}
where $\pi^+$ is defined by $\pi^+=\pi(\cdot\tc\O^+)$.
\end{lemma}
\proof
Using Lemma \ref{lemma:plus} it is enough to prove
\begin{equation*}
 \|\nu^{\wedge}_t(\cdot\tc\Omega^+)-\pi^+\| = o(1)\,,
\end{equation*}
uniformly in $t\in[L^{2+\e}, L^{5/2}(\log  L)^{-9}]$. 
Consider the function $f:\O\mapsto\bbR$ given by
$$
f(\sigma)=1_{\Omega^+}(\sigma)\,
\frac{\nu^{\wedge}_t(\sigma\tc\Omega^+)}{\pi^+(\sigma)}=1_{\Omega^+}(\sigma)\,
\frac{\nu^{\wedge}_t(\sigma)}{\pi(\sigma)}\,\frac{\pi(\Omega^+)}{\nu^{\wedge}_t(\Omega^+)}\,.
$$
Since $\O^+$ is increasing, Lemma \ref{density} shows that $f$ is increasing.
Therefore, the event
\begin{equation*}
 A=\{\sigma\in\O^+\,:\nu^{\wedge}_t(\sigma\tc\Omega^+)> \pi^+(\sigma) \}
 \end{equation*}
is increasing. 
Using monotonicity we have 
$$\nu^{\wedge}_t(A\tc\Omega^+) = \nu^{\wedge}_t(A)/\nu^{\wedge}_t(\Omega^+)\leq \mu^{\wedge}_t(A)/\nu^{\wedge}_t(\Omega^+)\,,$$
where $ \mu^{\wedge}_t$ denotes the evolution constrained to stay in $\O^+$; see Proposition \ref{muretto}.   
Therefore,
\begin{align*}
 \|\nu^{\wedge}_t(\cdot\tc\Omega^+)-\pi^+\|&=\nu^{\wedge}_t(A\tc\Omega^+)-\pi^+(A) \leq  \frac{
 \mu^{\wedge}_t(A)}{\nu^{\wedge}_t(\Omega^+)}-\pi^+(A)\,.
\end{align*}
The conclusion now follows from Lemma \ref{lemma:plus} and Proposition \ref{muretto}. \qed

\bigskip
The full power of Lemma \ref{lemma:peq} will be seen in the next sections. One of its 
consequences is the fact that the mixing time $\tmix$ can be bounded in terms of the relaxation time 
via
\begin{equation}\label{trelmixo}
\tmix\leq L^{2+\e} + c\,\trel\log L\,,
\end{equation}
for some constant $c>0$. Indeed, \eqref{trelmixo} follows quite easily from Lemma \ref{le:monopm},  Lemma \ref{lemma:peq} and \eqref{trelvar}; see Lemma \ref{trelo} below for a more subtle application of the same reasoning. 
Note that, since $\trel\gg L^{5/2-\e}$, the bound \eqref{trelmixo} improves considerably the standard estimate \eqref{trelmix} by replacing the factor $-\log\pi_*=O(L)$ with a factor $O(\log L)$.

\subsection{Characterization of the principal eigenfunction}
What follows refers to the model with no top/bottom boundaries. 
Recall that $\wedge$, (resp.\ $\vee$) denotes the maximal (resp. minimal) configuration in $\gO$. A function $g:\O\mapsto\bbR$ is called antisymmetric if $g(-\eta)=-g(\eta)$ for all $\eta\in\O$.
The following result gives a precise characterization of one eigenfunction corresponding to $-\gap$.

\begin{proposition}\label{mainvect}
There exists an increasing antisymmetric eigenfunction $g$ of $\cL$, such that  $\|g\|_{\bbL_2(\pi)}=1$. 
It satisfies
\begin{equation}\label{eigeng}
 \mathcal L g = -\gap  \ g.
\end{equation}
Moreover, when $L$ tends to infinity
\begin{equation*}
 g(\wedge)=\|g\|_{\bbL_\infty}=1+o(1).
\end{equation*}
and
\begin{equation}\label{lastp}
  \|g-\left(\ind_{\gO^+}-\ind_{\gO^-}\right)\|_{\bbL_1(\pi)}=o(1).
\end{equation}
\end{proposition}

\begin{proof}
 By decomposing $\ind_{\wedge}-\ind_{\vee}$ on a basis of eigenfunctions of $\mathcal L$ one sees that 
\begin{equation*}
g:=\lim_{t\to \infty} \frac{P_t(\ind_{\wedge}-\ind_{\vee})}{\|P_t(\ind_{\wedge}-\ind_{\vee})\|_{\mathbb L_2(\pi)}}. 
\end{equation*}
is an eigenfunction with unit $\bbL_2$ norm. It is increasing and antisymmetric as $P_t(\ind_{\wedge}-\ind_{\vee})$ is antisymmetric and increasing for all $t$ ($P_t$ preserves monotonicity and symmetries). To prove  \eqref{eigeng}, it suffices to show that the projection of $\ind_{\wedge}-\ind_{\vee}$ on the eigenspace of $\mathcal L$ associated to $-\gap$ is non-zero.
To do so, first observe that by reversibility,
\begin{equation*}
\frac{1}{\pi(\wedge)} \| P_t(\ind_{\wedge}-\ind_{\vee})\|_{\bbL_1(\pi)}=2\,\|\nu_t^{\wedge}-\nu_t^{\vee}\|.
\end{equation*}
Then, by the second point of Lemma \ref{le:monopm}
\begin{align*}
 \lim_{t\to\infty}\frac{1}{t}\log \| P_t(\ind_{\wedge}-\ind_{\vee})\|_{\bbL_2(\pi)}&=\lim_{t\to\infty}\frac{1}{t}\log \| P_t(\ind_{\wedge}-\ind_{\vee})\|_{\bbL_1(\pi)}\\ &= \lim_{t\to\infty}\frac{1}{t}\log \|\nu_t^{\wedge}-\nu_t^{\vee}\|=-\gap\,,
\end{align*}
(where we used equivalence of the norms in finite dimensional spaces).

\medskip

We now estimate the $\bbL_\infty$ norm of $g$. Let $\gep>0$ be small and fixed, and 
let $t_0$ be such that $L^{2+\gep}<t_0<L^{5/2-\gep}$.
The function $g$ is an eigenfunction for $P_{t_0}=\exp(t_0\mathcal L)$, with eigenvalue
$e^{-t_0\gap}$. Therefore
\begin{equation*}
 e^{-t_0\gap} g(\wedge)=\bbE\left[g(v^{\wedge}_{t_0})\right]\leq \pi^+(g)+2g(\wedge) \|\nu^{\wedge}_{t_0}-\pi^+\|,
\end{equation*}
where the last inequality follows from the fact that for any two  measures $\mu,\nu$ and any function $f$, 
\begin{equation*}
 |\mu(f)-\nu(f)|\leq 2\|f\|_{\bbL_\infty}\|\mu-\nu\|\,.
\end{equation*}
Hence, by Lemma \ref{lemma:peq} and the fact that $\gap^{-1}\gg t_0$ (cf.\ \eqref{5/2}):
\begin{equation}\label{oyy}
 g(\wedge)\leq \frac{\pi^+(g)}{e^{-t_0\gap}-2\|\nu^{\wedge}_{t_0}-\pi^+\|}=\pi^+(g)(1+o(1)).
\end{equation}
Moreover, by symmetry and Jensen's inequality
\begin{equation*}
1=\pi(g^2)\geq 2\sum_{\eta \in \gO^+} \pi(\eta) g(\eta)^2=2\pi(\gO^+)\pi^+(g^2)\geq 2\pi(\gO^+)\pi^+(g)^2\,,
\end{equation*}
so that
\begin{equation*}
 \pi^+(g)\leq (2\pi(\gO^+))^{-1/2}=1+o(1).
\end{equation*}
Therefore $g(\wedge)\leq 1+o(1)$ (and it is trivial to notice that $\|g\|_{\bbL_\infty}\geq \|g\|_{\bbL_2(\pi)}=1$).
\\

We turn to the proof of \eqref{lastp}.
First notice that by \eqref{oyy} one has $\pi^+(g)\geq (1+o(1)) g(\wedge)\geq 1+o(1)$ so that
\begin{equation}\label{pipiug}
\pi^+(g)= 1+o(1).
\end{equation}
Next, we prove that the variation of $g$ within $\gO^{\pm}$ is small.
Let $\mathcal L^+$ be the generator of the Markov chain restricted to $\gO^+$, as in Proposition \ref{muretto}. 
The associated Dirichlet form $\mathcal E^+$ is, for a function $f\in \bbL_2(\pi^+)$,
\begin{equation*} 
\mathcal E^+(f,f)=\frac{1}{2}\sum_{\eta,\eta'\in \gO^+}\pi^+(\eta) c(\eta,\eta')[f(\eta')-f(\eta)]^2,
\end{equation*}
where $ c(\eta,\eta')$, $ \eta,\eta'\in \gO^+$, denote the transition rates, which coincide with those of $\mathcal L$. 
If $L$ is large enough, $\gap\leq L^{-5/2+\gep}$ by \eqref{5/2}, while 
$\gap^+\geq L^{-2-\gep}$ by Proposition \ref{muretto},  so that
\begin{multline*}
 L^{-5/2+\gep}\geq  \gap =\mathcal E(g)\geq \pi(\gO^+) \mathcal E^+(g_{|\gO^+})\\
\geq \gap^+\pi(\gO^+)\var_{\pi^+}(g_{|\gO^+})\geq \frac{1}{4}L^{-2-\gep}\var_{\pi^+}(g_{|\gO^+})\,,\qquad\quad\quad\quad\quad\;\;
\end{multline*}
 where
we let $g_{|_{\Omega^+}}:=g \ind_{\Omega^+}$.
Therefore, one has
\begin{equation}\label{varisml}
\var_{\pi^+}(g_{|\gO^+})\leq 4 L^{-1/2+2\gep}\,,
\end{equation}
and the same is true for $\var_{\pi^-}(g_{|\gO^-})$, by antisymmetry.

Next,
\begin{equation*}
 \|g-\ind_{\gO^+}+\ind_{\gO-}\|_{\bbL_1(\pi)}\leq \|\ind_{\gO^+}(g-1)\|_{\bbL_1(\pi)}+\|\ind_{\gO^-}(g+1)\|_{\bbL_1(\pi)}+\|g\ind_{\gO\setminus(\gO^+\cup \gO^-)}\|_{\bbL_1(\pi)}\,.
\end{equation*}
The first two terms of the right-hand side are equal by symmetry. Adding and subtracting $\pi^+(g)$, and using Schwarz' inequality,
\begin{equation*}
  \|\ind_{\gO^+}(g-1)\|_{\bbL_1(\pi)}\leq \pi(\gO^+)\left[|\pi^+(g)-1|+ \sqrt{\var_{\pi^+}(g_{|\gO^+})}\right]=o(1)\,,
\end{equation*}
where the conclusion follows from \eqref{pipiug} and \eqref{varisml}. 
The third term $\|g\ind_{\gO\setminus(\gO^+\cup \gO^-)}\|_{\bbL_1(\pi)}$ is smaller than $\|g\|_{\bbL_\infty}\pi(\gO\setminus(\gO^+\cup \gO^-))=o(1)$.
\end{proof}

\section{Metastability}\label{metastab}
In this section we first prove Theorem \ref{th4}, which is mainly a consequence of the technical lemmas of the previous section and then move to the proof of  Theorem \ref{th3} and its corollary.

\subsection{Proof of Theorem \ref{th4}}

% \begin{proposition}\label{desmes}
%  For any $\gd\geq 0$, uniformly in $t\geq L^{2+\gd}$
% \begin{equation}
% \left\| \mu_t^{L,\wedge}-\left[\frac{1+\exp(-t/\Tr)}{2}\pi^+ + \frac{1-\exp(-t/\Tr)}{2} \pi^-\right]\right\|_{TV}=o(1).
% \end{equation}
% \end{proposition}
% \begin{proof}

We use 
 the notation $T=L^{2+\gd}$.
Equation \eqref{tmixhat} is an easy consequence of \eqref{PTT}.
Indeed, assuming \eqref{PTT}, for $t\ge T$, one has 
\begin{multline}
 \| P_t(\wedge,\cdot)-\pi\|=  \| \Big[\frac{1+e^{-t/\Tr}}{2}\pi^+ +\frac{1-e^{-t/\Tr}}{2}\pi^-\Big]  -\pi\|+o(1)
=\frac{1}{2}e^{-t/\Tr}+o(1).
\end{multline}
To prove the rest of the result, one first shows that proving \eqref{PTT} reduces to prove \eqref{PTT2}:
\begin{multline}
\Big\| \nu_t^{\wedge}-\Big[\frac{1+e^{-t/\Tr}}{2}\pi^+ + \frac{1-e^{-t/\Tr}}{2} \pi^-\Big]\Big\|\\
\leq \left\| \nu_T^{\wedge}P_{t-T}-\pi^+P_{t-T}\right\|
+\Big\|\pi^+P_{t-T}- \Big[\frac{1+e^{-t/\Tr}}{2}\pi^+ + \frac{1-e^{-t/\Tr}}{2} \pi^-\Big] \Big\|
\,.\qquad
\label{secondt}
\end{multline}
The inequality is just triangular inequality, combined with the observation that $\nu_T^{\wedge}P_{T-t}=\nu_t^{\wedge}$.
The first term on the right hand side is smaller than $\left\| \nu_T^{\wedge}-\pi^+\right\|$ (as $P_{T-t}$ contracts the norm) which is itself small, by Lemma \ref{lemma:peq} and the definition of $T$. It remains to estimate the second term i.e.\ to prove \eqref{PTT2}.

To do this, we use the fact that the density of
$\pi^+$ w.r.t.\ $\pi$ is very close to $g$, the eigenfunction described in Proposition \ref{mainvect}, so that the density of $\pi^+ P_t$ must be close to $P_t g$.
 Using reversibility, one can express the densities as follows:
% \begin{equation}
 $\frac{\dd \pi^+ P_t}{\dd \pi}=P_t \frac{\dd \pi^+}{\dd \pi}$.
%\end{equation}
Then we rewrite the second term in \eqref{secondt} as an $\bbL_1$ norm (omitting a harmless factor $1/2$)
\begin{multline*}
\Big\|P_{t-T} \frac{\dd \pi^+}{\dd \pi}- \frac{1}{2\pi(\gO^+)}\left(\ind_{\gO^+} + \ind_{\gO^-} \right)-\frac{1}{2\pi(\gO^+)}\,e^{-t/\Tr}\left(\ind_{\gO^+} - \ind_{\gO^-}\right)\Big\|_{\bbL_1(\pi)}\\
\leq \Big\| P_{t-T} \frac{\dd \pi^+}{\dd \pi}- \frac{1}{2\pi(\gO^+)}\left(\ind_{\gO^+}+\ind_{\gO^-}\right)- \frac{1}{2\pi(\gO^+)}\,e^{-t/\Tr}g\Big\|_{\bbL_1(\pi)} \\
+  \frac{1}{2\pi(\gO^+)}\,\|e^{-t/\Tr}\left(\ind_{\gO^+} - \ind_{\gO^-}-g\right)\|_{\bbL_1(\pi)}.
\end{multline*}
The last term above is small by Proposition \ref{mainvect}.
From \eqref{pioplusminus} we know that $2\pi(\gO^+)=1+o(1)$. 
One can then estimate the first term 
\begin{multline}\label{hepep}
 \Big \| P_{t-T} \frac{\dd \pi^+}{\dd \pi}- \frac{1}{2\pi(\gO^+)}\left(\ind_{\gO^+}+\ind_{\gO^-}\right)- \frac{1}{2\pi(\gO^+)}\,e^{-t/\Tr}g\Big\|_{\bbL_1(\pi)}\\ \le
 \| \left(\ind_{\gO^+}+\ind_{\gO^-}\right)-\ind \|_{\bbL_1(\pi)}+\Big\|P_t \Big(\frac{\dd \pi^+}{\dd \pi}-\ind-g\Big)\Big\|_{\bbL_1(\pi)}
+ o(1)\,,\end{multline}
where we used the triangular inequality, the fact that $P_t \ind=\ind$, and 
$$P_{t-T} g=e^{-(t-T)/\Tr}g
= (e^{-t/\Tr}+o(1))g\,,$$
which follows from $T=o(\Tr)$. On the right hand-side of \eqref{hepep}, the first term is small by \eqref{piplusminus}
%as $\pi(\gO^+)=1/2+o(1)$, 
and the second is bounded by $\|\frac{\dd \pi^+}{\dd \pi}-\ind-g\|_{\bbL_1(\pi)}$, which is small by Proposition \ref{mainvect}.
\qed

\subsection{Proof of Theorem \ref{th3}}
% 
% \begin{theorem}\label{expon}
%  One can find sets $S^{\pm}_L\subset \gO_L^{\pm}$ such that $\pi(S^{\pm}_L)=1/2+o(1)$, and that uniformly in $\eta \in S^+_L$ (resp. $S^{-}_L$) and uniformly in $t$
% \begin{equation}
%  \bbP_L^{\eta}\left[\tau^{-}>t\right]=\exp(-t/2\Tr)+o(1).
% \end{equation}
% (resp. the same with $\tau^+$) where
% \begin{equation}
%  \tau^-:=\inf\{t \ge0, \eta_t\in S^-_L\}
% \end{equation}
% \end{theorem}
Theorem \ref{th4} gives some intuition on why the result should be true, and it will be used to determine the time of the jump from one state to the other.
However, one needs another key ingredient to get the result, namely
the description of the quasi-stationary distribution.
The reason for this is that starting from the quasi-stationary distribution, a killed process dies exactly at exponential rate; see Section \ref{killing}. Therefore, most of our effort will focus on stochastic comparison with quasi-stationary distribution. 
Let us first give a brief roadmap to help the reader
through the proof of
Theorem \ref{th3}.
\\

\noindent
{\em Step 1}. The sets $S^\pm$ of Theorem  \ref{th3} for which we have the desired exponential hitting time description %rate from one to another 
are constructed by successively refining a first attempt.
One first defines $S^{0,\pm}$ as the sets of polymer configurations where the
eigenfunction $g$ in Proposition \ref{mainvect} is positive
(negative) and one verifies that their equilibrium probability is
$\frac 12
+o(1)$. Then one examines  the Dirichlet problem associated to the process killed in $S^{0,-}$
($S^{0,+}$) and one proves that the corresponding
eigenvalue $\gamma_0$ is of the same order as the spectral gap apart from a
crucial unspecified multiplicative factor in $[1/2,1]$. Similarly one verifies that the corresponding quasi-stationary
measure is very close to $\pi^\pm$, the equilibrium
measure $\pi$ conditioned to be in $S^{0,+}$ ($S^{0,-}$). In this way
we get the exponentiality of the hitting time of e.g. $S^{0,-}$ starting
from $\pi^+$ with a rate which is, modulo a multiplicative factor in
$[1/2,1]$, the
spectral gap (see Lemma \ref{atlantos}).   
\\

\noindent
{\em Step 2}.
Next one appropriately defines new sets $S^{1,\pm} \subset S^{0,\pm}$ 
in order to guarantee that this time the corresponding Dirichlet eigenvalue 
$\gamma_1$ is equal to $(\frac 12+o(1)) \gap$, and that the
hitting time of $S^{1,\mp}$ starting from
equilibrium conditioned to $S^{1,\pm}$
is exponential (with the correct rate). Again one of the key points is to
show that  $\pi^+$ is close to the quasi-stationary distribution associated the
process killed on entering $S^{1,-}$, and that the equilibrium
probability of $S^{1,\pm}$ is still $\frac 12+o(1)$.
\\

\noindent
{\em Step 3}.
Finally, one defines the final sets $S^{2,\pm} \subset S^{1,\pm}$ in
such a way that: a) the hitting time of $S^{2,\mp}$ starting from
\emph{any} configuration in $S^{2,\pm}$ (and not just from the
conditional equilibrium) 
is also exponential with the correct rate $\frac 12 \gap$; b) the equilibrium
probability of $S^{2,\pm}$ is still $\frac 12 +o(1)$. 
\\

\noindent
It is now time to begin the implementation of the above strategy. Let 
\begin{equation}
\label{eq:S0}
S^{0,+}:=\left\{ \eta\in \O, \ g(\eta)>0\right\},
\end{equation} 
where $g$ is the eigenfunction defined by Proposition \ref{mainvect} and $S^{0,-}\equiv-S^{0,+}$.
From Proposition~\ref{mainvect},
\begin{equation}\label{clo}
\| \ind_{S^{0,+}}-\ind_{\gO^+}\|_{\bbL_1(\pi)}=o(1).
\end{equation}
In particular, $\pi(S^{0,\pm})=1/2+o(1)$.
Let $S^{-}\subset S^{0,-}$ be a decreasing event.
We consider the quasi-stationary distribution $\nu^{+}:=\nu_{S^-}$ of the process 
killed when it hits $S^{-}$. Let $P_t^{*}=P_t^{S^-}$, resp.\
 $\mathcal L^{*}=\cL^{S^-}$, denote the semi-group, resp.\ the generator, associated to this process
 (see Section \ref{killing}),  
$-\g_{S^{-}}=-\gga$ be the largest eigenvalue of $\mathcal L^{*}$ and 
$\tau^-=\t_{S^-}$ be the  hitting time of $S^{-}$. From \eqref{killedp6}:
\begin{equation}\label{qsd}
 \bbP^{\nu^+}(\tau^->t)=e^{-\gga t}.
\end{equation}
Our first step is to prove that if $S^{-}$ has non-negligible measure, then $\gga$ is of the same order of the $\gap$. More precisely:

\begin{lemma} \label{wesh}
For any $ S^- \subset S^{0,-}$, one has 
$\pi(S^{-}) \leq \gga \,\Tr \leq 1 $.
\end{lemma}
%\begin{proof}
\proof
 The bound $\pi(S^{-}) \leq \gga \,\Tr$ %right-hand side 
 is rather standard, but we include its proof for the sake of completeness. 
 %As the generator of the killed process $\mathcal L^{*}$, is a negative self-adjoint operator in $\bbL_2(\pi)$,
%\begin{equation}\label{gggaaa}
% \gga=\mintwo{f\in \bbL_2(\pi)}{f_{|S^-}=0} \frac{\ll -\mathcal L^{*}f,f  \rr_{\pi}}{\ll f,f  \rr_{\pi}}=\mintwo{f\in \bbL_2(\pi)}{f_{|S^-}=0} \frac{\ll -\mathcal Lf,f  \rr_{\pi}}{\ll f,f  \rr_{\pi}}\,.
%\end{equation}
Let $f_0=g_{S^-}$ denote the minimizer in the variational principle defining $\g=\g_{S^-}$; see \eqref{killedp3}. Then
\begin{equation*}
 \var_{\pi} (f_0)=\ll f_0,f_0  \rr_{\pi}- \ll f_0,\ind_{(S^{-})^c} \rr_{\pi}^2 \geq \ll f_0,f_0  \rr_{\pi}\pi(S^-)\,,
\end{equation*}
where we used the Cauchy-Schwarz inequality for $\ll f_0,\ind_{(S^{-})^c} \rr_{\pi}^2$.
Therefore
\begin{equation*}
\gga=\frac{\cE(f_0,f_0)}{\pi(f_0^2)}
%\frac{\ll -\mathcal Lf_0,f_0  \rr_{\pi}}{\ll f_0,f_0  \rr_{\pi}}
\geq \pi(S^-)\,\frac{\cE(f_0,f_0)}{\var_\pi(f_0)}
%\frac{\ll -\mathcal Lf_0,f_0  \rr_{\pi}}{\var_{\pi}(f_0)}
\geq \pi(S^-)\,\gap\,.
\end{equation*}

 As for the bound $\gga \,\Tr \leq 1 $, 
 $\gga_{S^-}$ being a non-decreasing function of $S^-$ (for the inclusion), it is sufficient to prove the result for the maximal case $S^-=S^{0,-}$. 
%\begin{equation}\label{gggaaa}
%\gga_0=\min_{f\in \bbL_2(\pi)} \frac{\ll -\mathcal L^{*,0}f,f  \rr_{\pi}}{\ll f,f  \rr_{\pi}}.
%\end{equation}
Let $g$ be the eigenfunction defined in Proposition \ref{mainvect}.
%increasing antisymmetric eigenfunction associated to $-\gap$ for $\mathcal L$ (coming from Proposition \ref{mainvect}).
From \eqref{killedp2}, for all $\eta\in (S^{0,-})^c$
\begin{equation*}
 -(\mathcal L^{*} g_{|_{(S^{0,-})^c}})(\eta)= -(\mathcal L g)(\eta)+\sum_{\eta' \in S^{0,-}}c(\eta,\eta') g(\eta')\leq -(\mathcal L g)(\eta)=\gap\  g(\eta)\,,
\end{equation*}
where we use the fact that $g(\eta')<0$ for $\eta' \in S^{0,-}$.
Plugging this into \eqref{killedp3}, and using $g_{|_{(S^{0,-})^c}}\geq 0$, one gets
\begin{equation*}
 \gga
 \leq \frac{\ll - \mathcal L^{*}g_{|_{(S^{0,-})^c}},g_{|_{(S^{0,-})^c}}
% g_{|(S^{0,-})^c},g_{|(S^{0,-})^c}  
\rr_{\pi}}{
%\ll 
\pi\big(g_{|_{(S^{0,-})^c}}^2\big)
%g_{|_{(S^{0,-})^c}
}
%g_{|(S^{0,-})^c},g_{|(S^{0,-})^c}) 
%\rr_{\pi}}
\leq \gap\,.
\end{equation*}
%\end{proof}
\qed

\medskip
Next, we prove that the quasi-stationary distribution $\nu^+$ for the process killed on $S^{-}$ is very close to $\pi^+$ if $S^{-}$ has probability close to $1/2$.

\begin{lemma}\label{wwesh}
Uniformly for all  decreasing events $S^-\subset S^{0,-}$,
\begin{equation*}
 \|\nu^{+} - \pi^+\|\leq(2-4\pi(S^-))+o(1)\,.
\end{equation*}
\end{lemma}
\begin{proof}
We use triangular inequality to get
\begin{equation}\label{above}
 \|\nu^{+} - \pi^+\|\leq  \|\nu^{+} - \pi( \cdot \ | (S^{-})^c)\|+\| \pi^+-\pi( \cdot \ | (S^{-})^c)\|.
\end{equation}
%The second term is easier to handle, and we do that at the end of the proof.
We start with the first term. %we use several monotonicity arguments. 
First, from \eqref{killedp5} one has
the characterization
\begin{equation*}
  \nu^+=\lim_{t\to \infty} \frac{\gd_{\wedge} P_t^{*}}{\gd_{\wedge} P_t^{*}(\gO)}\,.
 \end{equation*}
Since the operator $P_t^{*}$ preserves monotonicity ($S^-$ is decreasing), arguing as in \cite[Lemma 16.6]{notePeres}, the density $\frac{\dd[ \gd_{\wedge} P_t^{*}]}{\dd \pi}$ is seen to be an increasing function
for every fixed $t\geq 0$.  Hence, passing to the limit $t\to\infty$, $\dd\nu^+/\dd \pi$ is an increasing function.
Therefore, 
\begin{equation*}
 A:=\left\{\eta\in (S^-)^c, \text{ such that } \frac{\nu^+(\eta)\pi((S^-)^c)}{\pi(\eta)}>1\right\}
\end{equation*}
is an increasing event. From standard properties of the total variation distance
\begin{equation*}
 \| \nu^+- \pi(\cdot | (S^-)^c\|= \nu^+(A)- \pi(A\ | (S^-)^c).
\end{equation*}
We shall prove that $\nu^+(A)$ is smaller than $\pi^+(A)+o(1)$ by the use of monotonicity and a chain of comparisons. 
Recall the notation $T=L^{2+\gd}$ ($\gd\in(0,1/4)$). We first compare
$\nu^+$ to $\nu^+P_T$: remark that
\begin{equation*}
 \nu^+P_T=\nu^+ P^{*}_T+\nu^+ (P_T-P^{*}_T)
\end{equation*}
where the two terms of the decomposition are positive measures. From quasi-stationarity one has $\nu^+ P^{*}_T=e^{-\gga T}\nu^+$ and therefore the total mass of the second term above is $1-e^{-\gga T}$. Hence
\begin{equation}\label{423}
\| \nu^+P_T-\nu^+\|=\frac{1}{2}\Big
\|\frac{\dd [\nu^+ (P_T-P^{*}_T)]}{\dd \pi}-\frac{\dd \nu^+}{\dd \pi}(1-e^{-\gga T})\Big\|_{\bbL_1(\pi)}\!\!\leq 1-e^{-\gga\, T}=o(1).
\end{equation}
The last equality comes from Lemma \ref{wesh} and the fact that $\trel\gg T$.
Next, from Lemma \ref{density}, $\gd_{\wedge}P_T$ 
stochastically dominates $\nu^+P_T$ so that $[\nu^+P_T](A)\leq [\gd_{\wedge}P_T](A)$.
Hence, from Lemma \ref{lemma:peq} and \eqref{423}:
% \begin{align*}
%  \nu^+(A)&\leq \pi^+(A)+\gd_{\wedge}P_T(A) - \nu^+P_T(A) \nonumber\\
%  & \leq \pi^+(A) +
%  \|\gd_{\wedge}P_T-\pi^+\|+\| \nu^+P_T-\nu^+\|=\pi^+(A)+o(1).
% \end{align*}
\begin{align*}
 \nu^+(A)&\leq \nu^+P_T(A) +o(1)\le \gd_{\wedge}P_T(A) +o(1)\nonumber\\
 & \leq \pi^+(A) +
 \|\gd_{\wedge}P_T-\pi^+\|+o(1)=\pi^+(A)+o(1).
\end{align*}
Therefore, going back to \eqref{above} 
\begin{align}\label{minus1}
 \|\nu^{+} - \pi^+\|&
 \leq \nu^+(A)- \pi(A\ | (S^-)^c) + 
 \| \pi^+-\pi( \cdot \ | (S^{-})^c)\| \\
& \leq \pi^+(A)- \pi(A\ | (S^-)^c) + 
 \| \pi^+-\pi( \cdot \ | (S^{-})^c)\| +
 o(1) \leq 2  \| \pi^+-\pi( \cdot \ | (S^{-})^c)\|+o(1)\,.
\nonumber
\end{align}
To estimate the right-hand side of \eqref{minus1}, notice that
\begin{equation}\label{minus2}
 \| \pi^+-\pi( \cdot \ | (S^{-})^c)\|\leq   \| \pi^+-\pi( \cdot \ | (S^{0,-})^c)\|+\| \pi(\cdot \ | (S^{0,-})^c)-\pi( \cdot \ | (S^{-})^c)\|\,,
\end{equation}
and the first term is $o(1)$ by Proposition \ref{mainvect}.
Moreover, since $S^{-}\subset S^{0,-}$
\begin{align}\label{minus3}
  \| \pi(\cdot \ | (S^{0,-})^c)-\pi( \cdot \ | (S^{-})^c)\|&=\frac{\pi((S^{-})^c)-\pi(( S^{0,-})^c)}{\pi((S^{-})^c)}\nonumber\\
&=\frac{1/2-\pi(S^-)+o(1)}{1-\pi(S^-)}\leq 1-2\pi(S^-)+o(1)\,.
\end{align}
Combining \eqref{minus1}, \eqref{minus2} and \eqref{minus3}, the desired result follows.
\end{proof}

Now one uses the fact that $\nu^+$ and $\pi^+$ are close in total variation distance to estimate the jumping time to $S^-$ starting from either $\wedge$ or from $\pi^+$. For the rest of this section, one defines, in analogy with $\tau^-$, the hitting times 
$\tau^{i,-}$ (resp. $\tau^{i,+}$), $(i=0,1,2)$ of the sets $S^{i,-}$ (resp. $S^{i,+}$) to be defined.

\begin{lemma}\label{atlantos}
 Uniformly for all $t\geq 0$ and all decreasing $S^{-}\subset S^{0,-}$, setting $\g=\g_{S^-}$:
\begin{itemize}
\item[(i)] $|\bbP^{\wedge}\left[\tau^->t\right]-e^{-\gga t}|\le(2-4\pi(S^-))+o(1)$
\item[(ii)] $|\bbP^{\pi^+}\left[\tau^->t\right]-e^{-\gga t}|\le(2-4\pi(S^-))+o(1)$
\end{itemize}
In particular, for $S^{-}=S^{0,-}$, setting $\g_0=\g_{S^{0,-}}$:
\begin{itemize}
\item[(iii)] $\bbP^{\wedge}\left[\tau^{0,-}>t\right]=e^{-\gga_0 t}+o(1)$
\item[(iv)] $\bbP^{\pi^+}\left[\tau^{0,-}>t\right]=e^{-\gga_0 t}+ o(1)$
\end{itemize}
\end{lemma}
\begin{proof}
Item (ii) follows from \eqref{qsd} and Lemma \ref{wwesh}. Indeed,
\begin{equation*}
|\bbP^{\pi^+}\left[\tau^->t\right]-\bbP^{\nu^+}\left[\tau^->t\right]|\leq \|\pi^+-\nu_+\|\,.
\end{equation*}
For item (i) (lower bound), we use the fact that $S^-$ is a decreasing event to get that 
\begin{equation*}
 \bbP^{\wedge}\left[\tau^->t\right]\geq \bbP^{\nu^+}\left[\tau^->t\right]=e^{-\gga t}.
\end{equation*}
For the upper bound it is sufficient to prove the result for $t\geq T=L^{2+\d}$, since $\g^{-1}\gg T$
(a consequence of Lemma \ref{wesh} and  \eqref{5/2}). One defines
\begin{equation*}
\bar\tau^{-}=\inf\{t\geq T,\ \eta(t)\in S^-\}.
\end{equation*}
Then, by the Markov property and quasi-stationarity
\begin{multline}
 \bbP^{\wedge}\left[\tau^->t\right]\leq \bbP^{\wedge}\left[\bar\tau^->t\right]= \bbP^{\gd_\wedge P_T}\left[\tau^->t-T\right]\leq \|\gd_\wedge P_T-\nu^+\|+e^{-\gga (t-T)}\\
\leq \|\gd_\wedge P_T-\pi^+\|+\|\pi^+-\nu^+\| + e^{-\gga t}+ o(1)\,,
\end{multline}
where we use $e^{-\gga (t-T)}=e^{-\gga t}+ o(1)$, which follows from $\g^{-1}\gg T$. 
The result then follows from Lemma 
\ref{lemma:peq} and Lemma \ref{wwesh}.  
Items (iii) and (iv) are consequences of (i) and (ii) and the fact that $\pi(S^{0,-})=1/2+o(1)$.
\end{proof}

From the  previously stated results, one may conclude that there exists $\gd_L$, a decreasing sequence tending to zero when $L$ tends to infinity, such that for every $t>0$:
\begin{equation}\label{hypo}
\begin{split}
\bbP^{\wedge}\left[\tau^{0,-}>t\right]&\leq e^{-\gga_0 t}+\gd_L,\\
\bbP^{\pi^+}\left[\tau^{0,-}>t\right]&\geq e^{-\gga_0 t}-\gd_L,\\
\pi(\gO^+)&\geq 1/2-\gd_L\\
 \pi^+( S^{0,+})&\geq 1-\gd_L.\\
\end{split}
\end{equation}
%With a slight abuse of notation, we will use again the notation $\gd_L$ in the sequel for some other decreasing sequences tending to zero.
Given such a $\gd_L$, one defines $S^{1,+}$ to be

\begin{equation*}
 S^{1,+}:=S^{0,+}\cap \{\eta\in \gO^+ \,:\;   \bbP^{\eta}\left[\tau^{0,-}>t\right]\geq e^{-t\gga_0}-3(\gd_L)^{1/4}\,,\;\,\forall t>0 \ \}
\end{equation*}
and $S^{1,-}\equiv -S^{1,+}$.
\begin{lemma}\label{SS1}
The set $S^{1,+}$ satisfies %so that it has the following properties:
 \begin{equation}
\begin{split}
\label{notquite}
\pi(S^{1,+})&=\frac{1}{2}+o(1)\,,\quad \text{and}\\
\bbP^{\eta}\left[\tau^{0,-}>t\right]&= e^{-t\gga_0}+o(1)\,,\quad  \text{ uniformly in} \;t\geq 0\,,\; \text{and } \eta \in S^{1,+}.
\end{split}
\end{equation}
\end{lemma}
\begin{proof}
The lower bound in the second point follows from the definition of $S^{1,\pm}$. 
For the upper
bound, it is just a consequence of the fact that 
\begin{equation}\label{mont}
\bbP^{\xi}\left[\tau^{0,-}>t\right]\leq \bbP^{\wedge}\left[\tau^{0,-}>t\right]\leq e^{-\gga_0 t} + \d_L\,,
\end{equation} for any $\xi\in\O$, by monotonicity, where the last bound follows from point $(iii)$ of Lemma \ref{atlantos}.  We turn to a proof of the first point.
For $t\geq 0$, one defines
\begin{equation*}
 S^{1,+,t}:=\{\eta\in \gO^+ :\;     \bbP^{\eta}\left[\tau^{0,-}>t\right]\geq e^{-t\gga_0}-2\gd_L^{1/4}\}.
\end{equation*}
From the second line in \eqref{hypo}, and using \eqref{mont} for all $\xi\in S^{1,+,t}$:
\begin{multline}
 e^{-\gga_0 t}-\gd_L\leq \bbP^{\pi^+}\left[\tau^{0,-}>t\right]
\leq \pi^+(S^{1,+,t})(e^{-\gga_0 t}+\gd_L)+(1-\pi^+(S^{1,+,t}))(e^{-\gga_0 t}-2\gd_L^{1/4})\,.
\end{multline}
This gives $-2\d_L \leq 2\gd_L^{1/4}\pi^+(S^{1,+,t})) -2\gd_L^{1/4} $, i.e.
\begin{equation}\label{hhop}
 \pi^+(S^{1,+,t})\geq 1-\gd_L^{3/4}.
\end{equation}
Next, define $S_i:=S^{1,+,i\gd_L^{1/4}\gga_0^{-1}}$, $i\in \bbN$. We claim that for all $L$ sufficiently large:
\begin{equation}\label{claims1}
 S^{1,+}\supset \bigcap_{i=1}^{\lfloor \gd_L^{-1/2}\rfloor} %S^{1,+,i\gd_L^{1/4}\gga_0^{-1}}\cap S^{0,+}:=\bigcap_{i=1}^{\lfloor \gd_L^{-1/2}\rfloor} 
 S_i\cap S^{0,+}\,.
\end{equation}
%(where the last equality is there to simplify the notation).
Indeed, let $\eta \in \bigcap_{i=1}^{\lfloor \gd_L^{-1/2}\rfloor} S_i$.
If $t\leq \gd_L^{-1/4}\gga_0^{-1}/2$, then clearly $\eta\in S_{\lceil t\gga_0\gd_L^{-1/4}\rceil}$ if $L$ is large enough.
In particular,
\begin{multline}
\bbP^{\eta} [\tau^{0,-}>t]\geq \bbP^{\eta} [\tau^{0,-}>\lceil t\gga_0\gd_L^{-1/4}\rceil\gd_L^{1/4}\gga_0^{-1}]
\geq e^{-\gga_0 t-\gd_L^{1/4}}-2\gd_L^{1/4}\geq e^{-\gga_0 t}-3\gd_L^{1/4}.
\end{multline}
If on the other hand $t\geq  \gd_L^{-1/4}\gga_0^{-1}/2$, then
\begin{equation*}
e^{-t\gga_0}-3\gd_L^{1/4}\leq 0\leq \bbP^{\eta} [\tau^{0,-}>t],
\end{equation*}
provided $\gd_L$ is small enough, i.e. $L$ is large enough. This proves \eqref{claims1}.
Moreover one has, from \eqref{hhop} and the fourth line of 
\eqref{hypo},
\begin{equation*}
 \pi^+\Big(\cap_{i=1}^{\lfloor \gd_L^{-1/2}\rfloor}S_i\cap S^{0,+}\Big)\geq 1-\gd_L^{1/4}-\gd_L.
\end{equation*}
From the third line of \eqref{hypo}, one gets that 
\begin{equation*}
 \pi\Big(\cap_{i=1}^{\lfloor \gd_L^{-1/2}\rfloor}S_i\cap S^{0,+}\Big)
 %\left(\bigcap_{i=1}^{\lfloor \gd_L^{-1/2}\rfloor}S_i \cap S^{0,+}\right)
 \ge(1-\gd_L^{1/4}-\gd_L)(1/2-\gd_L)\,.
\end{equation*}
This last estimate together with \eqref{claims1} implies the first statement of the lemma.
\end{proof}

The previous results allow us to compute the  value of $\gga_1 = \g_{S^{1,-}}$.

\begin{lemma}\label{gamma1}
Let $-\gga_{1}$ be the largest eigenvalue of $\mathcal L^{S^{1,-}}$, the generator of the process killed when it hits $S^{1,-}$.
We have
\begin{equation*}
 \gga_{1}\, \Tr= 1/2+o(1).
\end{equation*}
\end{lemma}
It is important to recall that, in contrast to $\gamma_1$, the 
eigenvalue $\gamma_0$ of the process killed in $\mathcal L^{S^{0,-}}$ was
estimated only up to a factor $2$ (cf. Lemma \ref{wesh}).
\begin{proof}
The inequality  $\gga_{1} \Tr\geq  1/2+o(1)$ comes from Lemma \ref{SS1} and Lemma \ref{wesh}.
Recall the definitions 
\begin{equation*}
 \begin{split}
\tau^{-,1}&=\inf\left\{t\geq 0, \eta(t) \in S^{1,-}\right\},\\
\tau^{+,0}&=\inf\left\{t\geq 0, \eta(t) \in S^{0,+}\right\}\,.
 \end{split}
\end{equation*}
According to Theorem \ref{th4}, \eqref{clo}, Lemma \ref{wesh}, Lemma \ref{atlantos}(i) and Lemma \ref{SS1}, one can find a new sequence $\gd_L$ going to zero such that
\begin{equation*}\begin{split}
 \nu_t^{\wedge}((S^{0,+})^c)&\leq (1-e^{-t/\Tr})/2+\gd_L
 \quad\mbox{for every}\quad  t>0\,, \\
 \bbP^{\xi}\left[ \tau^{+,0}>t \right]&\geq e^{-\gga_0 t}-\gd_L\,
 \quad\mbox{for every}\quad t\geq 0
\quad\mbox{and}\quad \xi\in S^{1,-}\,,\\
  \bbP^{\wedge}\left[\tau^{-,1}>t\right]&\leq e^{-\gga_1 t}+\gd_L\, \quad \quad\mbox{for every}\quad t\geq 0\,.
\end{split}\end{equation*}
Next, define 
\begin{equation*}
\bar \tau^{+,0}=\inf\left\{t\geq \tau^{-,1}\,: \;\eta(t) \in S^{0,+}\right\}
\end{equation*}
to be the first time the process enters $S^{0,+}$ after entering for
the first time $S^{1,-}$.
One has
\begin{multline}
\frac{1-e^{-t/\Tr}}{2}+\gd_L\geq \nu_t^{\wedge}((S^{0,+})^c)\geq \bbP^{\wedge}\left[\tau^{-,1}<t,\ \bar \tau^{+,0}>t \right]\\
\geq\bbE^{\wedge}\left[\ind\{\tau^{-,1}<t\} \bbP^{\eta(\tau^{-,1})}\left[\tau^{+,0}>t\right]\right]
\geq \bbP^{\wedge}\left[\tau^{-,1}<t\right]\min_{\eta\in S^{-,1}}\bbP^\eta\left[\tau^{+,0}>t\right]\\
\geq (1-e^{-\gga_1 t}-\gd_L)(e^{-t\gga_0}-\gd_L).
\end{multline}
We use this inequality for $t=\gga_1^{-1}\gd_L^{1/3}$ and get (using the fact that $x-x^2\leq 1-\exp(-x)\leq x$ for $x$ small enough)
\begin{equation*}
 \gd_L^{1/3}/(2\gga_1 \Tr) \geq (\gd_L^{1/3}-\gd_L^{2/3}-\gd_L)(1-\gd_L^{1/3}\gga_0/\gga_1-\gd_L)-\gd_L.
\end{equation*}
As $1\leq \gga_0/\gga_1\leq 3$ (for $L$ large enough, cf.\ Lemma \ref{wesh}), all of this gives us

\begin{equation*}
\gga_1 \Tr \leq \frac{1}{2[(1-\gd_L^{1/3}-\gd_L^{2/3})(1-3\gd^{1/3}-\gd_L)-\gd_L^{2/3}]},
 \end{equation*}
which ends the proof.
\end{proof}

Once again assume that $\gd_L$ is a sequence going
to zero, this time such that one has for every $t>0$
\begin{equation}
 \begin{split}
  \bbP^{\wedge}\left[\tau^{1,-}> t \right]&\leq e^{-\gga_1 t} +\gd_L,\\
  \bbP^{\pi^+}\left[\tau^{1,-}> t \right]&\geq e^{-\gga_1 t}-\gd_L, \\
  \pi(S^{1,+})&\geq 1/2-\gd_L.
\label{eq:2607}
  \end{split}
\end{equation}
Note that the sets $S^{1,\pm}$ are  not yet good candidates for the
sets $S^\pm$ of Theorem \ref{th3}, the reason being that 
\eqref{eq:2607} and Lemma \ref{gamma1} say that the hitting time of $S^{1,-}$ is
exponential with the correct rate, but only if one starts from either
the maximal configuration or from $\pi^+$, while we want this to hold
uniformly in the initial condition in $S^+$.
We need therefore a final 
step in order to fix this problem. We set
\begin{equation*}
 S^{2,+}:=\left\{\eta\in S^{1,+} \,:\; \ \bbP^{\eta}\left[\tau^{1,-}> t \right]\geq e^{-\gga_1 t}-3\gd_L^{1/4} \right\}\,,
\end{equation*}
and define $S^{2,-}=-S^{2,+}$. %$ be the set symmetric of $ S^{2,+}$.
The same computations of Lemma \ref{SS1} prove %succeed into showing
\begin{lemma}\label{same}
$\pi(S^{2,+})= 1/2+o(1)$.
\end{lemma}

Now we are ready to finish the proof of Theorem \ref{th3}, with $S^+:=S^{2,+}$.
Let $-\gamma_2$ be the largest eigenvalue of the generator $\mathcal L^{S^{2,-}}$ of the process killed when it reaches $S^{2,-}$.
From Lemma \ref{wesh}, one has $\pi(S^{2,+})\,\gap\leq \gga_2\leq \gga_1$. Therefore, Lemma \ref{same} yields 

\begin{equation}\label{gga2}
 \gga_2=(1/2+o(1))\,\gap\,.
\end{equation}
Let $\tau^{2,-}$ be the hitting time of $S^{2,-}$.
For any $\eta \in S^{2,+}$ (this is actually true for any $\eta$ in $\gO$), we get from monotonicity, Lemma \ref{atlantos}(i), Lemma \ref{same} and \eqref{gga2}
\begin{equation*}
 \bbP^{\eta}(\tau^{2,-}>t)\leq 
 \bbP^{\wedge}(\tau^{2,-}>t)\leq e^{-\gga_2 t}+ o(1)+(2-4\pi(S^{2,-}))\leq e^{- \frac{t}{2\Tr}}+o(1),
\end{equation*}
where $o(1)$  is uniform in $t$.
On the other side, the definition of $S^{2,+}$ and the obvious bound $\t^{2,-}\geq \t^{1,-}$ 
give that for any $\eta \in S^{2,+}$ 
\begin{equation*}
  \bbP^{\eta}(\tau^{2,-}>t)\geq  \bbP^{\eta}(\tau^{1,-}>t)\geq e^{-\gga_1 t}+o(1)\geq e^{-\frac{t}{2\Tr}}+o(1),
\end{equation*}
where the last inequality comes from Lemma \ref{gamma1}.
\qed

\subsection{Proof of Corollary \ref{corol}}

We use the same notation that in the previous proof. We set \\
$\gG:= (S^{2,+}\cup S^{2,-})^c$, and denote the local time spent 
by the Markov chain $(\eta_s)_{s\ge 0}$ in $\gG$ by 

\begin{equation}
 H_t=H_t(\gG):=\int_0^t \ind_{\eta(s)\in \gG}\dd s.
\end{equation}
Notice that, if $\tilde \pi^{+,2}:=\pi(\cdot \ | \ S_2^+)$, then $\|\tilde \pi^{+,2}-\pi^+\|=o(1)$.
Equation \eqref{PTT2} implies that there exists a sequence $\gd_L$ going to zero such that
\begin{equation*}
\bbP^{\tilde \pi^{+,2}}(\eta_s\in \gG)\le \gd_L, \quad \forall s\ge 0.
\end{equation*}
Integrating the first equation between zero and $\gd_L^{-1/2}\Tr $ one gets
\begin{equation*}
\bbE^{\tilde \pi^{+,2}}\left[H_{ \gd_L^{-1/2}\Tr}\right]\le \gd_L^{1/2}\Tr\,.
\end{equation*}
We set
\begin{equation*}
 S^{3,+}:=\left\{ \eta \in S^{2,+}, 
\bbE^{\eta}\left[ H_{\gd_L^{-1/2}\Tr }\right]\le \gd_L^{1/4}\Tr \right\}.
\end{equation*}
Using Markov's inequality we obtain
\begin{equation*}
 \bbP^{\eta}\left[  H_{ \gd_L^{-1/2}\Tr}\ge \gd_L^{1/8} \Tr \right]\le \gd_L^{1/8}, \quad \forall \eta\in S^{3,+}\, .
\end{equation*}
Moreover, by Theorem \ref{th3}, there exists a sequence $\gd'_L$ going to zero such that 
\begin{equation*}
 \bbP^{\eta}\left[\tau^{2,-}\in [0,\gd_L^{1/16}\Tr]\cup [\gd_L^{-1/2}\Tr ,\infty) \right]\le \gd'_L, \quad \forall \eta\in S^{2,+}.
\end{equation*}
On the event $\{\tau^{2,-}\in[\gd_L^{1/16}\Tr\cup \gd_L^{-1/2}\Tr ]\}$, one has
$H_{\tau^{2,-}}\le H_{ \gd_L^{-1/2}\Tr}$ and hence, for every $\eta\in S^{3,+}$,
\begin{equation*}
 \bbP^{\eta}\left[H_{\tau^{2,-}}\ge \gd_L^{1/16}\tau^{2,-}\right]\le  \gd_L^{1/8}+\gd'_L\,.
\end{equation*}
\qed

\section{Mixing time upper bound}\label{cinquemezzi}
In this section we prove Theorem \ref{th1}. Our approach will also yield a proof of Theorem \ref{th2}. 
The main ideas of the proof can be sketched as follows. 

\medskip

\noindent
{\em Step 1}. 
Lemma \ref{lemma:peq} shows that after a burn-in time 
$O(L^{2+\e})$ the distribution $\nu_t^{\wedge}$ has a smooth density w.r.t.\ the equilibrium $\pi$. 
The first step consists in using this fact together with \eqref{trelvar} and Lemma \ref{le:monopm} 
to reduce the mixing time upper bound to a lower bound on the spectral gap of the chain.

\medskip

\noindent
{\em Step 2}. To bound the spectral gap we decompose the polymer configurations using the variables $\si_x=\sign(\eta_x)$ introduced in \eqref{nus}. From the decomposition estimates in Proposition \ref{jerr},
we shall roughly 
obtain that the spectral gap of the chain is bounded below by  $\gap_*\times \gap_+$, 
where $\gap_*$
denotes the spectral gap of the heat bath dynamics for the variables $\si$, while $\gap_+$ stands for 
the spectral gap of the polymer with a wall (i.e.\ the polymer constrained to be non negative). From \cite{CMT1}, we
know that $\gap_+=\O(L^{-2})$.

\medskip

\noindent
{\em Step 3}. 
To prove a lower bound on $\gap_*$ we shall perform a second decomposition, this time by fixing the number of crossings (i.e.\ the number of sign switches) in the configuration $\si$. Another application of the bound from Proposition \ref{jerr} will then show that $\gap_*$ is roughly bounded below by a product of 
two spectral gaps, say $\gap_{*}^{(1)}$ 
and $\gap_{*}^{(2)}$. Here $\gap_{*}^{(1)}=\min_{n}\gap_{*,n}$, where 
$ \gap_{*,n}$ is the spectral gap of the 
dynamics on the variables $\si$ constrained to have $n$ crossings, while $\gap_{*}^{(2)}$ denotes the 
spectral gap of a birth and death chain associated to the number of crossings. 
We establish a lower bound $\gap_{*}^{(2)}=\O(1)$. Moreover, we show that when $n=1$, one has $\gap_{*,1} = \O(L^{-5/2})$. To prove a similar bound for every $n$ we introduce a new dynamics involving a fixed number $n$ of 
crossings: with rate 1, independently, each crossing equilibrates its position between the two neighboring crossings positions. If $\gap^n_{\rm eq}$ denotes the spectral gap of this process, a comparison argument shows
that  $\gap_{*,n}\geq \gap_{*,1} \times  \gap^n_{\rm eq}$.

\medskip

\noindent
{\em Step 4}. 
The final step consists in obtaining the lower bound $\gap^n_{\rm eq}=\O(L^{-\e})$. 
The first observation is that if $n\leq  \e\log L$ this estimate can be obtained by means of a direct coupling argument. 
The proof of the estimate for larger values of $n$ is based on a block dynamics argument which allows us to reduce the problem to the case of $n\leq  \e\log L$ crossings. 
The analysis of the block dynamics uses a further coupling argument.
It is worth observing that the coupling arguments used here make crucial use of the heavy tailed nature of the distribution of excursions at equilibrium; see Lemma \ref{equo}. 
%Once the system has been reduced to $n_0$ crossings, the desired conclusion is shown to follow from yet another coupling argument. 

\medskip
Before starting the actual proof, let us pause for a few remarks. The lower bound on $\gap_*$ described in Step 3 and Step 4 above is sharp (up to $O(L^\e)$ corrections). As detailed in Section~\ref{proofth2} below, 
Step 3 and Step 4  will essentially prove Theorem \ref{th2}. On the other hand, 
the final bound 
$\tmix=O(L^{5/2 + 2 + \e})$ for Theorem \ref{th1} is likely to be off by a factor $O(L^2)$. 
As explained in Step 1 above, this comes from the use of a
decomposition estimate that involves the product $\gap_*\times \gap_+$
rather than the minimum $\min\{\gap_*,\gap_+\}$, as it would be the
case if one could  efficiently decouple the mode associated to the
variables $\si$  from the rest.
% from that of the variables describing the polymer within each excursion. 

The following four subsections will develop 
the four steps described above in the given order. 
However, we warn the reader that, because of various technical obstacles, 
the above plan will not be followed very strictly and several detours will be needed.

\subsection{Reduction to spectral gap}
We start with the implementation of Step 1. For later purposes it is necessary to consider a variant of the original dynamics which avoids (very unlikely) configurations with too many crossings or 
too many zeros between consecutive crossings. 
 
Call $\chi$ the number of crossings in a configuration $\eta$:
\begin{equation}\label{chi}
\chi(\eta) = \sum_{x=-L+2}^{L-2}\ind(\eta_{x}=0\,,\;\eta_{x-1}\neq \eta_{x+1})\,.
\end{equation}
Note that only sites $x\in E_L\setminus (\{-L\}\cup\{L\})$ 
%interior even sites from the interval $\{-L,\dots,L\}$ 
appear in the summation. 
Define $\g_x = \ind(\eta_{x}=0\,,\;\eta_{x-1}\neq \eta_{x+1})$,  so that $\chi = \sum_{x=-L+2}^{L-2}\g_x$, and write %whenever $\chi>0$, 
\begin{equation}\label{csi}
\xi_0 = -L\,,\quad \xi_{\chi+1} = L\,,
\end{equation}
 and if $1\leq j\leq\chi$, let $\xi_j$ denote the position in $\{-L+2,\dots,L-2\}$ of the $j$-th ``1'' in the sequence 
 $\{\g_{-L+2},\dots,\g_{L-2}\}$. Thus, $\xi_1,\dots,\xi_\chi$ denote the positions of 
the internal crossings.  Finally, denote by $N(\xi_i,\xi_{i+1})$ the number of zeros in the path $\eta$ strictly between $\xi_i$ and 
 $\xi_{i+1}$. See Figure \ref{fig:definition}.
 
\begin{figure}[htp]
\begin{center}
\includegraphics[width=0.9\textwidth]{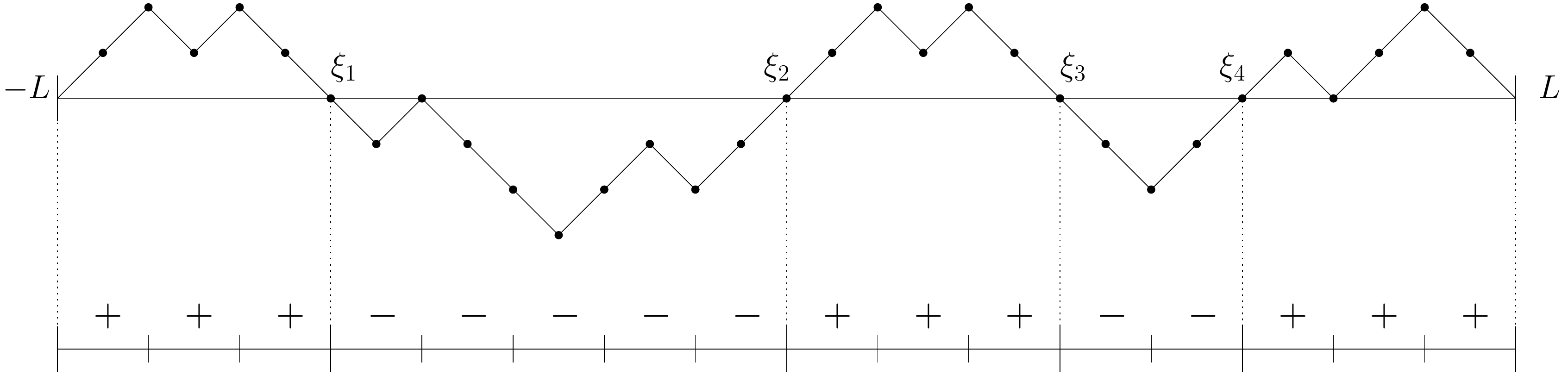}
\end{center}
\caption{A configuration $\eta$ of the polymer with $\chi(\eta)=4$ internal crossings in positions $\xi_1,\xi_2,\xi_3,\xi_4$. Note that in this case $N(\xi_0,\xi_1)=N(\xi_2,\xi_3)=N(\xi_3,\xi_4)=0$, and $N(\xi_1,\xi_2)=N(\xi_4,\xi_5)=1$. Below, the corresponding configuration of signs  $\si=\sign(\eta)$. 
}
\label{fig:definition} 
\end{figure}

Fix a constant $c_o>0$ and define the event
 \begin{equation}\label{oo}
\O^o = \Big\{\eta\in\O\,:\;\chi(\eta)\leq c_o\log L\,,\;\text{and}\;\max_{i=0,\dots,\chi}N(\xi_i,\xi_{i+1})\leq c_o\log L\Big\}\,.
\end{equation}
Clearly $\eta\in(\O^o)^c$ implies that there are at least $c_o\log L$ zeros in the path $\eta$. 
Since $\l<1$, the  bound of Lemma \ref{equo} shows that the number of zeros is exponentially integrable at equilibrium. Therefore, for any $p>0$, taking $c_o(p)$ large enough, we have
 \begin{equation}\label{ooo}
\pi(\O^o) = 1 + O(L^{-p})\,.
\end{equation}
The reason for introducing the restricted set $\O^o$ will be apparent in the sequel. For the moment, we point out that the restriction $\chi\leq c_o\log L$ is essential for our estimates in Section \ref{essent}, while both restrictions $\chi\leq c_o\log L$ and $N(\xi_i,\xi_{i+1})\leq c_o\log L$ will be needed in the estimate of the parameter $\g$ appearing in the decomposition of the spectral gap; see the proof of Proposition~\ref{jerr1} and Proposition~\ref{gapes} below. 

Next, consider the polymer process restricted to stay in the set $\O^o$, i.e.\ the continuous time Markov chain with state space $\O^o$, and generator \eqref{genero}, where the rates $r_{x,\pm}(\eta)$ are replaced by
\begin{equation}\label{rateso}
r^o_{x,\pm}(\eta) = r_{x,\pm}(\eta)\,\ind(\eta^{x,\pm}\,\in\,\O^o)\,.
\end{equation}
 Let also $\pi^o$ denote its reversible invariant measure, which is easily seen to coincide with $ \pi(\cdot\tc\O^o)$.
Let $\trel^o$ denote the relaxation time of the process defined above. 

\begin{lemma}\label{trelo}
Assume that $\trel^o=O(L^p)$ for some $p>0$. There exists $c>0$, such that for any $\e>0$ and for all $L\geq L_0(\e)$:
 \begin{equation}\label{trelbo}
\tmix\leq L^{2+\e} + c\,\trel^o\,\log L\,.
 \end{equation}
\end{lemma}
\proof
Let $t=L^{2+\e/2}+s$, with $s=c_1\trel^o$ for some $c_1$ to be fixed below. 
We prove that 
 \begin{equation}\label{trelo1}
 \|\nu_t^\wedge-\nu_t^\vee\|\leq 1/2\,.
 \end{equation}
From Lemma \ref{le:monopm} and Lemma \ref{lemma:hub}, \eqref{trelo1} implies 
that $\tmix\leq c_2(L^{2+\e/2}+s)\log L$, for some other constant $c_2$, which implies the lemma.
To prove \eqref{trelo1}, we first introduce some notation.
We write $\nu_u^{\eta,o}$ for the distribution at time $u$ of the state of the Markov chain
restricted to $\O^o$, when the initial configuration is some $\eta\in \O^o$.
If $\eta\notin\O^o$ we define $\nu_u^{\eta,o} = \d_{\eta}$ for all $u$. 
Next, we write $\nu_u^{\mu,o} = \sum_{\eta}\mu(\eta)\nu_u^{\eta,o}$ for a probability measure 
$\mu$ on $\O$. 
Using symmetry we can write
\begin{equation}\label{trelo2}
 \|\nu_t^\wedge-\nu_t^\vee\|\leq  2\|\nu_t^\wedge-\nu_s^{\pi^+}\| + 
 2\|\nu_s^{\pi^+} - \nu_s^{\pi^+,o}\|+ \|\nu_s^{\pi^+,o} - \nu_s^{\pi^-,o}\|\,.
 \end{equation}
 We start with the observation that
$$
 \|\nu_t^\wedge-\nu_s^{\pi^+}\| = \|\nu_s^{\nu_{t-s}^\wedge}-\nu_s^{\pi^+}\|
 \leq \|\nu_{t-s}^\wedge-\pi^+\| =o(1)\,,
$$
where the first bound is obtained by writing 
$\nu_s^{\nu_{t-s}^\wedge}-\nu_s^{\pi^+}=
\int(\nu_s^{\eta}-\nu_s^{\eta'})\r(\eta,\eta')$ with $\r$ the maximal coupling of $\nu_{t-s}^\wedge$ and $\pi^+$, 
and the last bound follows from Lemma \ref{lemma:peq}. 

We turn to the last term in \eqref{trelo2}. Let $\pi^{+,o}=\pi^+(\cdot\tc\O^o)$ and observe that
$$
\|\pi^+-\pi^{+,o}\| = o(1)\,.
$$ 
This last bound follows easily from \eqref{ooo} and \eqref{pioplusminus}. Moreover, 
the bound \eqref{trelvar} applied to the process restricted to $\O^o$
yields $$\|\nu_u^{\pi^{+,o},o} - \pi^{o}\|\leq c\,e^{-u/\trel^o}\,,$$
for some $c>0$ and for all $u>0$. 
Therefore, the third term in \eqref{trelo2} can be made smaller than, say, $1/4$ by taking 
$c_1$ large enough in the definition of the time $s$.  
It remains to prove that the second term in \eqref{trelo2} is $o(1)$. 
Since the initial condition is sampled from the same distribution $\pi^+$ we can couple the two 
processes $(\nu_u^{\pi^+})_{u\geq 0}$ and $(\nu_u^{\pi^+,o})_{u\geq 0}$ in such a way that they coincide until the first time %$u\geq 0$ 
when the unrestricted process exits from the set $\O^o$. 
(Note that this time can be zero.) Thus, %it suffices to bound the probability that $v\leq s$:
$$
\|\nu_s^{\pi^+} - \nu_s^{\pi^+,o}\| \leq \sum_{\eta_0\in\O}\pi^+(\eta_0)\bbP(\exists \,u\leq s\,:\; v_u^{\eta_0}\notin\O^o)\,.
$$
From \eqref{pioplusminus} we know that 
$\pi^+(\eta_0)\leq (2+o(1))\pi(\eta_0)$, so that the time-invariance with a union bound implies 
$$
\|\nu_s^{\pi^+} - \nu_s^{\pi^+,o}\| \leq \,2sL(1-\pi(\O^o))(2+o(1))\,,
$$
where we use the fact that the average number of updates up to time $s$ is bounded by $2sL$.
Since, by assumption, $s=O(L^p)$ for some $p$, we can use \eqref{ooo} to conclude
that $\|\nu_s^{\pi^+} - \nu_s^{\pi^+,o}\|=o(1)$. 
\qed
 
 \bigskip
 
 The following three subsections will focus on the upper bound $\trel^o=O(L^{5/2+2+\e})$. 
 Once this bound is established, Theorem \ref{th1} will follow immediately from Lemma \ref{trelo}.
 
\bigskip

\subsection{Decomposing along crossings configurations}\label{jerrum1}
Recall the definition of the variables $\si\in\{-1,+1\}^{O_L}$  
given by $\si_y=\sign(\eta_y)$, where $O_L$ is the set of sites in $\{-L,\dots,L\}$ with the same parity as $L+1$.  Note that the field 
 $\si$ specifies uniquely the field $\xi$ defined after \eqref{csi}, while $\xi$ specifies $\si$ up to a global sign switch; see Figure \ref{fig:definition}. 
The space $\O^o$ can be decomposed into disjoint
subspaces 
\begin{equation}\label{decompo}
\O^o = \cup_{\si}\O^o_\si\,,
\end{equation}
where $\O^o_\si$ denotes the set of $\eta\in\O^o$ such that $\sign(\eta_x)=\si_x$ for all $x\in O_L$.
Let $\cS_o$ denote the set of all $\si\in\{-1,+1\}^{O_L}$ such that $\O^o_\si\neq\emptyset$, i.e.\ the set of $\si\in\{-1,+1\}^{O_L}$ such that $\chi(\si)\leq c_o\log L$, where $\chi(\cdot)$, defined in \eqref{chi},
is seen as a function of $\si=\sign(\eta)$.
%there exists $\eta\in\O^o$ with $\si_x=\sign(\eta_x)$, $x\in O_L$.
Consider the continuous time Markov chain on ${\cS_o}$ with infinitesimal generator
\begin{equation}\label{generos}
\cG \varphi(\si) = 
\sum_{x\in O_L}\theta_{x}(\si)\left[\varphi(\si^{x})-\varphi(\si)\right] 
%\sum_{x\in V_L}r_{x,-}(\si)\left[\varphi(\si^{x,-})-\varphi(\si)\right] \,,
\end{equation}
where $\varphi:{\cS_o}\mapsto\bbR$,  $\si^{x}$ is the configuration $\si$ flipped at $x$, i.e.\ it 
is defined as $\si$ everywhere except at $x$ where it equals $-\si_x$, and the rates $\theta_{x}(\si)$ are given by
\begin{equation}\label{ratesr}
\theta_x(\si)=\sum_{\eta\in\O^o_\si}\pi(\eta\tc\O^o_\si)\big[r^o_{x,+}(\eta)\,
\ind(\si(\eta^{x,+})=\si^x) + r^o_{x,-}(\eta)\,\ind(\si(\eta^{x,-})=\si^x)\big]\,.
\end{equation}
The rates $r^o_{x,\pm}$ are given in \eqref{rateso}. Note that the measure 
\begin{equation}\label{nuso}
\nu_o(\si)=\sum_{\eta\in\O^o_\si}\pi(\eta\tc\O^o)\,,
\end{equation}
is the reversible distribution, i.e.\ $\nu_o(\si)\theta_x(\si)=\nu_o(\si^x)\theta_x(\si^x)$ holds for all $x\in O_L$ and $\si\in\cS_o$. 

In words, the process with generator $\cG$ is described as follows. 
Attach independent rate $1$ Poisson clocks to all sites $x\in O_L$. Let $\si$ be the current configuration. When site $x$ rings, choose a configuration $\eta$ sampled from the distribution $\pi(\cdot\tc\O^o_\si)$ and set $\eta'=\eta^{x,+}$ with probability $r^o_{x,+}(\eta)$,  $\eta'=\eta^{x,-}$ with probability $r^o_{x,-}(\eta)$, and $\eta'=\eta$ with probability $1-r^o_{x,+}(\eta) -r^o_{x,-}(\eta)$.
Finally, update $\si$ to $\si'$ given by $\si'_x=\sign(\eta'_x)$. 
Let $\gap^{\cS_o}$ denote the spectral gap of this Markov chain. 
\begin{proposition}\label{jerr1}
%For every $\e>0$, 
There exists $c>0$ such that for all $L$:
$$\trel^o\leq c\,L^{2}(\log L)^3\,(\gap^{\cS_o})^{-1}\,.$$
\end{proposition} 
\proof
We apply Proposition \ref{jerr} with the decomposition \eqref{decompo}. 
To each $\si$ we can associate the continuous time Markov chain 
with state space $\O^o_\si$, defined by the generator \eqref{genero} with 
the rates $r_{x,+}(\eta)$ replaced by 
 \begin{equation}\label{ratesi}
r^\si_{x,\pm}(\eta) = r^o_{x,\pm}(\eta)\,\ind(\eta^{x,\pm}\,\in\,\O^o_\si)\,,
\end{equation}
where the rates $r^o_{x,\pm}(\eta)$ are defined in \eqref{rateso}, with reversible 
equilibrium measure $\pi^o(\cdot\tc\O^o_\si)$. Call $\gap^\si$ the spectral gap of this Markov chain. 
For a given $\si$, 
this corresponds to independent continuous time Markov chains 
for each interval $\{\xi_i,\dots,\xi_{i+1}\}$, where the crossing positions $\xi_i$ have been
 defined in \eqref{csi}.  On a given interval $\{\xi_i,\dots,\xi_{i+1}\}$,  we have a polymer dynamics with a horizontal wall constraint (polymer above or below the wall depending on the sign of the field 
 $\si$ inside that interval). Moreover, within each interval the polymer is constrained to have smaller than $c_o\log L$ zeros. Let $\gap^{\si,i}$ denote the spectral gap of this process on the interval
 $\{\xi_i,\dots,\xi_{i+1}\}$. From the independence recalled above, one has
 \begin{equation}\label{gapsi}
\gap^\si = \min_{i=0,\dots,\chi}\gap^{\si,i}\,,
\end{equation}
where $\chi$ is the number of interior crossings defined in \eqref{chi}. It follows from Lemma \ref{wallc}
below that $(\gap^\si)^{-1} = O(L^{2}\log L)$, uniformly in $\si\in{\cS_o}$. 

Next, observe that the generator $\cG$ in \eqref{generos}
coincides with the generator $\wb \cL$ from Proposition~\ref{jerr} 
for the present choice of the decomposition. 
We can then use the bound of Proposition~\ref{jerr} with
 $\bar\l=\gap^{\cS_o}$, and $\l_{{\rm min}}=\min_{\si\in{\cS_o}}\gap^\si$: 
 \begin{equation}\label{jerr111}
\trel^o\leq c\,(\gap^{\cS_o})^{-1}\max\{1,L^{2}\log L(\gap^{\cS_o} + \g)\}\,,
\end{equation}
where $c>0$ is a constant, and 
\begin{equation}\label{gammas}
\g=\max_{\si\in{\cS_o}}\max_{\eta\in\O^o_\si}\sum_{x=-L+1}^{L-1}
\big[r^o_{x,+}(\eta)\,
\ind(\si(\eta^{x,+})=\si^x) + r^o_{x,-}(\eta)\,\ind(\si(\eta^{x,-})=\si^x)\big]\,.
\end{equation}
It is immediate to check that 
$\gap^{\cS_o}$ is smaller than a constant, so that $ \trel^o\leq c\,L^{2}\log L\,(1+\g)\,(\gap^{\cS_o})^{-1}$, by \eqref{jerr111}. It remains to give an upper bound 
 $\g=O((\log L)^2)$. %This is where the condition $N(\xi_i,\xi_{i+1})\leq c_o\log L$ appearing in the definition \eqref{oo} of $\O^o$ plays an important role.  (The constraint $\chi\leq c_o\log L$ will also be important later on). 
 From the definition \eqref{gammas}, 
 we see that $\g$
 is bounded above by the maximum over $\eta\in\O^o$ of the 
 number of sites $x\in O_L$ such that $\eta_{x-1}=\eta_{x+1}=0$. By definition of the set $\O^o$, the latter quantity is bounded by $(c_o\log L)^2$ in our setting. This ends the proof.
 \qed
 
 \bigskip
  
We turn to the lower bound on the gaps defined in \eqref{gapsi}, that was needed in the proof of Proposition \ref{jerr1}. The bound $(\gap^\si)^{-1}=O(L^2\log L)$, uniform over $\si\in{\cS_o}$, is an immediate consequence of Lemma \ref{wallc} below and \eqref{trelmix}. Consider the polymer dynamics under the constraints $N(\eta)\leq M$ (where $N(\eta)$ is the total number of zeros and $M$ is a positive constant) and $\eta\geq 0$, i.e.\ let $\G_{+,M}$
denote the set 
$$\G_{+,M}=\big\{\eta\in\O\,:\;N(\eta)\leq M\,,\;\text{and}\;\eta_x\geq 0
\;\text{for all } x%\in\{-L,\dots,L\}
\big\}\,,$$ 
and write 
 $\r^{\eta}_t$ for the law at time $t$ of the corresponding process. This is the continuous time Markov chain with generator \eqref{genero} with rates replaced by $r_{x,\pm}(\eta)\ind(\eta^{x,\pm}\in\G_{+,M})$ and with reversible measure $\pi(\cdot\tc\G_{+,M})$.  
 %If we had no constraint on the number of zeros, 
 %\cite[Theorem 3.1]{CMT1} would give us a $O(L^2\log L)$ upper bound on the mixing time.
 
  \begin{lemma}\label{wallc}
  There exists a constant $c>0$
such that for any $\eta\in\G_{+,M}$ and for any $M\geq c\,\log L$:
$$
\|\r^{\eta}_T - \pi(\cdot\tc\G_{+,M})\| = o(1)\,,\quad T=c\,L^2\log L\,.
$$  
  \end{lemma}
  \proof
  Let $\mu^{\eta}_t$ denote the evolution without the constraint $N(\eta)\leq M$, i.e.\
  the Markov chain with generator \eqref{genero} with rates given by $r_{x,\pm}(\eta)\ind(\eta^{x,\pm}\geq 0)$ and with reversible measure 
  $\pi^w=\pi(\cdot\tc \eta\geq 0)$. The mixing time of this ``wall'' constrained model 
  has been analyzed in \cite[Theorem~3.1]{CMT1}, where it is shown that, for some constant $c>0$, for all initial $\eta\geq 0$: 
  \begin{equation}\label{wiwall}
  \|\mu^{\eta}_t - \pi^w\|\leq c\,L^2\,\exp\Big(-\frac{t}{c\,L^2}\Big)\,.
  \end{equation}
As in \eqref{ooo}, standard equilibrium estimates imply that for any $p>0$ one can choose $c>0$ such that for all $M\geq c\,\log L$:
\begin{equation}\label{wiwall3}
\pi^w(N>M) = \|\pi^w -  \pi(\cdot\tc\G_{+,M})\|=O(L^{-p})\,.
\end{equation}
Next, observe that by monotonicity $\mu^{\eta}_t$ is stochastically dominated by $\r^{\eta}_t$.
  Let $\t^\eta$ denote the hitting time of the set $N(\cdot)>M$ for the process 
 with law  $(\mu^{\eta}_t)_{t\geq 0}$, and introduce the event
\begin{equation*}
 G_t:=\{\eta\in\Omega,\ \bbP(\t^\eta \leq t)\leq L^{-1}\}\,.
\end{equation*}
Note that, for fixed $t$, $G_t$ is an increasing event, and therefore 
$\mu_s^{\eta}(G_t) \leq \r_s^{\eta}(G_t)$, for any $s,t\geq 0$.

Let us fix now $t=c_1\,L^2\log L$.
For any $p>0$, using the time-invariance and a union bound with 
the fact that the expected number of updates up to time $t$ is bounded by $2tL$, one has
\begin{equation*}\label{wiwall4}
\sum_{\eta}\pi^w(\eta)\bbP\big(\t^\eta\leq t\big) \leq 2tL\,\pi^w(N>M)
 = O(L^{-p})\,,
\end{equation*}
where we use \eqref{wiwall3}.
Markov's inequality then implies that %for any $p>0$, for suitably large $c_o$:
\begin{equation*}
 \pi^w(G_t^c) = O(L^{1-p})\,.
\end{equation*}
From \eqref{wiwall} we then deduce that if $c_1$ is sufficiently large (in the definition of $t$):
\begin{equation}\label{wiwall5}
 \r_{t}^{\eta}(G_t^c)\leq  \mu_t^{\eta}(G_t^c)  =O(L^{1-p})\,.
\end{equation}
On the other hand, writing $\r^{\eta}_{2t}(A) - \pi^w(A) = \sum_{\eta'}\r^{\eta}_t(\eta')(\r^{\eta'}_t(A)-\pi^w(A))$ for any set $A$ one has 
$$
 \|\r^{\eta}_{2t} - \pi^w\| \leq \sum_{\eta'\in G_t}
 \r^{\eta}_t(\eta')
\|\r^{\eta'}_t-\pi^w\| + \r^{\eta}_t(G_t^c)\,.
$$
To estimate the first term above, note that the processes with laws $(\mu^{\eta}_t)_{t\geq 0}, (\r^{\eta}_t)_{t\geq 0}$
can be coupled in such a way that they coincide until time $\t^\eta$. Therefore, 
by definition of $G_t$, for any $\eta'\in G_t$:
\begin{equation}\label{wiwall6}
\|\r^{\eta'}_t - \pi^w\|\leq \bbP(\t^{\eta'}\leq t) + \|\mu^{\eta'}_t - \pi^w\| = O(L^{-1})\,,
\end{equation}
where we have used again \eqref{wiwall} to bound the last term above. 
In conclusion, using
   \begin{equation*}\label{wiwall2}
  \|\r^{\eta}_{2t} - \pi(\cdot\tc\G_{+,M})\|\leq \|\pi^w -  \pi(\cdot\tc\G_{+,M})\| +  \|\r^{\eta}_{2t} - \pi^w\|\,,
    \end{equation*}
together with \eqref{wiwall3}, \eqref{wiwall5}, and \eqref{wiwall6}, we arrive at $\|\r^{\eta}_{2t} - \pi(\cdot\tc\G_{+,M})\| = O(L^{-1})$ which implies the desired estimate. 
\qed

\bigskip
Thanks to Lemma \ref{trelo} and Proposition \ref{jerr1}, Theorem \ref{th1} will follow from 
the estimate
\begin{proposition}\label{gapes}
For any $\e>0$ and for all $L\geq L_0(\e)$:
 \begin{equation}\label{orelbo}
\gap^{\cS_o}\geq L^{-\frac52-\e} \,.
 \end{equation}
\end{proposition}
The following two subsections are devoted to the proof of Proposition \ref{gapes}.
\subsection{Decomposing according to  the number of crossings}\label{jerrum2}
We first decompose ${\cS_o}$ according to whether the first excursion has positive or negative sign, i.e.\
 \begin{equation}\label{decomposa}
{\cS_o} = {\cS}^+ \cup {\cS}^-\,,\qquad {\cS}^\pm=\{\si\in{\cS_o}\,:\; \si_{-L+1} = \pm 1\}\,.
\end{equation}
An application of Proposition \ref{jerr} with the decomposition \eqref{decomposa}
yields 
 \begin{equation}\label{jerr02}
\gap^{\cS_o} \geq \min\{\bar\l/3,\bar\l\l_{\min}/(\bar\l+3\g)\}\,,
 \end{equation}
where $\l_{\min}=\gap^{{\cS}^+}$ denotes the gap of the process restricted to ${\cS}^+$ (by symmetry, this equals   the gap of the process restricted to ${\cS}^-$), while $\bar\l$ is the gap of the symmetric two state Markov chain with transition rate 
$$c(+,-)=c(-,+) = \sum_{\si\in{\cS}^+}\nu_o(\si\tc{\cS}^+)\theta_{-L+1}(\si)\,,
\quad \g = \max_{\si\in{\cS}^+}\theta_{-L+1}(\si)\,,$$
$\theta_{-L+1}(\si)$ being the probability of a flip at $x=-L+1$ in $\si$; see \eqref{ratesr}.
The probability measure $\nu_o$ is defined in \eqref{nuso}.
Clearly, $\g\leq 1$. Since at equilibrium the polymer has a uniformly positive probability 
of taking the value $\eta_{-L+2}=0$, cf.\ \eqref{32}, it follows easily that $c(+,-)\geq c(\l)>0$ for some $L$-independent constant $c(\l)$. Thus, $ \bar\l$ is of order $1$, and for some $c=c(\l)>0$ one obtains
  \begin{equation}\label{jerr2}
\gap^{\cS_o}\geq c\,\gap^{{\cS}^+}\,.
 \end{equation}
Next, consider the number of crossings $\chi$ defined in \eqref{chi}. Since $\chi$ is a function of the signs 
$\si$ only, %so that with slight abuse of notation 
we can write $\chi(\si)$ for the number of crossings in a given $\si\in{\cS_o}$. 
Thus, the space ${\cS}^+$ is partitioned as
\begin{equation}\label{decomposp}
{\cS}^+ = \cup_{n=0}^m\cS^{+,n}\,,
\end{equation}
where $m=\lfloor c_o\log L\rfloor$, and $\cS^{+,n} = \{\si\in\cS^+\,:\;\chi(\si)=n\}$. 
We apply Proposition \ref{jerr}, this time with the decomposition \eqref{decomposp}. 
Thus, $\gap^{\cS^+}$ can be bounded below as in \eqref{jerr02} where $\l_{\min}= \min_{0\leq n\leq m}\gap^{+,n}$ now stands for the minimum of gaps of the process restricted to $\cS^{+,n}$, while  
 $\bar \l$ stands for the gap of the %nearest neighbor 
 random walk on $\{0,\dots,m\}$ with transition rates
\begin{gather}\label{cnpm}
c(n,n\pm 1) = \sum_{\si\in\cS^{+,n}}\sum_{x\in O_L}\nu_o(\si\tc\cS^{+,n}) \theta_x(\si)\ind(\si^x\in\cS^{+,n\pm 1})\,,\\
c(n,n\pm 2) = 
\sum_{\si\in\cS^{+,n}}\sum_{x\in O_L}\nu_o(\si\tc\cS^{+,n}) \theta_x(\si)\ind(\si^x\in\cS^{+,n\pm 2})\,.
\label{cnpm2}
\end{gather} 
Note that the transition $n\to n+1$ is necessarily obtained by creating a new crossing 
at the rightmost position $\xi_{n} = L-2$, while the transition $n\to n+2$ can be obtained by creating two adjacent 
crossings between two existing consecutive crossings.  These are the only transitions that are induced by the single flips $\si\to\si^x$. 
By construction, the reversible invariant measure of the chain defined by the rates \eqref{cnpm}, \eqref{cnpm2} 
%projected chain with gap $\bar \l$ 
is given by $$
\mu(n):=\nu_o(\chi=n\tc\cS^+) = \pi(\chi=n\tc \O^o)\,,\qquad n=0,\dots,m\,.$$
\begin{lemma}\label{gapbar}
The gap $\bar\l$ of the chain defined by \eqref{cnpm}, \eqref{cnpm2} satisfies $c\geq\bar\l\geq c^{-1}$ for some $c=c(\l)>0$ uniformly in $L\geq 2$.
\end{lemma}
\proof
By Lemma \ref{equo}, $\pi(\chi=n)$ is exponentially decaying in $n$,
%the delocalized regime $\l<1$, 
and the same applies to $\mu(n)=\pi(\chi=n\tc \O^o)$, up to $O(L^{-p})$ corrections; see \eqref{ooo}. 
To prove an upper bound on $\bar\l$ one can take the test function $\chi$ counting the number of crossings in the variational principle defining $\bar\l$. The variance of $\chi$ w.r.t.\ $\mu(\cdot)$ 
is a positive constant. On the other hand the Dirichlet form can be bounded from above by 
\begin{equation}\label{dirgamma}
4\sum_{n=0}^{m-1}\mu(n)[c(n,n+1)+c(n,n+2)]\,.
\end{equation}
Observe that for every $n$, and every $\si\in\cS^{+,n}$, one has 
\begin{equation}\label{indgamma}
\sum_{x\in O_L}\theta_x(\si)\big[
\ind(\si^x\in\cS^{+,n\pm 1})+\ind(\si^x\in\cS^{+,n\pm 2})\big]
%\ind(\si^x\in\cS^{+,n\pm 1})
=O(n)\,.
\end{equation}
 Indeed, each excursion can only contribute $O(1)$ to this sum since at equilibrium, in the delocalized phase, the expected number of zeros between consecutive crossings is finite (depending on $\l<1$); see  Lemma \ref{equo}. Now, \eqref{indgamma} implies that \eqref{dirgamma} is bounded above by 
 $c\,\sum_{n}\mu(n)\,n = O(1)$. This ends the proof of $\bar\l\leq c$ for some constant $c$.

To prove a lower bound on $\bar \l$ we can neglect the additional rates \eqref{cnpm2}. Let us check that 
the rate $c(n,n+1)$ in \eqref{cnpm} is bounded away from zero uniformly in $0\leq n\leq m-1$, $L\geq 2$ ($m$ being the maximal number of crossings allowed). 
We have
\begin{equation}\label{bacon}
c(n,n+1) \geq c\,\nu_o(\xi_{n}=L-4\tc\cS^{+,n})\,,
%\frac{\mu(n+1)}{\mu(n)}\,c(n+1,n)\,.
\end{equation}
for some constant $c=c(\l)>0$, where $\xi_n$ denotes the rightmost internal crossing.
We introduce some extra notation to characterize more explicitly the measures 
$\nu_{n,o}:=\nu_o(\cdot\tc\cS^{+,n})$. 
%It is convenient to change variables from $\si$ to $\xi_1,\dots,\xi_\chi$; see \eqref{csi}. 
%Let $\nu_{n,o}$ denote the reversible equilibrium measure of the process restricted to $\cS^{+,n}$. To describe $\nu_{n,o}$ explicitly, we introduce some extra notation. 
Let $\O^{+,o}_{j}$ denote the set of polymers 
$$\O^{+,o}_{j} =
\{\eta\geq 0:\;
\eta_{0}=\eta_j=0\,,\;\eta_{x+1}=\eta_x\pm 1\,,\;x=0,\dots,j-1\,, \;\text{and}\;N(\eta)\leq c_o\log L\}\,.$$
Note that these paths start at $0$ and end at $j$, so that $j$ must be even for this set to be non empty. 
Define the probability
$$
\r_{+,o}(j) = \frac{2^{-j}Z_j^{+,o}}{z_{+,o}}\,,\quad\; Z_j^{+,o} = 
\sum_{\eta\in\O^{o,+}_j}\l^{N(\eta)}\,,
$$
where $z_{+,o}$ is the normalization. From Lemma \ref{equo}, \eqref{3/2+} %and \eqref{ooo} 
one has that 
$2^{-j}Z_j^{+,o}=O( j^{-3/2})$ for large $j$, so that the probability $\r_{+,o}$ is well defined. 
Then, it is not hard to check that the measure $\nu_{n,o}$ introduced above is given by
 \begin{equation}\label{nu_no}
\nu_{n,o}(x_1,\dots,x_n)=
\frac{\r_{+,o}(x_1+L)\r_{+,o}(x_2-x_1)\cdots\r_{+,o}(x_n-x_{n-1})\r_{+,o}(L-x_n)}{Z_{n,L}^o}
\end{equation}
where $\nu_{n,o}(x_1,\dots,x_n)=\nu_{n,o}(\xi_1=x_1,\dots,\xi_n=x_n)$, if 
$-L<x_1<\cdots<x_n<L$ is any allowed configuration of the crossing positions $\xi_1,\dots,\xi_n$, 
and $Z_{n,L}^o$ is the normalizing constant. 
In particular, $\nu_{n,o}$ is a product measure $\r_{+,o}^{\otimes (n+1)}$
over $n+1$ positive 
increments $x_{i+1}-x_i$ (with $x_0=-L,x_{n+1}=L$), conditioned to have 
$\sum_{i=0}^n(x_{i+1}-x_i)=2L$. 

Going back to \eqref{bacon},  with the notation in \eqref{nu_no} we have 
$$
\nu_o(\xi_{n}=L-4\tc\cS^{+,n}) = \frac{\r_{+,o}(4)\r_{+,o}^{\otimes(n)}(\sum_{i=1}^{n}(\xi_{i}-\xi_{i-1})=2L-4)}{\r_{+,o}^{\otimes(n+1)}(\sum_{i=1}^{n+1}(\xi_{i}-\xi_{i-1})=2L)}\,.
$$
Since $\sqrt L\gg m\geq n$, and $\r_{+,o}(4)>0$ uniformly, 
with the same arguments of Lemma \ref{lemma:ungros} below
one easily sees that $\nu_o(\xi_{n}=L-4\tc\cS^{+,n}) >0$ uniformly. (We omit the details here to avoid repetitions). 

Once we have that the rate $c(n,n+1)$ in \eqref{cnpm} is bounded away from zero uniformly, 
%in 
%$0\leq n\leq m-1$, $L\geq 2$ ($m$ being the maximal number of crossings allowed), 
a standard bound using the isoperimetric constant and Cheeger's inequality (cf.\ e.g.\ \cite{LS})
implies that $\bar \l$ is bounded away from zero uniformly. 
Alternatively, the same conclusion can be derived from more refined bounds 
for birth and death chains \cite{Miclo}.  
\qed

\medskip

We return to the application of Proposition \ref{jerr} with decomposition \eqref{decomposp}. The constant $\g$ is now given by 
$$
\g=\max_n \max_{\si\in\cS^{+,n}}\sum_{x\in O_L}\theta_x(\si)\big[
\ind(\si^x\in\cS^{+,n\pm 1})+\ind(\si^x\in\cS^{+,n\pm 2})\big]\,.
$$
Recall \eqref{indgamma}.
 Since $n\leq m = O(\log L)$ we obtain that $\g = O(\log L)$. In conclusion, using \eqref{jerr2} and Lemma \ref{gapbar},
 we arrive at
 \begin{equation}\label{jerr3}
\gap^{{\cS_o}} \geq c\,(\log L)^{-1}\min_{1\leq n\leq m}\gap^{+,n}\,.
\end{equation}
Thanks to \eqref{jerr3}, the proof of Proposition \ref{gapes} has been reduced to the proof of a lower bound $\gap^{+,n}=\O(L^{-5/2-\e})$, uniformly in $n\leq c_o\log L$. 
 Note that the case $n=0$ does not appear in \eqref{jerr3}. Indeed, 
 in that case 
  the corresponding equilibrium is concentrated on the single ``all plus'' element $\cS^{+,0}$. We start with the simplest case, i.e.\ $n=1$, which can be analyzed by rather standard arguments. 
 
   \begin{lemma}\label{52}
  There exists a constant $c=c(\l)>0$,
such that 
\begin{equation}\label{521}
\gap^{+,1}\geq c\,L^{-5/2}\,.
\end{equation}
  \end{lemma}
  \proof
  When $n=1$, the process restricted to $\cS^{+,1}$ consists of a nearest neighbor random walk on 
 $E_L$,  the sites in $\{-L+2,\dots,L-2\}$ with the same parity as $L$. Letting $\si(x)$ denote the 
 element of $\cS^{+,1}$ with unique crossing at $x$, the corresponding jump rates are given by
 \begin{equation}\label{522}
  c(x,x\pm 2) = \sum_{y\in O_L} \theta_y(\si(x))\ind(\si(x)^y = \si(x\pm 2))\,.
 \end{equation}  
 These rates are of order $1$ (except at $x=-L+2$ where $c(-L+2,-L)=0$ and at $x=L-2$ where 
 $c(L-2,L)=0$) since there is a uniformly positive probability for the polymer conditioned to have signs $\si(x)$ of having a $0$ at $x\pm2$. Moreover, the reversible invariant probability measure $\r(x)$ for this chain is
  given by 
\begin{equation}\label{5210} 
  \r(x) = \pi^o(\eta_{x-1},\eta_x,\eta_{x+1} = (1,0,-1)\tc\sign(\eta)\in\cS^{+,1})\,.
 \end{equation}
  Recalling \eqref{ooo} and Lemma \ref{equo}, the event $\sign(\eta)\in\cS^{+,1}$ has uniformly positive probability at equilibrium, and $\r(x)\asymp \r_0(x)$, where $\r_0$ is the probability  
 \begin{equation}\label{523}
\r_0(x)=\frac{(L+x)^{-3/2}(L-x)^{-3/2}}{z_L}\,,\quad x\in E_L\,,
%=-L+2,-L+4,\dots,L-4,L-2\,,
\end{equation}
where $z_L$ is the normalizing constant. 
 Since the rates \eqref{522}
 are of order 1, a standard comparison argument shows that it is sufficient to prove 
 the bound  $\gap_0^{+,1}\geq c\,L^{-5/2}$ for the gap of 
 the chain with reversible measure $\r_0$ given by \eqref{523} with rates defined by 
 $c_0(x,x+2)=1$ and $c_0(x,x-2)=\r_0(x-2)/\r_0(x)$. The latter process has the Dirichlet form
\begin{equation}\label{5201}
 \cE_0(f,f)=\sum_{x=-L+2}^{L-4}\r_0(x)[f(x+2)-f(x)]^2\,.
 \end{equation}
 On the other hand, writing $f(y)-f(x)=\sum_{j=x}^{y-2}[f(j+2)-f(j)]$, $x<y$, and using Schwarz' inequality, the variance $\var_{\r_0}(f)$ can be bounded above by:
 \begin{align*}
 \var_{\r_0}(f) &= \frac12\sum_{x,y}\r_0(x)\r_0(y)[f(y)-f(x)]^2\\
 &\leq 2L\sum_{j=-L+2}^{L-4}\r_0(j)[f(j+2)-f(j)]^2
 \sum_{x\leq j}\sum_{y>j}\frac{\r_0(x)\r_0(y)}{\r_0(j)}\,.
 \end{align*}
From \eqref{523} one has 
$$
\max_{j} \sum_{x\leq j}\sum_{y>j}\frac{\r_0(x)\r_0(y)}{\r_0(j)} = O(L^\frac32)\,.
 $$
 It follows that $\gap_0^{+,1}\geq c\,L^{-5/2}$, which implies \eqref{521}. \qed

\medskip
\begin{remark} \rm The bound of Lemma \ref{52} is optimal in the sense that 
\begin{equation}\label{5202}
\gap^{+,1}\leq c^{-1}\,L^{-5/2}\,.
\end{equation}
This can be seen by taking a test function $\varphi(x)$ in the variational principle for $ \gap^{+,1}$, of the form $\varphi(x) = g(x/L)$, where $g:[-1,1]\mapsto\bbR$ is given by 
$g(s)=-1$ if $s<-1/2$, $g(s)=1$ if $s>1/2$, and is linear between $-1/2$ and $1/2$. 
With this choice one has that the variance of $\varphi$ w.r.t.\ $\r$ defined by \eqref{5210} is of order $1$, while the Dirichlet form - given by \eqref{5201} with $\r$ instead of $\r_0$ - is bounded above by 
the probability of $x\in[-L/2,L/2]$ times $L^{-2}$, i.e.\ $L^{-5/2}$. This implies \eqref{5202}.
\end{remark}
\medskip

We turn to the proof of a lower bound on $\gap^{+,n}$ for $n>1$. Recall the definition \eqref{nu_no} of the measure $\nu_{n,o}$.
%To prove a lower bound on $\gap^{+,n}$ for $n>1$ w
We introduce a further dynamics. 
We view the $\xi_i$ as particle positions. Each particle $i=1,\dots,n$ has an independent Poisson clock with parameter 1. When particle $i$ rings, we freeze all positions $\xi_k$, $k\neq i$, and update $\xi_i$ with the new position $\xi'_i$ sampled from the conditional distribution 
$\nu_n^i:=\nu_{n,o}(\cdot\tc\xi_k, \,k\neq i)$. 
The Dirichlet form of this process is given by 
 \begin{equation}\label{dirieqn}
\cE^{n,o}_{\rm eq}(f,f) = \sum_{i=1}^n\nu_{n,o}\left[\var_{\nu_n^i}(f)\right]\,,
\end{equation}
where $\var_{\nu_n^i}(f) = \nu_n^i(f^2) - \nu_n^i(f)^2$ is the variance conditioned on the values of $\xi_k, \,k\neq i$. 
Let $\gap_{\rm eq}^{n,o}$ denote the associated spectral gap:
  \begin{equation}\label{gapeqn}
\gap_{\rm eq}^{n,o} = \inf_{f}\frac{\cE^{n,o}_{\rm eq}(f,f) }{\var_{\nu_{n,o}}(f)}\,,
\end{equation}
where the infimum ranges over all functions of the crossing positions $\xi_1,\dots,\xi_n$. 
The next estimate allows one to reduce the proof of Proposition \ref{gapes} to the proof of a lower bound 
$\gap^{n,o}_{\rm eq}=\O(L^{-\e})$, for every $n\leq c_o\log L$. 
 \begin{lemma}\label{neq}
  There exists a constant $c=c(\l)>0$,
such that for all $n\geq 1$.
\begin{equation}\label{neq1}
\gap^{+,n}\geq c\,L^{-5/2} \gap^{n,o}_{\rm eq}\,.
\end{equation}
  \end{lemma}
  \proof
For $n=1$ this coincides with the result of Lemma \ref{52}, since in this case 
$\nu_{1,o}[\var_{\nu_1^1}(f)]$ coincides with $\var_{\nu_{1,o}}(f)$. 
The general case follows from the observation that Lemma \ref{52} can be applied with 
the segment $\{\xi_{i-1},\dots,\xi_{i+1}\}$ replacing the usual $\{-L,\dots,L\}$, to obtain
$$
\var_{\nu_n^i}(f)\leq \frac1c\,(\xi_{i+1}-\xi_{i-1})^{5/2}\cE^i(f,f)\leq \frac1{c_1}\,L^{5/2}\cE^i(f,f)\,,
$$
where $c_1=c/2^{5/2}$, and $\cE^i(f,f)$ denotes the Dirichlet form of the nearest neighbor random walk on $E_L$, corresponding to the rates
\eqref{522} (with the segment $\{\xi_{i-1},\dots,\xi_{i+1}\}$ in place of $\{-L,\dots,L\}$). 
Taking the $\nu_{n,o}$-expectation, and summing over $i=1,\dots,n$ one obtains the estimate
$$
\cE^{n,o}_{\rm eq}(f,f) \leq \frac1{c_1}\,L^{5/2}\sum_{i=1}^n\nu_{n,o}\left[\cE^i(f,f)\right]\,.$$
The desired conclusion follows from the observation that 
$\sum_{i=1}^n\nu_{n,o}\left[\cE^i(f,f)\right]=\cE^{+,n}(f,f)$, where $ \cE^{+,n}(f,f)$ is the Dirichlet form of the process restricted to $\cS^{+,n}$ with spectral gap
given by $\gap^{+,n}$. \qed

\subsection{Lower bound on $\gap^{n,o}_{\rm eq}$}\label{essent}
To complete the proof we need the lower bound $ \gap^{n,o}_{\rm eq}=\O(L^{-\e})$.
We first remove the restriction $N(\eta)\leq c_o\log L$ in the definition of the measure $\nu_{n,o}$. Namely, introduce the probability measure
\begin{equation}\label{nu_n}
\nu_{n}(x_1,\dots,x_n)=
\frac{\r_{+}(x_1+L)\r_{+}(x_2-x_1)\cdots\r_{+}(x_n-x_{n-1})\r_{+}(L-x_n)}{Z_{n,L}}\,,
\end{equation}
defined by the kernel
\begin{equation}\label{zplus}
\r_+(j)=\frac{2^{-j}Z_j^{+}}{z_+}\,,\quad\; Z_j^{+}=\sum_{\eta\in\O^{+}_j}\l^{N(\eta)}\,,
\end{equation}
where $$
\O^{+}_{j} =\{\eta\geq 0:\;
\eta_{0}=\eta_j=0\,,\;\eta_{x+1}=\eta_x\pm 1\,,\;x=0,\dots,j-1\}\,.$$
Note that \eqref{nu_n} coincides with \eqref{nu_no} except for the removal of the constraint $N(\eta)\leq c_o\log L$ in the definition of the kernel $\r_+$. As in \eqref{ooo} one can check that $\r_+$ is equal to $\r_{+,o}$ up to $O(L^{-p})$ corrections for some large $p>0$ and therefore 
$\nu_n(x_1,\dots,x_n)=\nu_{n,o}(x_1,\dots,x_n)(1+O(L^{-p}))$,  uniformly. 
Proceeding as in \eqref{dirieqn} and \eqref{gapeqn} we define 
 \begin{equation}\label{dirieqn1}
\cE^{n}_{\rm eq}(f,f) = \sum_{i=1}^n\nu_{n}\left[\var_{\nu_n^i}(f)\right]\,,
\end{equation}
and $\gap^n_{\rm eq}$,  the spectral gap associated to the measure $\nu_n$ and the Dirichlet form \eqref{dirieqn1}. From the previous observations we see that $\var_{\nu_n}(f)\leq c\,\var_{\nu_{n,o}}(f)$
and $\cE^{n,o}_{\rm eq}(f,f)\leq c\,\cE^{n}_{\rm eq}(f,f)$, for some constant $c>0$, 
for every function $f$. Therefore $\gap^n_{\rm eq}\leq c^2\,\gap^{n,o}_{\rm eq}$. 
For later purposes it is important to keep track of the dependence on $L$ in our notation, 
and therefore we write $\gap^n_{\rm eq}(L)$ below.
The desired bound $ \gap^{n,o}_{\rm eq}=\O(L^{-\e})$ follows from the comparison mentioned above and 
\begin{proposition}
 \label{lemma:epsilon}
For any $\e>0$, there exists $L_0(\e)>0$ such that for $L\geq L_0$ and $n \leq c_o\log L$
\begin{equation}
\label{bodes}
\gap_{\rm eq}^n(L)\geq L^{-\e}\,.
\end{equation}
\end{proposition}
As a preliminary step towards the proof of Proposition \ref{lemma:epsilon}, we establish the following lemma.
\begin{lemma}
 \label{lemma:epsilon1}
 There exists $c>0$ such that for any $n$, uniformly in $L\geq n$:
\begin{equation*}
\gap_{\rm eq}^n(L)\geq c\,e^{-n/c}\,.
\end{equation*}
\end{lemma}
The crucial point of the above estimate is that it does not depend on the size of the system $L$, 
but only on the number of particles $n$.  Note that Lemma \ref{lemma:epsilon1} gives the required lower bound 
\eqref{bodes}
if we take $n\leq \e_1\log L$ with $\e_1$ suitably small. The case of larger $n$ will be considered afterwards.

\subsubsection{Proof of Lemma \ref{lemma:epsilon1} }
Every particle has an independent Poisson(1) clock. When particle $i$ rings we update its position 
according to the equilibrium measure conditioned on the neighboring positions 
$\xi_{i-1}$ and  $\xi_{i+1}$, that is the probability for the $i$-th particle to be in position 
$\xi_i'=\xi_{i-1}+k$, $k=2,4,\dots,\xi_{i+1}-\xi_{i-1}-2$, is given by 
\begin{equation}\label{arcsine}
\frac{\r_+(k)\r_+(\xi_{i+1}-\xi_{i-1}- k)}{Z}\,,\quad\; \text{where  } \;\;Z=\sum_{k=2}^{\xi_{i+1}-\xi_{i-1}-2}\r_+(k)\r_+(\xi_{i+1}-\xi_{i-1}- k)\,.
\end{equation}
An important property of 
\eqref{arcsine} is that the probability $\a$ of the event $\xi_i'=\xi_{i-1}+2$ (or $\xi_i'=\xi_{i+1}- 2$) is positive, uniformly in 
the length $\xi_{i+1}-\xi_{i-1}$. 
The idea is to use this property to prove that with probability at least $\frac12\alpha^n$, 
%with $\alpha>0$ independent of $L$, 
uniformly in the initial configuration, the process hits the minimal configuration
\begin{equation}\label{minimale}
\xi_1=-L+2,\xi_2=-L+4,\dots,\xi_n=-L+2n\,,
\end{equation} 
before time $n^2$. Once this result is available, one concludes with a simple coupling argument. 
Indeed, using an independent coupling, the total variation distance between two evolutions at time $t=m\times n^2$, with arbitrary  initial conditions, is bounded above by $(1-\a^{2n}/4)^m\leq e^{-m\a^{2n}/4}$. Thus, 
the mixing time of this chain is at most $8n^2\a^{-2n}$. Using e.g.\ \eqref{trelmix} we obtain the lemma. 

To prove the above claim, we use the notation $(\t_m,\ell_m)_{m\geq 1}$ for the sequence of updating marks, that is the $m$-th update occurs at time $\t_m$ and it concerns the $\ell_m$-th particle, 
where $\ell_m$ is a number in $\{1,\cdots,n\}$.
Consider the event $E$ that there exist integers $1\leq \phi(1)<\phi(2)<\cdots<\phi(n)$ such that 
$\t_{\phi(n)}\leq n^2$, 
$\ell_{\phi(i)}=i$, 
and such that for every $i=1,\cdots,n$ and $\phi(i-1)<j\leq \phi(i)$, 
one has $ \ell_j\neq \ell_{\phi(i-1)}$ (where $\phi(0)=0$ and $\ell_0=0$). In words, $E$ is the event that 
within time $n^2$ there has been a sequence of $n$ updates at times $\t_{\phi(1)},\dots,\t_{\phi(n)}$, such that
the update at time $\t_{\phi(i)}$ 
concerned the $i$-th particle and such that the $i$-th particle is not touched again 
before time $\t_{\phi(i+1)}$. Conditioned on the event $E$ one has a probability of at least $\a^n$ of 
hitting the configuration described in \eqref{minimale}. Indeed, at time $\t_{\phi(1)}$ we set $\xi_1=2$ with probability $\a$, at time $\t_{\phi(2)}$ we set $\xi_2=4$ with probability $\a$, and so on. Therefore, to prove the lemma it remains to show that the event $E$ has probability at least $1/2$.  This can be easily seen as follows. Consider the event $F$ that a sequence of integers 
$1\leq \psi(1)<\psi(2)<\cdots<\psi(n)$ exists such that 
$\t_{\psi(n)}\leq n^2$, and
$\ell_{\psi(i)}=i$, for every $i=1,\cdots,n$. Let $\t^{(1)}$ denote the first time particle $1$ is updated and define recursively $\t^{(i)}$ as the first time after $\t^{(i-1)}$ when particle $i$ is updated. Clearly, 
$F=\{\t^{(n)}\leq n^2\}$.  Using $\bbE[\t^{(n)}]=n$, 
by Markov's inequality the probability of $F$ is at least $1-n/n^2\geq 1/2$ for $n\geq 2$.
%
%The event $F$ has at least probability $1/2$, since  
%the expectation of the first time when all particles have been updated 
%a subsequence occurs is $n$, and  
%by Markov's inequality the probability of $F$ is at least $1-n/n^2\geq 1/2$ for $n\geq 2$.  
On the other hand $E=F$, since one can define the sequence $\phi$ from $\psi$ 
by choosing $\phi(n)=\psi(n)$, and for $1\leq i \leq n-1$,
\begin{equation*}
 \phi(n-i)=\max\{j< \psi(n-i+1)\,:\; \ell_j=n-i\}.
\end{equation*}
This ends the proof of Lemma \ref{lemma:epsilon1}. \qed

\subsubsection{Proof of Proposition \ref{lemma:epsilon}}
The proof of Proposition \ref{lemma:epsilon} is based on a block-dynamics argument that allows one to reduce to $n\leq \e_1\log L$ particles, in which case the result will follow from Lemma \ref{lemma:epsilon1}. Fix an integer $K\leq n$. A block is a collection of particles with adjacent labels, and 
our particles  will be partitioned into $\D:=\lfloor n/K\rfloor$ 
non-overlapping blocks, in such a way that the first $\D-1$ blocks
 contain $K$ particles, and the last block contains 
at most $K-1$ particles.
% $n -(\D-1)\lfloor n/\D \rfloor $
% (at most $2K-1$) particles. 
For the sake of simplicity, since it
does not change any of our estimates, 
we will suppose that all blocks have exactly $K$ particles, i.e. $n=K\D$. With this notation, the configuration of the $i$-th block can be described by the variables 
\begin{equation}\label{ithB}
(\xi_{K(i-1)+1},\dots,\xi_{Ki})\,,\qquad i=1,\dots,\D\,.
\end{equation}

As usual, the $\xi_i$ are interpreted as particle positions or crossing positions.
 To define the block-dynamics, we consider independent Poisson(1) clocks on each block, when
one of them rings we put all crossings of the relative block simultaneously at equilibrium 
 conditioned on the position of all crossings belonging to the other blocks.
 That is, if $B_i$ denotes the $i$-th block \eqref{ithB}, and 
$\var_{B_i}$ is  the variance with respect to 
\begin{equation*}
\nu_n(\cdot\tc B_j,\;j\neq i)\,,
\end{equation*}
then the Dirichlet form of the block-dynamics is given by
\begin{equation}\label{bbll}
 \cE_{\rm bl}^n(f,f)=\sum_{i=1}^{\D}\nu_n[\var_{B_i}(f)]\,,
\end{equation}
for any function $f$ of the particle positions.
Call $\gap_{\rm bl}^n$  the associated  spectral gap.
The gap of the original dynamics for a single block $B_i$ is given by $\gap_{\rm eq}^K(\ell)$, where 
$\ell=\xi_{iK+1}-\xi_{(i-1)K}$ is the length of the portion of the system occupied by the $K$ particles in the $i$-th block:
$$
\gap_{\rm eq}^K(\ell)\,\var_{B_i}(f)\leq \nu_n\left[\var_{\nu_n^i}(f)\tc B_j,\;j\neq i\right]\,,
$$
where $\var_{\nu_n^i}(f)$ is as in \eqref{dirieqn1}. Using this estimate in \eqref{bbll}, taking the $\nu_n$-expectation and summing over $i$ one obtains  
%From a well known estimate 
%that controls the gap 
%of the original dynamics in terms of the gap of the block-dynamics (see e.g. \cite[Theorem ??]{MAR_FLOUR}) one has:
\begin{equation}
\label{eq:formule}
\gap^{n}_{\rm eq}(L)\geq \gap_{\rm bl}^n\times
 \min_{\ell} \gap^{K}_{\rm eq}(\ell)\,.
\end{equation}
Note that this is a special case of a well known estimate 
that controls the gap 
of the original dynamics in terms of the gap of the block-dynamics (see e.g.\ \cite{MAR_FLOUR}). 
From Lemma \ref{lemma:epsilon1} we know that $ \gap^{K}_{\rm eq}(\ell)\geq e^{-c\,K}$ uniformly in the length $\ell$. Therefore, Proposition \ref{lemma:epsilon} follows directly from Lemma \ref{lemma:gapbk} below by taking $K=\e_1\log L$, $K\leq n\leq c_o\log L$,  with $\e_1$ sufficiently small (depending on $\e$) and $\D=n/K$ (the case $n\leq \e_1\log L$ being contained in Lemma \ref{lemma:epsilon1}). 

\label{subsec:block}
\begin{lemma}
For all $\D\geq 1$, there exists $c(\D)>0$ such that for all $L$ and  $\D\leq n\leq c_o \log L$, 
\label{lemma:gapbk}
 \begin{equation}\label{5delta}
  \gap_{\rm bl}^n\geq c(\D)\,K^{-4\D}\,,\quad \;K=n/\D\,.
 \end{equation}
\end{lemma}

Before proving Lemma \ref{lemma:gapbk}, we need to establish some preliminary facts. 
Consider $\nu_n^{2L}:=\nu_n$, the equilibrium measure \eqref{nu_n} of the system of $n$ particles over the segment $\{-L,\dots,L\}$. 
Passing to the  increment variables $\z_i:=\xi_i - \xi_{i-1}$, $i=1,\dots,n+1$ one writes
$\nu_n^{2L}$ as the conditional probability $\r_+^{\otimes (n+1)}
\big(\cdot\tc\sum_{i=1}^{n+1}\z_i=2L\big)$. 

\begin{lemma}
\label{lemma:ungros}
For all   $n=o(L^{1/6})$: 
\begin{equation*}
\nu_n^{2L}(\z_1\geq 2L-L^{1/3})=\frac{1}{n+1}(1+o(1))\,.
\end{equation*}
\end{lemma}
\begin{proof}[Proof of Lemma \ref{lemma:ungros}]
Observe that
\begin{align*} 
 \nu_n^{2L}(\z_1\geq 2L-L^{1/3})&=
 %\frac{\sum_{i=2L-L^{1/3}}^{L}Z^+_i \sum_{j_1+\cdots+j_{n}=L-i}Z^+_{j_1}\cdots Z^+_{j_{n}}}{\sum_{j_0+\cdots+j_{n}=L}Z^+_{j_0}\cdots Z^+_{j_{n}}}\\
	%&=
	\frac{\sum_{i\geq 2L-L^{1/3}}\rho_+(i) \rho_+^{\otimes (n)}(\sum_{i=1}^{n}\z_i=2L-i)}
	{\rho_+^{\otimes (n+1)}(\sum_{i=1}^{n+1}\z_i=2L)}\,.
	%&\sim \frac{\sum_{i=L-L^{1/3}}^{L}\rho(i) \rho^{\otimes (n)}(L-i)}{(n+1)\rho(L)}\\
\end{align*}
Recall that $\r_+(j)\sim c_+j^{-3/2}$, for some known constant $c_+>0$; see \eqref{3/2+}. 
In particular, 
\begin{equation}\label{c_+}
\r_+(i)=\r_+(2L) \left(1+O\big(L^{-2/3}\big)\right)\,,\quad  
\;i\in\{2L-L^{1/3},\cdots,2L\}\,.
\end{equation} 
Also, since $2L\gg (n+1)^2$, using  \cite[Theorem A]{DON}, 
one has 
\begin{equation}\label{doney}
\rho_+^{\otimes (n+1)}\Big(\sum_{i=1}^{n+1}\z_i=2L\Big) = (n+1) (1+o(1))\,\rho_+(2L)\,.
\end{equation}
Therefore,
$$
(n+1)\nu_n^{2L}(\z_1\geq 2L-L^{1/3}) = (1+o(1))\,\rho_+^{\otimes (n)}\Big(\sum_{i=1}^{n}\z_i\in[0,L^{1/3})\Big)\,.
$$
It remains to show that $ \rho_+^{\otimes (n)}\big(\sum_{i=1}^{n}\z_i\in[0,L^{1/3})\big)=1+o(1)$. However, this is an immediate consequence of the assumption $n=o(L^{1/6})$
and well known estimates for heavy tailed random variables.
\end{proof}

\medskip

\begin{proof}[Proof of Lemma \ref{lemma:gapbk}]
As a corollary of Lemma \ref{lemma:ungros}, one has that 
for $L$ large enough, uniformly in $2\ell\geq L^{1/3}$ and in $n=o( \ell^{1/6})$, 
\begin{equation}
\label{elluntier}
 \nu_n^{2\ell}(\z_1>2\ell-L^{1/3})\geq \nu_n^{2\ell}(\z_1>2\ell-\ell^{1/3})\geq \frac{1}{2(n+1)}\,.
\end{equation}
We shall use this observation together with a coupling argument. 
Start the dynamics at two arbitrary 
initial configurations $\xi,\xi'$ of $n$ particles. We call
$\xi(t),\xi'(t)$ the corresponding
states of the dynamics at time $t$.  
It will be shown that for a suitable coupling $\bbP$:
\begin{equation}
\label{claimc}
\bbP(\xi(1)=\xi'(1))\geq c(\D)(5(K+1))^{-4\D}\,.
\end{equation}
Once this estimate is available the conclusion follows easily. Indeed, \eqref{claimc} implies that   
at time $T$ the total variation distance between $\xi(T)$ and $\xi'(T)$, 
is bounded above by $$(1-c(\D)(6(K+1))^{-4\D})^T\leq e^{-T\,c(\D)(5(K+1))^{-4\D}}\,,$$ and  
therefore the mixing time of the chain is bounded by $2c(\D)^{-1}(5(K+1))^{4\D}$. Using \eqref{trelmix} we obtain the same bound for the inverse spectral gap. After adjusting the constant $c(\D)$, this proves Lemma \ref{lemma:gapbk}.

\begin{figure}[htp]
\begin{center}
\includegraphics[width=0.9\textwidth]{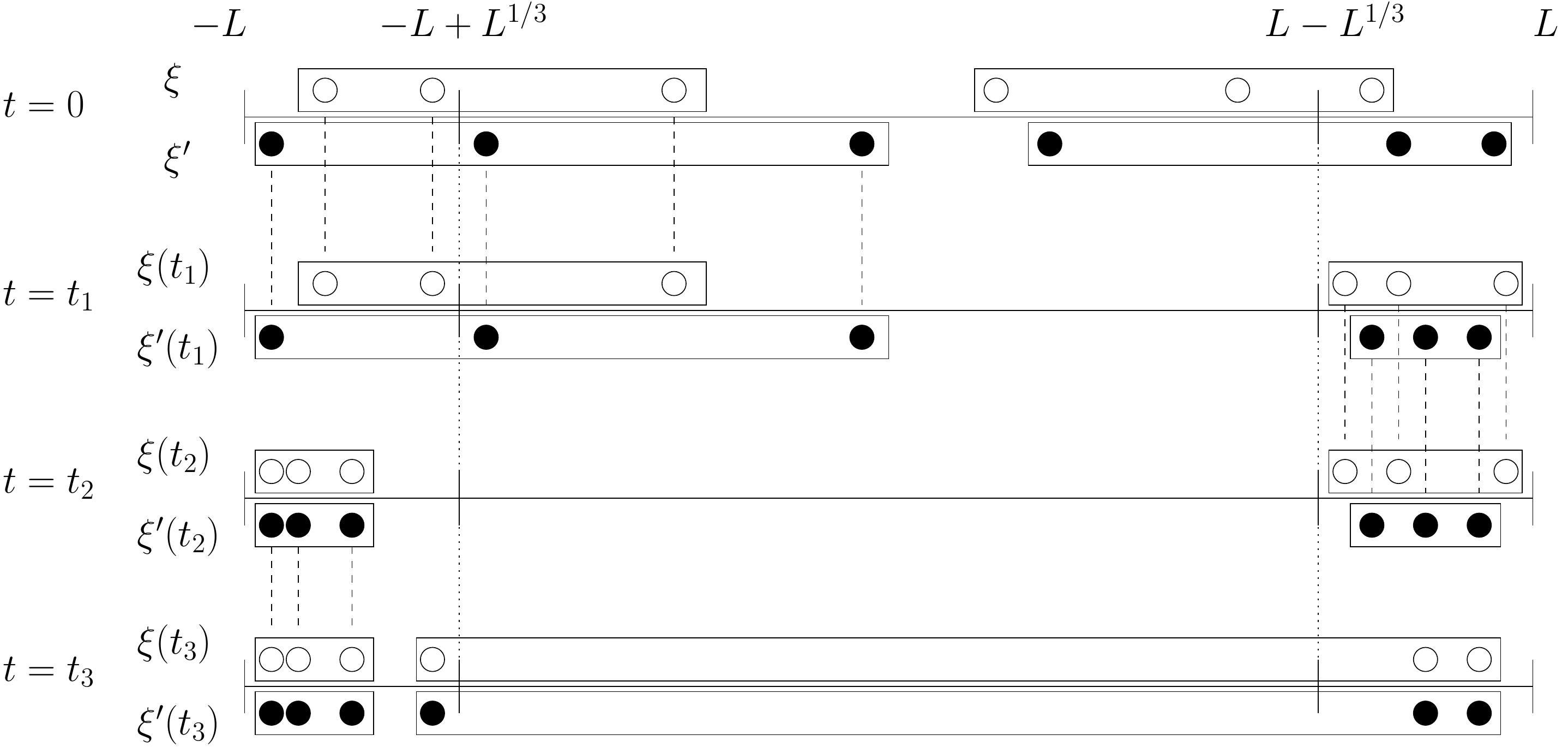}
\end{center}
\caption{Illustration of the coupling used in the proof of Lemma \ref{lemma:gapbk} in the case $\D=2$, $n=6$ and $K=3$. At time $t_1$, all particles in the second block are placed to the right of $L-L^{1/3}$. At time $t_2$, all particles in the first block are matched to the left of $-L+L^{1/3}$. At time $t_3$, all particles in the second block are matched. 
}
\label{fig:coupling} 
\end{figure}

To prove \eqref{claimc} we proceed as follows. 
We have $\D$ independent Poisson(1) clocks at each block, and   we note $(t_m,c_m)_{m\geq 1}$ the sequence of update marks:
for any $m\geq 1$, $t_m$ is the time of the $m$-th update and $c_m\in\{1,\cdots,\D\}$ is the 
label of the block to be updated at time $t_m$. 
Consider the event $E$ that within time $t=1$ a sequence of updates
$t_1,\dots,t_{\D-1},t_\D,t_{\D+1},\dots,t_{2\D-1}$ has occurred, such that $c_m=\D-m+1$ for 
$m=1,\dots \D$, and $c_{\D+i}=i+1$ for $i=1,\dots,\D-1$. 
Clearly, $E$ has a positive probability $c(\D)>0$. 
On the event $E$ we 
define a coupling of the two evolutions by using  the same marks $(t_m,c_m)$, 
i.e.\ at time $t_m$ we update block $c_m$ in both configurations. We refer to Figure \ref{fig:coupling} for a representation of the case $\D=2$. 
At time $t_1$ we have to update the particle positions in the block $B_\D$. 
The corresponding equilibrium is of the form $\nu_K^{2\ell}$, with $2\ell=L-\xi_{K(\D-1)}$. Then, the new particle positions will satisfy 
\begin{equation}
\label{elluntieras}
\xi_{K(\D-1)+1}\geq L-L^{1/3}\,,\quad\text{ with probability at least }\; \frac1{2(K+1)}\,.
\end{equation} 
Indeed, either 
$\xi_{K(\D-1)}\geq L-L^{1/3}$  and the requirement is automatically satisfied because of  $\xi_{K(\D-1)+1}\geq\xi_{K(\D-1)}$, or $\xi_{K(\D-1)}< L- L^{1/3}$, in which case $2\ell\geq L^{1/3}$ and the claim follows from \eqref{elluntier}. 
Therefore, using an independent coupling, one has a probability at least $(4(K+1)^2)^{-1}$ for the event 
$$\min\{\xi_{K(\D-1)+1}(t_1),\xi'_{K(\D-1)+1}(t_1)\} \geq L-L^{1/3}\,.$$
Then, one updates the block $B_{\D-1}$ at time $t_2$, %then the block $B_{\D-2}$ at time $t_3$ 
and so on until one updates the block $B_2$ at time $t_{\D-1}$. 
Iterating the  argument given above, one has a probability of at least $(4(K+1)^2)^{-\D+1}$ for the event 
\begin{equation}\label{tieras}
 \min\{\xi_{K+1}(t_{\D-1}),\xi'_{K+1}(t_{\D-1})\} \geq L-(\D-1)L^{1/3}\,.
\end{equation}
Next, at time $t_\D$, one updates the block $B_1$. 
Let us show that on the event \eqref{tieras} there is a coupling $\bbP$ of the two equilibria on block $B_1$ 
(conditioned on the value of $\xi_{K+1}(t_{\D-1})$ and $\xi_{K+1}'(t_{\D-1})$ respectively) 
such that the event $$M_1:=\{\xi_i(t_{\D})=\xi'_i(t_{\D})\leq -L+L^{1/3}\,,\;\, i=1,\dots,K\}\,,$$
has  probability 
\begin{equation}
\label{evem}
\bbP(M_1)\geq (5(K+1)^2)^{-1}\,.
\end{equation}
%for all $L$ large enough. 
Clearly, it suffices to show that $\xi_K(t_{\D})=\xi_K'(t_{\D})\leq -L+L^{1/3}$ with at least that probability
(all other particles $\xi_1(t_{\D}),\dots,\xi_{K-1}(t_{\D})$ are then automatically matched using the diagonal coupling since $B_1$ is the first block and $\xi_0=\xi_0'=-L$ is fixed). 
Setting $2\ell = \xi_{K+1}(t_{\D-1})+L$ and $2\ell' = \xi'_{K+1}(t_{\D-1})+L$, we need to couple 
the measures $\nu_K^{2\ell}(\xi_K=i)$, $\nu_K^{2\ell'}(\xi'_K=i)$. Consider first the problem of
coupling $\nu_1(i):=\nu_K^{2\ell}(\xi_K=L+i\tc \xi_K\leq -L+L^{1/3})$ and 
$\nu_2(i):=\nu_K^{2\ell'}(\xi'_K=L+i\tc \xi'_K\leq -L+L^{1/3})$. 
We have
\begin{equation}
\label{evemo}
\|\nu_1 - \nu_2\| = O(L^{-2/3})\,.
\end{equation}
Indeed, first note that, by \eqref{tieras}, one has 
\begin{equation}
\label{evemos}
2L\geq 2\ell\,,\, 2\ell'\geq 2L-(\D-1)L^{1/3}\,.
\end{equation} 
With the notation $\rho_+^{\otimes (K)}(i):=\rho_+^{\otimes (K)}(\sum_{j=1}^K\z_j=i)$ one has
\begin{align*}
\|\nu_1 - \nu_2\|&=\frac12\sum_{i=1}^{L^{1/3}}
|\nu_1(i)-\nu_2(i)|\\ &=
\frac12\sum_{i=1}^{L^{1/3}}
\Big|
\frac{\rho_+^{\otimes (K)}(i)\rho_+(2\ell-i)}{\sum_{m=1}^{L^{1/3}}\rho_+^{\otimes {K}}(m)\rho_+(2\ell-m)}-
\frac{\rho_+^{\otimes (K)}(i)\rho_+(2\ell'-i)}{\sum_{m=1}^{L^{1/3}}\rho_+^{\otimes {K}}(m)\rho_+(2\ell'-m)}
\Big|\,.
\end{align*}
Using \eqref{c_+} and \eqref{evemos} one has $\rho_+(2\ell'-i)=\rho_+(2\ell-j)(1+O(L^{-2/3}))$ uniformly in $0\leq i,j\leq L^{1/3}$.
Therefore \eqref{evemo} follows from 
$$
\|\nu_1 - \nu_2\|=O(L^{-2/3})\times
\sum_{i=1}^{L^{1/3}}\frac{\rho_+^{\otimes (K)}(i)\rho_+(2\ell-i)}{\sum_{m=1}^{L^{1/3}}\rho_+^{\otimes {K}}(m)\rho_+(2\ell-m)}
= O(L^{-2/3})\,.
$$
%Once \eqref{evemo} is established, w
We turn to the proof of \eqref{evem}.
We define the coupling of $\nu_K^{2\ell}(\xi_K=i)$, $\nu_K^{2\ell'}(\xi'_K=i)$ as follows. Flip two independent coins with head probability $p=\nu_K^{2\ell}(\xi_K\leq -L+L^{1/3})$ and $p'=\nu_K^{2\ell'}(\xi'_K\leq -L+L^{1/3})$. If both coins end up being head, then sample the pair 
$\xi_K,\xi_K'$ according to the maximal coupling of $\nu_1,\nu_2$ defined by $\|\nu_1 - \nu_2\|$. If the $p$ coin is head and the $p'$ coin is tail then sample independently 
$\xi_K$ according to $\nu_K^{2\ell}(\cdot \tc \xi_K\leq -L+L^{1/3})$ and $\xi_K'$ according to $\nu_K^{2\ell'}(\cdot \tc \xi'_K> -L+L^{1/3})$.
Similarly, if the $p$ coin is tail and the $p'$ coin is head then sample independently $\xi'_K$ according to $\nu_K^{2\ell'}(\cdot \tc \xi'_K\leq -L+L^{1/3})$ and $\xi_K$ according to $\nu_K^{2\ell}(\cdot \tc \xi_K> -L+L^{1/3})$. Finally, if both coins are tail then sample independently $\xi_K,\xi'_K$ according to $\nu_K^{2\ell}(\cdot \tc \xi_K> -L+L^{1/3})$, $\nu_K^{2\ell'}(\cdot \tc \xi'_K> -L+L^{1/3})$. Using this coupling, we obtain that 
$$
\bbP(M_1)\geq p\,p'\,(1-\|\nu_1 - \nu_2\| ) \geq (4(K+1)^2)^{-1}(1+o(1))\,,
$$
where we have used the fact that $\min\{p,p'\}\geq 1/2(K+1)$ by \eqref{elluntier}. This ends the proof of \eqref{evem}.

Repeating the same argument leading to \eqref{evem} one shows that on the event $M_1$, the event 
$$M_2:=\{\xi_{K+i}(t_{\D+1})=\xi'_{K+i}(t_{\D+1})\leq -L+2L^{1/3}\,,\;\, i=1,\dots,K\}\,,$$
has  probability at least $(5(K+1)^2)^{-1}$. Thus, iterating, one 
concludes that 
$$
\bbP(\xi(1)=\xi'(1))\geq \bbP(E)\, (4(K+1)^2)^{-\D+1}\,(5(K+1)^2)^{-\D} \geq c(\D)(5(K+1))^{-4\D}\,.
$$
This ends the proof of \eqref{claimc}. 
 \end{proof} 

\subsection{Proof of Theorem \ref{th2}}\label{proofth2}
Concerning the bound $\trel^\cS\leq L^{5/2 +\e}$, we note that it follows from 
the same arguments used for the proof of Proposition \ref{gapes}. In fact, the situation is simpler here due to the absence of constraints on the number of crossings. We omit the details.

Concerning the lower bound, we can actually prove that for some constant $c>0$, one has
\begin{equation}
\label{low52}
\trel^\cS\geq c\,L^{5/2}\,(\log L)^{-1}\,.
\end{equation} 
We use an argument similar to that in \eqref{5202}. 
Recall that 
\begin{equation}
\label{trelvars}
\trel^\cS\geq \frac{\var_\nu(f)}{\cD(f,f)}
\end{equation}
for any $f:\cS\mapsto\bbR$, where $\nu$ is defined by \eqref{nus} and 
$$
\cD(f,f)=\sum_{x\in O_L}\nu\left[\var_{\nu_x}(f)\right]\,,
$$
$\nu_x$ denoting %the variance w.r.t.\ 
the conditional probability $\nu(\cdot\tc\si_y\,,\;y\neq x)$. 
Let $\z$ denote the number of $+1$ in $\si$, i.e.\
$$
\z(\si)=\sum_{x\in O_L} \ind(\si_x=+1)\,,\qquad\si\in\cS\,.
$$
Note that $\z\in\{0,\dots,L\}$. 
Define the function $f(\si) = g(\z(\si)/L)$, where $g:[0,1]\mapsto[-1,1]$ is given by 
$g(s)=-1$ for $s<1/4$, $g(s)=1$ for $s>3/4$, and by the  linear interpolation for $s\in[1/4,3/4]$. 
Since $\z=0$ and $\z=L$ have both positive probability uniformly in $L$, one has 
$\var(f)\geq c$ for some constant $c>0$. Let us estimate the Dirichlet form of $f$. 
We have
$$
\cD(f,f)\leq c\,L^{-2}\sum_{\si\in\cS}\nu(\si)
\sum_{x\in O_L}p_x(\si) \ind(L/4\leq \z\leq 3L/4)\,,
$$
for some constant $c$, where we write $p_x(\si)$ for the probability of a flip at $x$ in $\si$. Recall that the sum over $x$ of the probabilities $p_x(\si)$ between consecutive crossings give a contribution of order $1$; this follows easily from Lemma \ref{equo}. That yields $\sum_{x\in O_L}p_x(\si) \leq c\,\chi$ for a suitable constant $c$, where 
$\chi$  is the number of crossings \eqref{chi}. 
Therefore, adjusting the constant $c$:
$$
\cD(f,f)\leq c\,L^{-2}\,\nu\left(\chi; \;L/4\leq \z\leq 3L/4\right)\,.
$$
However, $$
\nu\left(\chi; \;L/4\leq \z\leq 3L/4\right) \leq \nu\left(\chi\,;\;\chi > c\,\log L\right) + c\,(\log L)\,\nu\left(L/4\leq \z\leq 3L/4\right) 
\,.
$$
 From Lemma \ref{equo} we deduce that $\nu\left(\chi\,;\;\chi > c\,\log L\right)=O(L^{-p})$ for some large $p$, provided $c$ is large. On the other hand $L/4\leq \z\leq 3L/4$ implies that there exists a crossing in some position $x\in\{-L+L/8,\dots,L-L/8\}$.  From the estimates of Lemma \ref{equo} this last event has probability $O(L^{-1/2})$, so that 
 \begin{equation}
\label{vars}
\nu\left(L/4\leq \z\leq 3L/4\right)  = O(L^{-1/2})\,.
\end{equation}
 In conclusion, adjusting the constant $c$, one has $\cD(f,f)\leq c\,L^{-5/2}\,\log L$. This ends the proof of 
 \eqref{low52}. \qed

\section*{Acknowledgements}
This work has been carried out while H.~Lacoin, F.~Simenhaus and F.L.~Toninelli
were visiting the Department of
Mathematics of the University of Roma Tre under the support of the ERC Advanced Research
Grant ``PTRELSS''.
They all gratefully acknowledge the kind hospitality and
support.

\end{document}